\documentclass[12pt,oneside]{amsart}

\usepackage{geometry}                
\usepackage{graphicx}
\usepackage{amssymb}
\usepackage{epstopdf}
\usepackage{color}
\usepackage{bbm, dsfont}
\usepackage[hidelinks]{hyperref}

\hypersetup{
	colorlinks=false,
	pdfborder={0 0 0},
	pdfborderstyle={/S/U/W 0},
}

\numberwithin{equation}{section}

\usepackage{verbatim}

\usepackage[dvipsnames]{xcolor}

\usepackage{ulem}
\newtheorem{theorem}{Theorem}[section]
\newtheorem{lemma}[theorem]{Lemma}
\newtheorem{corollary}[theorem]{Corollary}
\newtheorem{proposition}[theorem]{Proposition}

\theoremstyle{definition}
\newtheorem{definition}[theorem]{Definition}
\theoremstyle{remark}
\newtheorem{remark}[theorem]{Remark}
\newtheorem*{remark*}{Note}

\theoremstyle{example}
\newtheorem{example}[theorem]{Example}

\numberwithin{equation}{section}

\usepackage{amsmath}
\usepackage{mathrsfs}

\usepackage{verbatim}
\usepackage{amsmath,amssymb,amsthm}             
\usepackage{amssymb}             
\usepackage{amsfonts}            
\usepackage{mathrsfs}            
\usepackage{amsthm}
\usepackage{algorithm}  
\usepackage{algorithmicx}  
\usepackage{algpseudocode}  
\usepackage{mathtools}
\usepackage{commath}
\usepackage{physics}
\usepackage{bm}
\usepackage{graphicx}

\usepackage{float}
\usepackage{listings}
\usepackage{subfigure}
\usepackage{multirow}
\usepackage{diagbox}
\usepackage{color}
\usepackage{bbm}
\usepackage[export]{adjustbox}
\usepackage{enumerate}
\usepackage{bbm}

\usepackage{amsmath}
\usepackage{mathrsfs}

\newcommand{\RNum}[1]{\uppercase\expandafter{\romannumeral #1\relax}}

\newcommand{\specificthanks}[1]{\@fnsymbol{#1}}

\DeclareFontFamily{OML}{rsfs}{\skewchar\font'177}
\DeclareFontShape{OML}{rsfs}{m}{n}{ <5> <6> rsfs5 <7> <8> <9>
	rsfs7 <10> <10.95> <12> <14.4> <17.28> <20.74> <24.88> rsfs10 }{}
\DeclareMathAlphabet{\mathfs}{OML}{rsfs}{m}{n}

\newcounter{cnstcnt}
\newcommand{\cl}{%
	\refstepcounter{cnstcnt}%
	\ensuremath{c_{\thecnstcnt}}}
\newcommand{\cref}[1]{\ensuremath{c_{\ref{#1}}}}

\newcounter{newcnstcnt}
\newcommand{\Cl}{%
	\refstepcounter{newcnstcnt}%
	\ensuremath{C_{\thenewcnstcnt}}}
\newcommand{\Cref}[1]{\ensuremath{C_{\ref{#1}}}}

\newcommand{\overbar}[1]{\mkern 1.5mu\overline{\mkern-1.5mu#1\mkern-1.5mu}\mkern 1.5mu}

\DeclareFontFamily{U}{mathx}{}
\DeclareFontShape{U}{mathx}{m}{n}{<-> mathx10}{}
\DeclareSymbolFont{mathx}{U}{mathx}{m}{n}
\DeclareMathAccent{\widehat}{0}{mathx}{"70}
\DeclareMathAccent{\widecheck}{0}{mathx}{"71}

\begin{document}

	\title{One-arm probabilities for metric graph Gaussian free fields below and at the critical dimension}



		\author{Zhenhao Cai$^1$}
		\address[Zhenhao Cai]{School of Mathematical Sciences, Peking University}
		\email{caizhenhao@pku.edu.cn}
		\thanks{$^1$School of Mathematical Sciences, Peking University}

		\author{Jian Ding$^1$}
		\address[Jian Ding]{School of Mathematical Sciences, Peking University}
		\email{dingjian@math.pku.edu.cn}
	\maketitle
	%
	%
	
	 	\begin{abstract}
	 	For the critical level-set of the Gaussian free field on the metric graph of $\mathbb Z^d$, we consider the one-arm probability $\theta_d(N)$, i.e., the probability that the boundary of a box of side length $2N$ is connected to the center. We prove that $\theta_d(N)$ is $O(N^{-\frac{d}{2}+1})$ for $3\le d\le 5$, and is at most $N^{-2+o(1)}$ for $d=6$. Our upper bounds match the lower bounds in a previous work by Ding and Wirth (2020) up to a constant factor for $3\le d\le 5$, and match the exponent therein for $d=6$. Combined with our previous result (2023) that $\theta_d(N) \asymp N^{-2}$ for $d>6$, this seems to present the first percolation model whose one-arm probabilities are essentially completely understood in all dimensions. In particular, these results fully confirm Werner's conjectures (2021) on the one-arm exponents: 
	 	\begin{equation*}
	 		\text{(1) for}\ 3\le d<d_c=6,\ \theta_d(N)=N^{-\frac{d}{2}+o(1)};\ \text{(2) for}\ d>d_c,\ \theta_d(N)=N^{-2+o(1)}.
	 	\end{equation*}
	 	Prior to our work, Drewitz, Prévost and Rodriguez (2023) obtained upper bounds for $d\in \{3, 4\}$, which are very sharp although lose some diverging factors. In the same work, they conjectured that $\theta_{d_c}(N) = N^{-2+o(1)}$, which is now confirmed. Moreover, in a recent concurrent work, Drewitz, Prévost and Rodriguez (2024) independently obtained the up-to-constant upper bound for $d=3$.
	 	\end{abstract}

	\section{Introduction}

	In this paper, we study the Gaussian free field (GFF) on the metric graph $\widetilde{\mathbb{Z}}^d$, where we assume $d\ge 3$ unless stated otherwise. Specifically, for each adjacent pair $x\sim y$ on the integer lattice $\mathbb{Z}^d$, consider a compact interval $I_{\{x,y\}}$ of length $d$ with two endpoints
	identical to $x$ and $y$ respectively. Then the metric graph $\widetilde{\mathbb{Z}}^d$ is defined as the union of all these intervals. The GFF on $\widetilde{\mathbb{Z}}^d$, denoted by $\{\widetilde{\phi}_v\}_{v\in \widetilde{\mathbb{Z}}^d}$, can be constructed by the following two steps: 
	\begin{enumerate}
		\item  Sample a discrete Gaussian free field $\{\phi_x\}_{x\in \mathbb{Z}^d}$, which is a mean-zero Gaussian field on the lattice $\mathbb{Z}^d$, whose covariance is given by
		\begin{equation*}
			\mathbb{E}\left(\phi_x\phi_y \right)= G(x,y), \ \forall x,y\in \mathbb{Z}^d.  
		\end{equation*}
		Here the Green's function $G(x,y)$ is the average number of visits at $y$ by a simple random walk on $\mathbb{Z}^d$ starting from $x$.

		\item  For any $x\sim y\in \mathbb{Z}^d$, the values of $\widetilde{\phi}_v$ for $v\in I_{\{x,y\}}$ are given by an independent bridge on $I_{\{x,y\}}$ of a Brownian motion with variance $2$ at time $1$, conditioned on the boundary values $\phi_x$ at $x$ and $\phi_y$ at $y$.

	\end{enumerate}

	Percolation for the level-set $\widetilde{E}^{\ge h}:=\big\{v\in \widetilde{\mathbb{Z}}^d: \widetilde{\phi}_v\ge h\big\}$ ($h\in \mathbb{R}$) has been widely studied. Notably, it was proved in \cite{lupu2016loop} that the critical level $\widetilde{h}_*$ of $\widetilde{E}^{\ge \cdot}$ exactly equals  $0$ for all $d\ge 3$. Precisely, for any $h<0$, $\widetilde{E}^{\ge h}$ almost surely percolates (i.e. contains an infinite connected component). At the critical level $h=\widetilde{h}_*=0$, \cite[Proposition 5.2]{lupu2016loop} shows that the \textit{two-point function} satisfies 
	\begin{equation}\label{two_point}
		\mathbb{P}\big(x \xleftrightarrow{\widetilde{E}^{\ge 0} } y\big) \asymp |x-y|^{2-d}, \ \ \forall x\neq y\in \mathbb{Z}^d, 
	\end{equation}
	where $\big\{A_1\xleftrightarrow{\widetilde{E}^{\ge 0} } A_2\big\}$ represents the event that $A_1$ and $A_2$ are connected by $\widetilde{E}^{\ge 0}$ (i.e. there exists a path in $\widetilde{E}^{\ge 0}$ connecting $A_1$ and $A_2$), ``$f \asymp g$'' means $cg \le  f \le  Cg$ for some constants $C>c>0$ depending only on $d$, and $|\cdot|$ is the Euclidean norm. As shown in the proof of \cite[Theorem 2]{lupu2016loop}, (\ref{two_point}) and the uniqueness of the infinite cluster imply that $\widetilde{E}^{\ge 0}$ does not percolate. Consequently, the \textit{one-arm probability} 
	\begin{equation}
		\theta_d(N):= \mathbb{P}\Big[ \bm{0}\xleftrightarrow{\widetilde{E}^{\ge 0}} \partial B(N) \Big]  \to 0\  \mbox{as}\ N\to \infty,
	\end{equation}
where $\bm{0}$ is the origin of $\mathbb{Z}^d$, $\partial A:= \{x\in A: \exists y\in \mathbb{Z}^d\setminus A\ \text{such that}\ y\sim x\}$ and $ B(N):=[-N,N]^d\cap \mathbb{Z}^d$. There is a series of articles that have estimated the decay rate of $\theta_d(N)$. In \cite{ding2020percolation}, polynomial bounds for all dimensions were established using a martingale argument: 
	\begin{itemize}
		\item  When $d=3$, $cN^{-\frac{1}{2}}\le \theta_3(N)\le CN^{-\frac{1}{2}}\ln^{\frac{1}{2}}(N)$;

		\item When $d\ge 4$, $cN^{-\frac{d}{2}+1}\le \theta_d(N) \le CN^{-\frac{1}{2}}$.

	\end{itemize}
	After that, these estimates were generalized to a broad class of transient
	graphs by \cite{drewitz2023critical}. Later, inspired by \cite{kozma2011arm,werner2021clusters}, the authors proved that in high dimensions (i.e. $d\ge 7$), the metric graph GFF falls into the mean-field regime. I.e., for $d\ge 7$, 
	\begin{equation}\label{ineq_high_d}
		cN^{-2} \le \theta_d(N)\le CN^{-2}.
	\end{equation}
	In three and four dimensions, more accurate estimates were obtained in \cite{drewitz2023arm}: 
	\begin{itemize}
		\item   When $d=3$, $\theta_3(N)\le CN^{-\frac{1}{2}}\ln\ln(N)$;

		\item  When $d=4$,  $\theta_4(N)\le CN^{-1} \ln^2(N)[\ln\ln(N)]^2$.

	\end{itemize}
	However, in addition to the findings already established, some fundamental questions regarding $\theta_d(N)$ in low dimensions (i.e. $3\le d\le 6$) remain unresolved. For $d\in \{3,4\}$, while it was known that $\theta_d(N)= N^{-\frac{d}{2}+1+o(1)}$, there remains a logarithmic disparity between the upper and lower bounds. In fact, as discussed in \cite[Section 1.4]{ding2020percolation}, these logarithmic factors hide important information about the geometry of critical clusters. To this end, here comes the first question. 
	    
	    \vspace{0.36em}  
	    
	    \begin{center}
	    \textbf{Question 1:} \textit{What is the exact order of $\theta_d(N)$ for $d\in \{3,4\}$?}
	    \end{center}
	    
	    \vspace{0.36em}

	    \noindent For $d\in \{5,6\}$, even the exponent of $\theta_d(N)$ is unknown. In \cite{werner2021clusters}, it was predicted through the dimensions of critical clusters that $\theta_d(N)= N^{-\frac{d}{2}+1+o(1)}$ for $3\le d\le 5$ (more details can be found in Remark \ref{remark_werner}). This prediction was later reiterated and extended to $d=6$ by \cite{drewitz2023arm}. Overall, the second question arises as follows.
	    
	     \vspace{0.36em}  
	    
	    \begin{center}
	    	\textbf{Question 2:} \textit{What is the exponent of $\theta_d(N)$ for $d\in \{5,6\}$? }
	    \end{center}
	    
	    \vspace{0.36em}

	    Our main result in this paper solves these two questions. Precisely, we prove

\begin{theorem}\label{thm1}
	For $d\in \{3,4,5\}$, there exists $\Cl\label{const_thm_1}>0$ such that for all $N\ge 1$, 
\begin{equation}\label{thm1_1.4}
   \theta_d(N) \le \Cref{const_thm_1} N^{-\frac{d}{2}+1}. 
\end{equation}
For $d=6$, there exists $\Cl\label{const_thm_2}>0$ such that for all sufficiently large $N\ge 1$,
\begin{equation}\label{one_arm_6d}
	\theta_6(N) \le   \Cref{const_thm_2}N^{-2}\mathrm{exp}\big(\ln^{\frac{1}{2}}(N)\ln\ln(N)\big)=N^{-2+o(1)}. 
\end{equation}
\end{theorem}
\begin{remark*}
	A week prior to our manuscript, Drewitz, Prévost and Rodriguez posted a very nice paper \cite{drewitz2024critical} which proved (\ref{thm1_1.4}) for $d=3$. We would like to express our thanks to Drewitz, Prévost and Rodriguez for pointing out the independence between our manuscript and \cite{drewitz2024critical}.
\end{remark*}

As a natural extension of the one-arm probability $\theta_d(N)$, the \textit{crossing probability} for an annulus is also of interest. Namely, for any $N\ge n\ge 1$, we consider
\begin{equation}\label{def_crossing}
	\rho_d(n,N):=\mathbb{P}\Big[ B(n)\xleftrightarrow{\widetilde{E}^{\ge 0}} \partial B(N) \Big].
\end{equation}
Our second result reveals the exact order of $\rho_d(n,N)$ for all dimensions except the critical dimension $6$. Additionally, we derive a bound which yields the exponent of $\rho_6(n,N)$ but with a subpolynomial disparity as in (\ref{one_arm_6d}). 
\begin{theorem}\label{thm2}
	For $d\in \{3,4,5\}$, there exist $\Cl\label{const_thm2_1},\cl\label{const_thm2_2}>0$ such that for all $N\ge n\ge 1$,
	\begin{equation}
\cref{const_thm2_2}\Big(\frac{n}{N}\Big)^{\frac{d}{2}-1}	\le 	\rho_d(n,N)\le \Cref{const_thm2_1}\Big(\frac{n}{N}\Big)^{\frac{d}{2}-1}. 
	\end{equation}
	For $d=6$, there exist $\Cl\label{const_thm2_3},\cl\label{const_thm2_4}>0$ such that for all sufficiently large $N\ge n\ge 1$,
	\begin{equation}
  \cref{const_thm2_4} \Big(\frac{n}{N}\Big)^{2}	\le 	\rho_6(n,N)\le \Cref{const_thm2_3}\Big(\frac{n}{N}\Big)^{2}\mathrm{exp}\big(2\ln^{\frac{1}{2}}(N)\ln\ln(N)\big). 
	\end{equation} 
	For any $d\ge 7$, there exist $\Cl\label{const_thm2_5}(d),\cl\label{const_thm2_6}(d)>0$ such that for all $N\ge n\ge 1$,
	\begin{equation}\label{thm2_1.9}
\cref{const_thm2_6}\Big(1+n^{4-d}N^2 \Big)^{-1} 	  \le 	\rho_d(n,N)\le \Cref{const_thm2_5} n^{d-4}N^{-2}.
	\end{equation} 
\end{theorem}

\begin{remark}[critical loop soup]
	\cite[Proposition 2.1]{lupu2016loop} introduced a coupling between the GFF and the loop soup at the critical intensity $\frac{1}{2}$ (the criticality of $\frac{1}{2}$ was proved by \cite{chang2023percolation}) on the metric graph. Under this coupling, every GFF sign cluster (i.e. maximal connected subgraph on which $\widetilde{\phi}_{\cdot}$ has the same sign) is exactly a loop soup cluster. Thus, by the symmetry of the GFF, the analogues of Theorems \ref{thm1} and \ref{thm2} are valid for the critical loop soup on $\widetilde{\mathbb{Z}}^d$. 
\end{remark}

\begin{remark}[Werner's conjectures]\label{remark_werner}
	(1) It was first predicted in \cite[Conjectures A and C]{werner2021clusters} that for $d\in \{3,4,5\}$, the scaling limit of critical clusters (for either the metric graph GFF or the critical loop soup) exists and has the fractal dimension $\frac{d}{2}+1$. In fact, this conjecture implies that in a box of side length $O(N)$, the number of vertices included in some macroscopic cluster should be approximately $N^{\frac{d}{2}+1+o(1)}$, which intuitively suggests that for $d\in \{3,4,5\}$ and any constant $C>0$,  
	\begin{equation}\label{1.3}
	 \theta_d(N)= |B(CN)|^{-1}\mathbb{E}\Big[\sum\nolimits_{x\in B(CN)}\mathbbm{1}_{x\xleftrightarrow{\widetilde{E}^{\ge 0} } \partial B_x(N)} \Big] = N^{-\frac{d}{2}+1+o(1)}, 
	\end{equation}
	where $B_x(N):= x+B(N)$. By Theorem \ref{thm1}, (\ref{1.3}) is now fully confirmed. We hope Theorem \ref{thm1} can in turn facilitate the proof of \cite[Conjectures A and C]{werner2021clusters}.

	  \vspace*{0.15cm}

    (2) For the high-dimensional cases (i.e. $d>6$), it was conjectured in \cite[Section 5]{werner2021clusters} that unlike in low dimensions, sign clusters for the metric graph GFF (or loop clusters for the critical loop soup) become asymptotically independent and thus behave similarly to the Bernoulli percolation. Notably, in high dimensions the mean-field behaviors have been established for the critical Bernoulli percolation: 
    \begin{itemize}
    	\item For $\mathbb{Z}^d$ ($d\ge 2$), let $\mathbb{P}_{p}$ represent the law of the Bernoulli percolation with parameter $p\in [0,1]$, and let $p_c(d)$ be the critical percolation parameter. The analogue of the two-point function estimate (\ref{two_point}), i.e., 
    	\begin{equation}\label{two_point_bernoulli}
    		\mathbb{P}_{p_c}(x\xleftrightarrow{} y) \asymp |x-y|^{2-d}, \ \ \forall x\neq y\in \mathbb{Z}^d,
    	\end{equation}
    	was established for $d\ge 19$ in \cite{hara1990mean}, and later extended to $d\ge 11$ in \cite{fitzner2017mean}.

    	\vspace*{0.1cm}

    	\item In \cite{aizenman1987uniqueness,barsky1991percolation}, it was proved that assuming the triangle condition, i.e., 
    	\begin{equation*}\label{tri_condition}
    		\sum\nolimits_{x,y\in \mathbb{Z}^d}\mathbb{P}_{p_c}(\bm{0}\leftrightarrow x)\mathbb{P}_{p_c}(x\leftrightarrow y)\mathbb{P}_{p_c}(y \leftrightarrow \bm{0})<\infty,
    	\end{equation*}
    	the decay rate of the critical cluster volume satifies 
    	\begin{equation*}\label{volumn_Bernoulli}
    		\mathbb{P}_{p_c}\big( \big| \{x\in \mathbb{Z}^d: x\xleftrightarrow{} \bm{0}\}\big|\ge M\big) \asymp  M^{-\frac{1}{2}}. 
    	\end{equation*}
    	
    
    	\vspace*{0.1cm}

    	\item It was proved in \cite{kozma2011arm} that for $d>6$, assuming the two-point function (\ref{two_point_bernoulli}), the one-arm probability satisfies $\mathbb{P}_{p_c}\big[ \bm{0}\xleftrightarrow{} \partial B(N) \big] \asymp N^{-2}$.

   	
    \end{itemize}
    To sum up, the aforementioned conjectures in \cite{werner2021clusters} proposed that $6$ is the critical dimension $d_c$ such that below $d_c$ the exponent of $\theta_d(N)$ relies on the dimension, and above $d_c$ it remains constant. Considering that our previous result in \cite{cai2023one} (see (\ref{ineq_high_d})) demonstrates﻿ the mean-field behavior for $\widetilde{\phi}$ in high dimensions, and that (\ref{1.3}) is established by Theorem \ref{thm1}, now we can confirm that $d_c=6$.

 \vspace*{0.15cm}

    (3) Although the case in the critical dimension (i.e. when $d=6$) was deliberately not mentioned in \cite{werner2021clusters} (see the end of \cite[Section 3]{werner2021clusters}), our Theorem \ref{thm1} shows that $\theta_6(N)= N^{-2+o(1)}$, where the exponent matches the rules both  below and above $d_c$ (since $\frac{d}{2}-1=2$ when $d=6$), as expected by \cite{drewitz2023arm}.

\end{remark}


\begin{remark}(conjecture on the order of $\theta_6(N)$) 
	Inspired by a discussion with Tom Hutchcroft, the following seems to be a reasonable conjecture for the one-arm probability in the critical dimension $6$: there exist constants $C,c,\delta>0$ such that for all sufficiently large $N\ge 1$,
	\begin{equation}
	cN^{-2}\ln^{\delta}(N) \le \theta_6(N) \le CN^{-2}\ln^{\delta}(N).
	\end{equation}
	This conjecture is mainly based on similar behaviors of several other models. E.g., the critical two-point function of the weakly self-avoiding walk \cite{bauerschmidt2015critical}, and the decay rate of the critical cluster volume in the hierarchical percolation \cite{hutchcroft2022critical}.

\end{remark}

\subsection{Related literature}


In this subsection, we review some previous results related to level-sets of Gaussian free fields, both on lattices and on metric graphs.

\vspace*{0.15cm}

\textbf{Discrete GFF.} The level-set $E^{\ge h}:= \{x\in \mathbb{Z}^d:\phi_x\ge h\}$ ($h\in \mathbb{R}$) of the discrete GFF has been widely studied. It was proved in \cite{bricmont1987percolation, rodriguez2013phase} that for $d\ge 3$, the percolation of $E^{\ge h}$ exhibits a non-trivial phase transition, where the critical level $h_*(d)$ is non-negative. Using the metric graph GFF as an auxiliary model, \cite{drewitz2018sign} further proved that $h_*(d)$ is strictly positive for all $d\ge 3$. It was established in \cite{drewitz2015high} that the asymptotic value of $h_*(d)$ (as $d\to \infty$) is $\sqrt{2G(\bm{0},\bm{0})\ln(d)}$. Notably, the phase transition of $E^{\ge h}$ was proved to be sharp in \cite{duminil2023equality}. I.e., for any $h>h_*(d)$, the radius of the cluster of $E^{\ge h}$ decays exponentially; for any $h<h_*(d)$, with high probability $E^{\ge h}\cap B(N)$ includes a macroscopic cluster (i.e. a cluster with radius of order $N$) and moreover, every two macroscopic clusters are connected by $E^{\ge h}\cap B(2N)$. In \cite{muirhead2024percolation}, this sharpness was extended to a wide class of Gaussian percolation models. \cite{drewitz2014chemical} established the shape theorem for supercritical level-sets of discrete GFFs. Furthermore, different versions of the so-called decoupling inequality, which is a powerful tool to handle the correlation of the GFF, were presented in \cite{popov2015decoupling, popov2015soft}. For subcritical level-sets of discrete GFFs, numerous estimates for the decay rate of the cluster radius were established in \cite{goswami2022radius}. These estimates were extended to a broad class of Gaussian percolation models in \cite{Franco2024percolation}. For discrete GFFs on regular trees, many bounds on critical and near-critical quantities were proved in \cite{vcerny2023critical}.

\vspace*{0.15cm}

\textbf{Metric graph GFF.} Several types of dimensions for the incipient infinite cluster of the metric graph GFF on $\widetilde{\mathbb{Z}}^d$ in high dimensions were calculated in \cite{ganguly2024ant}. Moreover, a variant of the metric graph GFF named the ``gauge-twisted GFF'' was introduced and studied in \cite{lupu2022equivalence}. Notably, as an interesting application of the metric graph GFF, \cite{duminil2020existence} employed the metric graph GFF as an auxiliary model to prove that the Bernoulli percolation on any graph with isoperimetric dimension greater than $4$ exhibits a non-trivial percolation phase transition.

\subsection{Statements about constants}

We use the notations $C$ and $c$ for the constants whose values change according to the context. The numbered notations $C_1,C_2,c_1,c_2,...$ are used for the constants fixed throughout the paper. We use the upper-case letter $C$ (possibly with some superscript or subscript) to represent large constants, and the lower-case letter $c$ to denote small ones. When a constant depends on some parameter or variable, we will point it out in parentheses. A constant without additional specification can only depend on the dimension $d$.

\subsection{Outline of the proof of Theorem \ref{thm1}}


Thanks to \cite{lupu2016loop}, the two-point connecting probability for the critical GFF level-set on the metric graph has been computed precisely. Thus, the main difficulty in computing the one-arm probability is to enhance a point-to-point estimate to a point-to-set estimate. For $d>6$, this was carried out in \cite{cai2023one} following the framework of \cite{kozma2011arm}. The main challenge in \cite{cai2023one} stems from the strong correlation of the GFF, due to which many arguments in \cite{kozma2011arm} cannot be simply adapted. It turns out that for $d>6$, despite being strong, the correlation can be controlled and in the conceptual level this leads to the mean-field behavior; in the technical level, the manipulation of the correlation used many properties of the GFF including a powerful coupling with the critical loop soup introduced in \cite{lupu2016loop}. For $d\le 6$, however, even in the conceptual level the correlation is sufficiently strong to lead to a different behavior. In a sense, it is the favorable structural properties of the GFF that enable the derivation of one-arm probabilities. For instance, thanks to the Markov property and the harmonicity of the GFF, the authors in \cite{lupu2018random} proposed an exploration martingale which was then further developed in \cite{ding2020percolation} to derive the one-arm exponent for $d=3$. The main obstacle for more precise estimates is that a priori there was no accurate relation between the quadratic variation of such exploration martingale and the Euclidean diameter of the critical cluster. After \cite{ding2020percolation}, much progress was made in \cite{drewitz2023arm,drewitz2023critical,drewitz2024critical} which includes a careful study on the capacity of the critical level-set cluster (and this is closely related to the quadratic variation of the exploration martingale). In particular, \cite{drewitz2023arm} used a clever interplay between the GFF and the loop soup to control the capacity for $d\in \{3, 4\}$, thereby improving \cite{ding2020percolation}.

%
%
%
%
%
%
%
%
%
%
%
%
%
%
%
%
%
%
%
%
%
%
%
%
%
%

In this paper, our main contribution is to develop a new method to control the quadratic variation of the exploration martingale and to relate it to the Euclidean diameter of the critical cluster. While we will also use the coupling between the GFF and the loop soup, our primary strategy lies in exploring the intrinsic property of the GFF. In fact, the starting point of our proof can be viewed as \cite[Equation (18)]{lupu2018random}, which gives an explicit formula for the set-to-set connecting probability \textit{conditioned on positive GFF values for all vertices in both sets}. This motivates us to first explore the negative clusters from the boundary of the box (and thus after the exploration we effectively have positive boundary conditions, though on a random set) and then investigate quantities in the explicit formula \cite[Equation (18)]{lupu2018random} for the one-arm probability. Notably, the following two fighting forces (under the assumption of a large one-arm probability) are the key to our proof strategy. On the one hand, a large one-arm probability indicates that the negative clusters explored from the boundary of the box (note that by symmetry they share the same law as the positive ones) are expected to be sizable, and thus the conditional expectation at the origin (which can be considered as the final value of the exploration martingale) is required to be small. On the other hand, if the explored negative clusters are large, then the exploration process typically takes a long time to stop, i.e., the quadratic variation of the exploration martingale is big, implying that the final value of the martingale has a good chance of being large. Of course, transforming the preceding heuristics into a rigorous mathematical proof is by all means non-trivial. In the rest of this subsection, we elaborate our proof outline in a more precise (and thus inevitably more technical) manner.

The rest of this subsection consists of the following three parts. In the first part, we present the outline for our proof of Theorem \ref{thm1} in the case of $d=3$, which also includes the backbone of our proof for higher dimensions (i.e. $4\le d\le 6$). In the second part, we point out why the approach provided in the first part is insufficient for $4\le d\le 6$ and then depict how we use further techniques to extend it. As demonstrated later, the critical dimension $d_c=6$ naturally emerges along with our computations since numerous crucial steps hold true for $d<6$ but fail exactly at $d=6$ due to the simultaneous vanishing of some leading terms in certain estimates. In the third part, we explain how to slightly modify the proof setting to obtain the upper bound for $d=6$ in (\ref{one_arm_6d}) with the subpolynomial factor.

\noindent(Warning: for the convenience of exposition, the notations and definitions in this subsection are not necessarily consistent with the formal proof in later sections.)



\subsubsection{Proof outline for $d=3$}\label{section_1.2.1}

Our proof strategy is based on proof by contradiction. We take a large constant $\lambda>0$ and assume that $\theta_3(N)\le \lambda N^{-\frac{1}{2}}$ does not hold for all $N\ge 1$. Let $N_*$ be the smallest integer satisfying $\theta_3(N_*)> \lambda N_*^{-\frac{1}{2}}$ (see Definitions \ref{def_es} and \ref{def_af}). Our proof is then based on the analysis of harmonic averages in scales comparable to $N_*$, and it naturally takes advantage of the minimality of $N_*$ as assumed. Precisely, for any $x\in \widetilde{\mathbb{Z}}^d$, the harmonic average at $x$, denoted by $\mathcal{H}_x^*$ (see (\ref{4.9})), is the conditional expectation of $\widetilde{\phi}_{x}$ given the values of $\widetilde{\phi}_\cdot$ on $\mathcal{C}^-_{\partial \mathcal{B}_x(N_*/2)}$, where $\mathcal{B}_x(N_*/2)$ is the Euclidean ball with center $x$ and radius $N_*/2$, and $\mathcal{C}^-_{\partial\mathcal{B}_x(N_*/2)}$ is the union of $\partial \mathcal{B}_x(N_*/2)$ and all negative clusters of $\widetilde{\phi}_{\cdot}$ intersecting $\partial \mathcal{B}_x(N_*/2)$. The contradiction arises from the following two ingredients.

\begin{enumerate}
	\item   With the assumption on the existence of $N_*$, we are able to construct an event $\mathsf{F}$ with a significant probability (measurable with respect to $\widetilde{\phi}$) on which an independent Brownian motion on $\widetilde{\mathbb{Z}}^d$ (refer to Section \ref{subsection_BM} for the definition) starting from $\bm{0}$ will hit $\mathcal{C}^-_{\partial \mathcal{B}(N_*/4)}$ before reaching $\partial \mathcal{B}(0.01N_*)$ with high probability (see Lemma \ref{lemma_F}). As a result, we establish that the sum of harmonic averages on $\partial \mathcal{B}(\frac{3N_*}{16})$ may reach an unexpectedly high level with a significant probability (see Proposition \ref{prop_1}).

	\item   Through an analysis of the conditional distribution of the average of $\widetilde{\phi}_\cdot$ on $\partial \mathcal{B}(\frac{3N_*}{16})$, we establish an upper bound for the probability that the sum of harmonic averages on $\partial \mathcal{B}(\frac{3N_*}{16})$ takes a large value (see Proposition \ref{prop_2}), which leads to a contradiction with the result stated in Item (1).

\end{enumerate}

In what follows, we provide a sketch for the most crucial components of this paper, namely, the construction of the event $\mathsf{F}$ and the analysis of the Brownian motion as mentioned in Item (1). First of all, using the assumption that $\theta_3(N_*)> \lambda N_*^{-\frac{1}{2}}$ and the formula in \cite{lupu2018random} for the connecting probability of $\widetilde{E}^{\ge 0}$ (see Lemma \ref{lemma_connecting_boundary}), we show that there exists an integer $k_*\ge 1$ such that for each $x\in \mathbb{Z}^d$, the harmonic average $\mathcal{H}_x^*$ exceeds $\lambda N_*^{-\frac{d}{2}+1}k_*^{-C}2^{-k_*}$ with probability at least $2^{-k_*}$ (see Lemma \ref{lemma_k_diamond}). Moreover, applying the exploration martingale argument (see Section \ref{section_EM}), we know that $\mathcal{H}_x^*\ge \lambda N_*^{-\frac{d}{2}+1}k_*^{-C}2^{-k_*}$ (if this happens, we call $x$ a good point; see Definition \ref{def_good_point}) indeed indicates that there exists an integer $j\ge 1$ such that an independent Brownian motion starting from $x$ will hit $\mathcal{C}^-_{\partial \mathcal{B}_x(N_*/2)}$ before exiting $B
_x(2^{-j}k_*^{-C}2^{-0.5k_*}N_*)$ with probability at least $p_j:=j^{-2}2^{-j(d-2)}\lambda^2 k_*^{-C}2^{(\frac{6-d}{2})k_*}$ (if this happens, we call $x$ a $j$-nice point; see (\ref{newadd_5.49})). (Note that here the natural requirement that $p_j \le 1$ in fact poses a constraint on $j$, but we will ignore such technical complications in our discussion since such a constraint is in our favor anyway.) Since the probability of having a good point is at least $2^{-k_*}$, typically a box of size $N_*$ should contain a $2^{-k_*}$ fraction of good points (we call it a good box if this happens; see Definition \ref{def_good_box}). Through a second moment method (see Lemma \ref{lemma_excellent_box}), we show that a good box must contain a sub-box from which an independent Brownian motion will visit at least $k_*^{-C}2^{-k_*}N_*^2$ good points before escaping faraway from this sub-box with probability $k_*^{-C}$ (we call this sub-box an excellent box; see Definition \ref{def_excellent_box}). In fact, given that the Brownian motion visits $k_*^{-C}2^{-k_*}N_*^2$ good points, it hits the negative clusters with high probability. To see this, recall that every good point must be $j$-nice for some $j\ge 1$. As a result, there exists some $j\ge 1$ such that among these $k_*^{-C}2^{-k_*}N_*^2$ good points, more than $n_{j}^{(1)}:=j^{-2}k_*^{-C}2^{-k_*}N_*^2$ points are $j$-nice. Every time when the Brownian motion starts from a $j$-nice point $x$, with probability $p_j$ it will hit the negative clusters before exiting $B_x(2^{-j}k_*^{-C}2^{-0.5k_*}N_*)$, which typically takes $n_{j}^{(2)}:=2^{-2j}k_*^{-C}2^{-k_*}N_*^2$ steps. Therefore, the number of such exits is at least $n_{j}^{(3)}:=n_{j}^{(1)}/n_{j}^{(2)}=j^{-2}2^{2j}k_*^{-C}$. Thus, by the strong Markov property, the probability that the Brownian motion does not hit the negative clusters after visiting $k_*^{-C}2^{-k_*}N_*^2$ good points can be bounded from above by (recalling that $p_j=j^{-2}2^{-j(d-2)}\lambda^2 k_*^{-C}2^{(\frac{d-2}{2})k_*}$)
\begin{equation}\label{outline_1.14}
	(1-p_j)^{n_{j}^{(3)}}\le \mathrm{exp}\big(-p_jn_{j}^{(3)}\big)\le \mathrm{exp}\big(-j^{-4}2^{(4-d)j}\lambda^2k_*^{-C}2^{(\frac{6-d}{2})k_*}\big).
\end{equation}
When $d=3$, since $j^{-4}2^{(4-d)j}=j^{-4}2^{j}\ge c$ and $k_*^{-C}2^{(\frac{6-d}{2})k_*}= k_*^{-C}2^{\frac{3}{2}k_*}\ge c 2^{k_*}$ for all $j,k_*\ge 1$, the right-hand side of (\ref{outline_1.14}) is at most $e^{-c\lambda^22^{k_*}}$. In conclusion, by taking a sufficiently large $\lambda$, the Brownian motion starting from an excellent box will hit the negative clusters before escaping faraway with probability at least $k_*^{-C}-e^{-c\lambda^22^{k_*}}\ge k_*^{-C'}$. Thus, by defining $\mathsf{F}$ as the event that $\bm{0}$ is surrounded by $k_*^{10C'}$ layers of excellent boxes (see (\ref{def_F})), we obtain the desired property for the Brownian motion on $\mathsf{F}$. I.e., hitting the negative clusters with high probability.

\subsubsection{Extension to higher dimensions}

For $d\ge 4$, since $j^{-4}2^{(4-d)j}\ge c$ does not hold for all $j\ge 1$, it is no longer straightforward to conclude that the right-hand side of (\ref{outline_1.14}) is small. Therefore, we need some additional restriction on the scale $j$ in (\ref{outline_1.14}). To achieve this, we establish an a priori upper bound for the crossing probability of $\widetilde{E}^{\ge 0}$ (see Proposition \ref{lemma_bound_crossing}), which is then used to exclude cases where $j$ is too large. Precisely, by adding an additional criterion to the definition of a good point $x$, stating that $x$ has to be somewhat faraway from $\mathcal{C}^-_{\partial \mathcal{B}_x(N_*/2)}$ (see Definition \ref{def_good_point}), we can further require that $j$ satisfies (see (\ref{cal1})) 
\begin{equation}\label{outline_1.15}
	2^{j}\le k_*^{C}2^{(\frac{6-d}{2(d-2)})k_*}\lambda^{\frac{4}{d-2}}.
\end{equation}
Combined with (\ref{outline_1.14}), it implies that the escape probability for the Brownian motion starting from an excellent box is upper-bounded by (see (\ref{5.61}))
\begin{equation}\label{outline_1.16}
\mathrm{exp}\big(-ck_*^{-C}2^{(\frac{2(6-d)}{d-2})k_*}\lambda^{\frac{2(6-d)}{d-2}}\log_2^{-C}(\lambda)\big). 
\end{equation}
For $d\in \{4,5\}$, since $2^{(\frac{2(6-d)}{d-2})k_*}\gg k_*^{C}$ and $\lambda^{\frac{2(6-d)}{d-2}}\gg \log_2^{C}(\lambda)$, we derive the same bound as in Section \ref{section_1.2.1} from (\ref{outline_1.16}), thus confirming the desired property for $\mathsf{F}$.

When $d=6$, all the leading factors in (\ref{outline_1.16}), namely $2^{(\frac{2(6-d)}{d-2})k_*}$ and $\lambda^{\frac{2(6-d)}{d-2}}$, vanish simultaneously. Consequently, regardless of any restrictions imposed on $j$ and $k_*$, there is no hope to conclude that the right-hand side of (\ref{outline_1.14}) is small. At this stage, it is readily seen that these computations provide a manifestation (in the technical level) on the emergence of the critical dimension $d_c=6$.

\subsubsection{Upper bound for $d=6$}

To derive the upper bound for $\theta_6(N)$, we assume that $\lambda(N)\ge 1$ is now a non-decreasing function depending on $N$, as opposed to being a constant as in previous cases. Then the estimates (\ref{outline_1.14}) and (\ref{outline_1.15}) are respectively updated to 
\begin{equation}\label{outline_1.17}
	(1-p_j)^{n_{j}^{(3)}}\le  \mathrm{exp}\big(-j^{-4}2^{-2j}\lambda^2(N_*)k_*^{-C}\big),
\end{equation}
\begin{equation}\label{outline_1.18}
	2^{j}\le k_*^{C}\big[\lambda(N_*)\lambda(\widecheck{N}_*)\big]^{\frac{1}{2}},
\end{equation}
where $\widecheck{N}_*\le[\lambda(N_*)]^{-\frac{1}{d}}2^{-\frac{1}{d}k_*}N_*$. By  (\ref{outline_1.17}) and (\ref{outline_1.18}), the escape probability for the Brownian motion starting from an excellent box is at most 
\begin{equation}\label{1.19}
		\mathrm{exp}\Big(-k_*^{C}\lambda(N_*)\Big[\lambda\Big([\lambda(N_*)]^{-\frac{1}{d}}2^{-\frac{1}{d}k_*}N_*\Big)\Big]^{-1}\log_2^{-C}\big(\lambda(N_*)\big)\Big). 
\end{equation}
By selecting the function $\lambda(N)=C\mathrm{exp}\big(\ln^{\frac{1}{2}}(N)\ln\ln(N)\big)$ (the reason behind this choice is detailed in Remark \ref{remark_lambda}), we ensure that (\ref{1.19}) is small and thus establish the upper bound in Theorem \ref{thm1} for $\theta_6(N)$.

\subsection{Organization of the paper}

In Section \ref{section_prepare}, we fix some necessary notations and review some useful results. Section \ref{Section_crossing} establishes an upper bound for the crossing probability (Proposition \ref{lemma_bound_crossing}), which not only plays a crucial role in the proof of Theorem \ref{thm1} but also provides a conditional proof for the upper bounds in Theorem \ref{thm2} assuming Theorem \ref{thm1}. In Section \ref{section_proof_thm1}, we provide the two main ingredients in the proof of Theorem \ref{thm1} (Propositions \ref{prop_1} and \ref{prop_2}) and then demonstrate Theorem \ref{thm1} assuming them. We also verify Proposition \ref{prop_2} in Section \ref{section_proof_thm1}. Subsequently, Proposition \ref{prop_1} is established in Section \ref{section_block}. Finally, Section \ref{section_thm2} presents the proofs of the lower bounds in Theorem \ref{thm2}.


\section{Preliminaries}\label{section_prepare}
    
  In order to facilitate the exposition, in this section we collect some necessary notations and useful results for the Brownian motion and the Gaussian free field on the metric graph $\widetilde{\mathbb{Z}}^d$.

\subsection{Basic notations for graphs}

\begin{itemize}

		\item  For any $x\in \mathbb{Z}^d$ and $N\ge 0$, we denote the box $B_x(N):= x+[-N,N]^d$. We also denote the Euclidean ball $\mathcal{B}_x(N):=\{y\in \mathbb{Z}^d:|y-x|\le N\}$ and the continuous box (i.e. the metric graph of the box)
		\begin{equation}\label{def_continuous_box}
			\widetilde{B}_x(N):=\bigcup_{y_1\sim y_2\in B_x(N): \{y_1,y_2\}\cap B_x(N-1)\neq \emptyset} I_{\{y_1,y_2\}}.
		\end{equation}
		We need the Euclidean ball mainly because its exiting distribution for a simple random walk is comparable to the uniform distribution on the boundary (see e.g. \cite[Lemma 6.3.7]{lawler2010random}). Note that 
	\begin{equation}\label{inclusion_box}
		B_x(d^{-\frac{1}{2}}N) \subset  \mathcal{B}_x(N)\subset  B_x(N).
	\end{equation}
	When $x=\bm{0}$, we abbreviate 
	$$
	B(N):=B_{\bm{0}}(N), \ \ \mathcal{B}(N):=\mathcal{B}_{\bm{0}}(N)\ \ \text{and} \ \ \widetilde{B}(N):=\widetilde{B}_{\bm{0}}(N).
	$$


	\item  Recall that $\partial A:= \{x\in A: \exists y\in \mathbb{Z}^d\setminus A\ \text{such that}\ y\sim x\}$ for all $A\subset \mathbb{Z}^d$. We also denote by $\partial^{\mathrm{e}} A:=\{x\in \mathbb{Z}^d\setminus A:\exists y\in A\ \text{such that}\ y\sim x\}$ the external boundary of $A$.

	\item  For any $v_1,v_2\in \widetilde{\mathbb{Z}}^d$, we denote by $\|v_1-v_2\|$ the graph distance between $v_1$ and $v_2$ on the metric graph $\widetilde{\mathbb{Z}}^d$. Note that $\|x-y\|=d$ for all $x\sim y\in \mathbb{Z}^d$.

	\item For any $U_1,U_2\subset \widetilde{\mathbb{Z}}^d$, let $\mathrm{dist}(U_1,U_2):=\inf_{v_1\in U_1,v_2\in U_2}\|v_1-v_2\|$.

	\item For any $U\subset  \widetilde{\mathbb{Z}}^d$, let $\overbar{U}:= \{v\in  \widetilde{\mathbb{Z}}^d:\mathrm{dist}(\{v\},U)=0 \}$ be the closure of $U$. Let $U^{\circ}:=\{v\in \widetilde{\mathbb{Z}}^d:\mathrm{dist}(\{v\}, \widetilde{\mathbb{Z}}^d \setminus U)>0 \}$ be the interior of $U$. We denote the boundary of $U$ in $\widetilde{\mathbb{Z}}^d$ by $\widetilde{\partial} U:=\overbar{U}\setminus U^{\circ}$.

	\item In this paper, we use $D$ (possibly with superscript or subscript) to represent a subset of $\widetilde{\mathbb{Z}}^d$ consisting of finitely many compact connected components.

\end{itemize}


\subsection{Brownian motion $\big\{\widetilde{S}_t\big\}_{t \ge 0}$ on $\widetilde{\mathbb{Z}}^d$}\label{subsection_BM}
The Brownian motion $\big\{\widetilde{S}_t\big\}_{t \ge 0}$ on the metric graph $\widetilde{\mathbb{Z}}^d$ is constructed as follows. In the interior of some interval $I_e$, $\widetilde{S}_t$ behaves as a one-dimensional standard Brownian motion. When visiting a lattice point $x\in \mathbb{Z}^d$, $\widetilde{S}_t$ uniformly chooses an interval from $\{I_{\{x,y\}}\}_{y\sim x}$ and behaves as a Brownian excursion from $x$ in this interval. Once there is an excursion hitting a point $y$ adjacent to $x$, the next step continues as the same process from the new starting point $y$. The total local time of all Brownian excursions
at $x$ in this single step (i.e. the part of $\widetilde{S}_t$ from $x$ to $y$) is an independent exponential random variable with rate $1$. For any $v\in \widetilde{\mathbb{Z}}^d$, we denote by $\widetilde{\mathbb{P}}_v$ the law of $\big\{\widetilde{S}_t\big\}_{t \ge 0}$ starting from $v$ and denote by $\widetilde{\mathbb{E}}_v$ the expectation under $\widetilde{\mathbb{P}}_v$. As mentioned in \cite[Section 2]{lupu2016loop} (see also \cite[Section 2.5]{cai2023one}), for any $x\in \mathbb{Z}^d$, the projection of $\big\{\widetilde{S}_t\big\}_{t \ge 0}\sim \widetilde{\mathbb{P}}_x$ on $\mathbb{Z}^d$ is exactly a simple random walk $\big\{S_n\big\}_{n \ge 0}$ on $\mathbb{Z}^d$ starting from $x$.


\textbf{Hitting times.} For any $D\subset \widetilde{\mathbb{Z}}^d$, let $\tau_{D}:= \inf\{t\ge 0:\widetilde{S}_t\in D \}$ be the first time when $\widetilde{S}_t$ hits $D$ (we set $\inf\emptyset =+\infty$ for completeness). Especially, when $D=\{v\}$ for some $v\in \widetilde{\mathbb{Z}}^d$, we may omit the braces and write $\tau_{v}:=\tau_{\{v\}}$. Moreover, when $D=\emptyset$, one has $\tau_{\emptyset}=\inf\emptyset=+\infty$.



\subsection{Green's function (for a set)}

For $D\subset \widetilde{\mathbb{Z}}^d$, the Green's function for $D$ is
\begin{equation}
	\widetilde{G}_D(x,y):=\int_{0}^{\infty} \Big\{\widetilde{q}_t(x,y) - \widetilde{\mathbb{E}}_x \big[ \widetilde{q}_{t-\tau_{D}}(\widetilde{S}_{\tau_D},y)\cdot \mathbbm{1}_{\tau_D<t}  \big]\Big\}dt,\ \ \forall x, y\in \widetilde{\mathbb{Z}}^d\setminus D,
\end{equation}
where $\widetilde{q}_t(\cdot,\cdot)$ is the transition density of the Brownian motion on $\widetilde{\mathbb{Z}}^d$ relative to the Lebesgue measure on $\widetilde{\mathbb{Z}}^d$. Note that $\widetilde{G}_D(\cdot,\cdot)$ is decreasing with respect to $D$. For $D=\emptyset$, we abbreviate $\widetilde{G}(\cdot ,\cdot ):=\widetilde{G}_{\emptyset}(\cdot ,\cdot )$. Also note that restricted to $\mathbb{Z}^d$, $\widetilde{G}(\cdot ,\cdot )$ exactly equals to $G(\cdot ,\cdot )$, implying that there exist $\Cl\label{const_green_1}(d)>1>\cl\label{const_green_2}(d)>0$ such that (for brevity, we set $0^{-a}:=1$ for all $a>0$ throughout this paper)
\begin{equation}\label{bound_green}
\cref{const_green_2} |v-w|^{2-d}\le 	\widetilde{G}(v ,w)\le \Cref{const_green_1} |v-w|^{2-d}, \ \ \forall v,w\in \widetilde{\mathbb{Z}}^d. 
\end{equation}
As stated in \cite[Section 3]{lupu2016loop}, $\widetilde{G}_D(\cdot,\cdot)$ is finite, symmetric, continuous and extends continuously to $\widetilde{\partial} D$ by taking value $0$ on the boundary. In addition, the GFF on $\widetilde{\mathbb{Z}}^d\setminus D^{\circ}$ (with zero boundary condition) is exactly the Gaussian field on $\widetilde{\mathbb{Z}}^d\setminus D^{\circ}$, where the covariance between $\widetilde{\phi}_{v}$ and $\widetilde{\phi}_{w}$ equals $\widetilde{G}_D(v,w)$ for all $v,w\in \widetilde{\mathbb{Z}}^d\setminus D^{\circ}$. We denote its law and the corresponding expectation by $\mathbb{P}^D$ and $\mathbb{E}^D$ respectively.

\begin{lemma}\label{lemma2.1}
	For any $d\ge 3$, there exists $C(d)>0$ such that for any $D\subset \widetilde{\mathbb{Z}}^d$ and $x\sim y\in \mathbb{Z}^d\setminus D$ such that $D\cap I_{\{x,y\}}^{\circ}= \emptyset$, 
	\begin{equation}\label{new_add_2.3}
	\widetilde{G}_D(x,x)\le C	\widetilde{\mathbb{P}}_x(\tau_D>\tau_y). 
	\end{equation}
\end{lemma}
\begin{proof}
	Let $z:=\frac{1}{2}(x+y)$ be the midpoint of $I_{\{x,y\}}$. By the strong Markov property and the symmetry of the Green's function, we have
	\begin{equation}\label{new_add_2.4}
		\widetilde{G}_D(x,x) = \tfrac{\widetilde{G}_D(x,z)}{\widetilde{\mathbb{P}}_z(\tau_D>\tau_x)}= \tfrac{  \widetilde{\mathbb{P}}_x(\tau_D>\tau_z)\widetilde{G}_D(z,z)}{\widetilde{\mathbb{P}}_z(\tau_D>\tau_x)}.
	\end{equation} 
	Since $x\sim y$ and $D\cap I_{\{x,y\}}^{\circ}= \emptyset$, one has  
   \begin{equation}\label{new_add_2.5}
   	 \widetilde{\mathbb{P}}_z(\tau_D>\tau_x) \ge \widetilde{\mathbb{P}}_z(\tau_y>\tau_x) \overset{\text{(by symmetry)}}{=} \tfrac{1}{2}. 
   \end{equation}
	Meanwhile, by the strong Markov property, 
	\begin{equation}\label{new_add_2.6}
		\widetilde{\mathbb{P}}_x(\tau_D>\tau_y) \ge \widetilde{\mathbb{P}}_x(\tau_D>\tau_z)  \widetilde{\mathbb{P}}_z(\tau_x>\tau_y) \overset{\text{(by symmetry)}}{=}  \tfrac{1}{2}\widetilde{\mathbb{P}}_x(\tau_D>\tau_z). 
	\end{equation}
	Combining (\ref{new_add_2.4}), (\ref{new_add_2.5}) and (\ref{new_add_2.6}), we get 
	\begin{equation}\label{new_add_2.7} 
			\widetilde{G}_D(x,x) \le 4\widetilde{G}_D(z,z) \widetilde{\mathbb{P}}_x(\tau_D>\tau_y). 
	\end{equation}
	Noting that $\widetilde{G}_D(z,z)\le \widetilde{G}(z,z)\le \Cref{const_green_1}$ (by (\ref{bound_green})), we derive (\ref{new_add_2.3}) from (\ref{new_add_2.7}). 
			\end{proof}

\subsection{Strong Markov property for the GFF}

The strong Markov property is a fundamental property of $\widetilde{\phi}$. To be precise, we introduce the following notation.

\begin{definition}[harmonic average]\label{def_harmonic_average}
	For any $D_1\subset D_2\subset \widetilde{\mathbb{Z}}^d$ and $v\in \widetilde{\mathbb{Z}}^d$, suppose that all values of $\widetilde{\phi}$ on $D_2$ are given. Then we define the average of the boundary condition with respect to the harmonic measure of $D_2$ (i.e. the hitting distribution on $D_2$ of the Brownian motion) restricted to $D_1$ as follows:
	\begin{equation}\label{def_Hv}
		\mathcal{H}_v(D_1,D_2):=\left\{\begin{array}{ll}
			0   &\   \text{if}\ v\in D_2^{\circ}; \\
			\sum\nolimits_{w\in \widetilde{\partial} D_1} \widetilde{\mathbb{P}}_v\big(\tau_{D_2}=\tau_w<\infty \big) \widetilde{\phi}_{w}  &\  \text{otherwise}.
		\end{array}
		\right.
	\end{equation} 
	Note that $\widetilde{\partial} D_1$ is countable since $D_1$ is composed of finitely many connected components and $\widetilde{\mathbb{Z}}^d$ is locally one-dimensional. Especially, when $D_1=D_2=D$, we abbreviate $\mathcal{H}_v(D):=\mathcal{H}_v(D,D)$. 
\end{definition}

\begin{lemma}[{\cite[Theorem 8]{ding2020percolation}}]\label{lemma_strong_markov}
	Suppose that $\mathcal{A}\subset \widetilde{\mathbb{Z}}^d$ is a random compact set measuable with respect to $\{\widetilde{\phi}_v\}_{v\in \widetilde{\mathbb{Z}}^d}$ such that for any open set $U\subset \widetilde{\mathbb{Z}}^d$, the event $\{\mathcal{A}\subset U\}$ is measurable with respect to $\mathcal{F}_U$ (i.e. the $\sigma$-field generated by $\{\widetilde{\phi}_v\}_{v\in U}$). Then conditioning on $\mathcal{F}_{\mathcal{A}}$ (i.e. the $\sigma$-field generated by the configuration of $\mathcal{A}$ and all values of $\widetilde{\phi}$ on $\mathcal{A}$), on the event $\{\mathcal{A}=D\}$ (for some $D\subset \widetilde{\mathbb{Z}}^d$), $\{\widetilde{\phi}_v\}_{v\in \widetilde{\mathbb{Z}}^d\setminus D}$ has the same distribution as $\{\widetilde{\phi}_v'+\mathcal{H}_v(D) \}_{v\in \widetilde{\mathbb{Z}}^d\setminus D}$, where $\{\widetilde{\phi}_v'\}_{v\in \widetilde{\mathbb{Z}}^d\setminus D}$ is an independent GFF with law $\mathbb{P}^{D}$.  
\end{lemma}

For brevity, we write ``$\xleftrightarrow{\widetilde{E}^{\ge 0}}$'' as ``$\xleftrightarrow{\ge 0}$''. We also denote by ``$A_1 \xleftrightarrow{\le 0} A_2$'' the event that $A_1$ and $A_2$ are connected by the negative level-set $\widetilde{E}^{\le 0}:= \big\{v\in \widetilde{\mathbb{Z}}^d: \widetilde{\phi}_v\le 0\big\}$. Note that $A_1 \xleftrightarrow{\ge 0} A_2$ and $A_1 \xleftrightarrow{\le 0} A_2$ occur if $A_1\cap A_2\neq \emptyset$.

\begin{example}[negative cluster] 
	For any non-empty $A\subset \mathbb{Z}^d$, we denote 
	\begin{equation}
		\mathcal{C}_{A}^{-}:= \Big\{v\in \widetilde{\mathbb{Z}}^d: v\xleftrightarrow{\le 0} A  \Big\}.
	\end{equation}
	Note that $A \subset  \mathcal{C}_{A}^{-}$. Moreover, by the continuity of $\widetilde{\phi}_\cdot$, we have $\widetilde{\phi}_w=0$ for all $w\in  \widetilde{\partial} 	\mathcal{C}_{A}^{-}\setminus A$, which implies that  $\mathcal{H}_v(\mathcal{C}_{A}^{-})=\mathcal{H}_v(A,\mathcal{C}_{A}^{-})$ for all $v\in \widetilde{\mathbb{Z}}^d$. Therefore, since $\mathcal{C}_{A}^{-}$ satisfies the condition for $\mathcal{A}$ in Lemma \ref{lemma_strong_markov}, we have: conditioning on $\mathcal{F}_{\mathcal{C}_{A}^{-}}$, on $\{\mathcal{C}_{A}^{-}=D\}$ (for some $D\subset \widetilde{\mathbb{Z}}^d$), $\{\widetilde{\phi}_v\}_{v\in \widetilde{\mathbb{Z}}^d\setminus D}$ has the same distribution as $\{\widetilde{\phi}_v'+\mathcal{H}_v(A,D) \}_{v\in \widetilde{\mathbb{Z}}^d\setminus D}$, where $\{\widetilde{\phi}_v'\}_{v\in \widetilde{\mathbb{Z}}^d\setminus D}$ is an independent GFF with law $\mathbb{P}^{D}$. 
\end{example}

See Figure \ref{fig_hm} for an illustration of the harmonic average $\mathcal{H}_x(\mathcal{C}^-_{\partial B_x(N)})$.


 \begin{figure}[h!]
	\centering
	\includegraphics[width=0.3\textwidth]{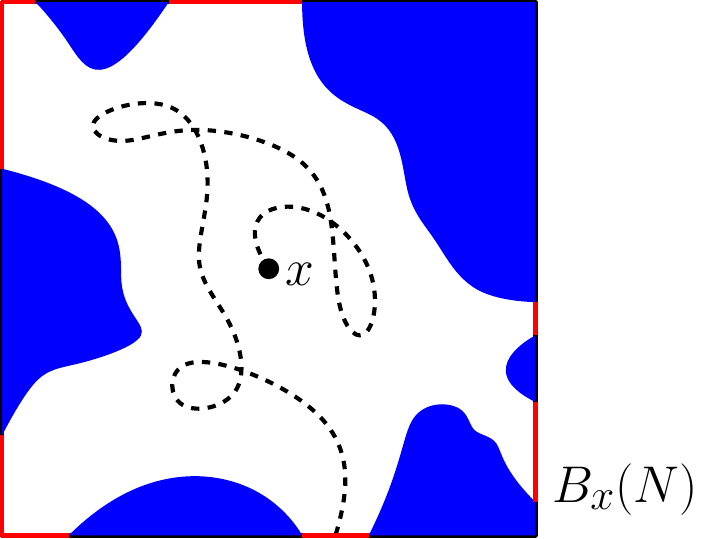}
	\caption{In this figure, the blue area represents the portion of $\mathcal{C}^-_{\partial B_x(N)}$ that lies inside $B_x(N)$, where $x\in \mathbb{Z}^d$ and $N\ge 1$. Note that the GFF values on the boundary of the blue area inside $B_x(N)$ are exactly $0$. The red part of $\partial B_x(N)$ represents the positive boundary values on $\mathcal{C}^-_{\partial B_x(N)}$. The harmonic average $\mathcal{H}_x(\mathcal{C}^-_{\partial B_x(N)})$ is the average of the positive boundary values with respect to the hitting distribution of an independent Brownian motion (i.e. the dashed curve) starting from $x$ and stopped upon hitting $\mathcal{C}^-_{\partial B_x(N)}$.}\label{fig_hm}
\end{figure}

\subsection{Connecting probability}

In this subsection, we review a useful formula in \cite{lupu2018random} for the connecting probability between two sets, and then provide some applications. For any $D\subset \widetilde{\mathbb{Z}}^d$ and $v,w\in \widetilde{\partial} D$, we denote by $\mathbb{K}_D(v,w)$ the effective equivalent
 conductance between $v$ and $w$ in the metric graph $\widetilde{\mathbb{Z}}^d\setminus D^{\circ}$. Note that in \cite{lupu2018random} $\mathbb{K}_D$ is written as $C^{\mathrm{eff}}_D$, and we changed the notation because the letter ``$C$'' is already used to represent a constant. As presented in \cite[Section 2.2]{lupu2018random}, $\mathbb{K}_D$ can be considered as the boundary
excursion kernel. I.e., 
\begin{equation}\label{def_KD}
	\mathbb{K}_D(v,w) := \lim\limits_{\epsilon \to 0+} \epsilon^{-1}\sum\nolimits_{v'\in \widetilde{\mathbb{Z}}^d: \|v'-v\|=\epsilon} \widetilde{\mathbb{P}}_{v'} (\tau_{D}=\tau_{w}<\infty).  
\end{equation}

\begin{lemma}[{\cite[Equation (18)]{lupu2018random}}]\label{lemma_connecting_boundary}
	For any disjoint $D_1,D_2\subset \widetilde{\mathbb{Z}}^d$, suppose that the values of $\widetilde{\phi}$ on $D_1\cup D_2$ are given and are all non-negative. Then the conditional probability of the event $\{D_1\xleftrightarrow{\ge 0}D_2\}$ is given by 
	\begin{equation}
		1-e^{-2\sum_{z_1\in \widetilde{\partial }D_1,z_2\in \widetilde{\partial} D_2}\mathbb{K}_{D_1\cup D_2}(z_1,z_2)\widetilde{\phi}_{z_1}\widetilde{\phi}_{z_2}}. 
	\end{equation}
\end{lemma}

By applying Lemma \ref{lemma_connecting_boundary}, we are able to relate the connecting probability and the harmonic average as follows.

\begin{lemma}\label{lemma_connecting_hm}
	For any non-empty, disjoint $A_1,A_2\subset \mathbb{Z}^d$, 
	\begin{equation}
		\mathbb{P}\big( A_1 \xleftrightarrow{\ge 0} A_2\big) = \mathbb{E}\Big(1- e^{-4\sum\nolimits_{y\in A_1} \widehat{\mathcal{H}}_{y}(A_2, \mathcal{C}^-_{A_1\cup A_2})\widetilde{\phi}_{y} } \Big),
	\end{equation}
	where for any $D_1\subset D_2\subset \widetilde{\mathbb{Z}}^d$ and $y\in \mathbb{Z}^d$, 
	\begin{equation}\label{def_hat_H}
		\widehat{\mathcal{H}}_y(D_1, D_2) := (2d)^{-1}\sum\nolimits_{z\sim y:I_{\{y,z\}}^{\circ}\cap D_2=\emptyset } \mathcal{H}_z(D_1, D_2).
	\end{equation} 
\end{lemma}
\begin{proof}
	Note that $\mathcal{C}_{A_1}^{-}\cup \mathcal{C}_{A_2}^{-}=\mathcal{C}_{A_1\cup A_2}^{-}$. Combining Lemmas \ref{lemma_strong_markov} and \ref{lemma_connecting_boundary} (where we take $D_1=\mathcal{C}_{A_1}^{-}$ and $D_2=\mathcal{C}_{A_2}^{-}$) with the facts that $\widetilde{\phi}_v=0$ for all $i\in \{1,2\}$ and $v\in  \widetilde{\partial} 	\mathcal{C}_{A_i}^{-}\setminus A_i$, we get that 
		\begin{equation}\label{eq_A1_A2}
			\begin{split}
					\mathbb{P}\big( A_1 \xleftrightarrow{\ge 0} A_2\big)=&  \mathbb{E}\Big[\mathbb{P}\Big( A_1 \xleftrightarrow{\ge 0}  A_2 \mid  \mathcal{F}_{	\mathcal{C}_{A_1\cup A_2}^{-}} \Big) \Big]\\
					=& \mathbb{E}\Big(1- e^{- 2\sum_{y_1\in A_1,y_2\in A_2} \mathbb{K}_{\mathcal{C}_{A_1\cup A_2}^{-}}(y_1,y_2) \widetilde{\phi}_{y_1}\widetilde{\phi}_{y_2} }   \Big).
			\end{split} 
	\end{equation}
For any $y_1\in A_1$ and $y_2\in A_2$, it follows from (\ref{def_KD}) that 
	\begin{equation}\label{newadd_2.14}
		\mathbb{K}_{\mathcal{C}_{A_1\cup A_2}^{-}}(y_1,y_2) = \lim\limits_{\epsilon \to 0+} \epsilon^{-1}\sum\nolimits_{z\sim y_1} \widetilde{\mathbb{P}}_{y_1'(z,\epsilon)} \big(\tau_{\mathcal{C}_{A_1\cup A_2}^{-}}=\tau_{y_2}<\infty\big),
	\end{equation}
	where $y_1'(z,\epsilon)$ is the point in $I_{\{y_1,z\}}$ with $\|y_1'(z,\epsilon)-y_1\|=\epsilon$. For any $z\sim y_1$ and small enough $\epsilon>0$, since the Brownian motion from $y_1'(z,\epsilon)$ (if stopped upon hitting $\mathcal{C}_{A_1\cup A_2}^{-}$) must reach $z$ before $A_2$, we have 
	\begin{equation}\label{newadd_2.15}
		\begin{split}
				&\widetilde{\mathbb{P}}_{y_1'(z,\epsilon)} \big(\tau_{\mathcal{C}_{A_1\cup A_2}^{-}}=\tau_{y_2}<\infty\big)\\
				= & \widetilde{\mathbb{P}}_{y_1'(z,\epsilon)} \big(\tau_{\mathcal{C}_{A_1\cup A_2}^{-}}>\tau_{z} \big)\widetilde{\mathbb{P}}_{z} \big(\tau_{\mathcal{C}_{A_1\cup A_2}^{-}}=\tau_{y_2}<\infty\big). 
		\end{split}
	\end{equation}
	Moreover, if $I_{\{z,y_1\}}^{\circ}\cap \mathcal{C}_{A_1\cup A_2}^{-}\neq \emptyset$, then for all sufficiently small $\epsilon>0$, we have $\widetilde{\mathbb{P}}_{y_1'(z,\epsilon)} \big(\tau_{\mathcal{C}_{A_1\cup A_2}^{-}}>\tau_{z} \big)=0$  since within $I_{\{y_1,z\}}$ there must be some point in $\mathcal{C}_{A_1\cup A_2}^{-}$ between $y_1'(z,\epsilon)$ and $z$. Otherwise (i.e. $I_{\{z,y_1\}}^{\circ}\cap \mathcal{C}_{A_1\cup A_2}^{-}= \emptyset$), applying the optional stopping theorem, we have 
	\begin{equation}\label{newadd_2.16}
	\widetilde{\mathbb{P}}_{y_1'(z,\epsilon)} \big(\tau_{\mathcal{C}_{A_1\cup A_2}^{-}}>\tau_{z} \big) = \widetilde{\mathbb{P}}_{y_1'(z,\epsilon)}\big(\tau_{y_1}>\tau_{z}\big)=d^{-1}\epsilon. 
	\end{equation}
	Combining (\ref{newadd_2.14}), (\ref{newadd_2.15}) and (\ref{newadd_2.16}), we have 
	\begin{equation*}
	\mathbb{K}_{\mathcal{C}_{A_1\cup A_2}^{-}}(y_1,y_2) =  d^{-1}\sum\nolimits_{z\sim y_1:I_{\{z,y_1\}}^{\circ}\cap \mathcal{C}_{A_1\cup A_2}^{-} = \emptyset}\widetilde{\mathbb{P}}_{z} \big(\tau_{\mathcal{C}_{A_1\cup A_2}^{-}}=\tau_{y_2}<\infty\big), 
	\end{equation*}
	which implies that 
\begin{equation*}
	\begin{split}
		&\sum\nolimits_{y_2\in A_2} \mathbb{K}_{ \mathcal{C}_{A_1\cup A_2}^{-}}(y_1,y_2) \widetilde{\phi}_{y_2} \\
		=& d^{-1}\sum\nolimits_{y_2\in A_2} \sum\nolimits_{z\sim y_1:I_{\{z,y_1\}}^{\circ}\cap \mathcal{C}_{A_1\cup A_2}^{-} = \emptyset} \widetilde{\mathbb{P}}_{z} \big(\tau_{\mathcal{C}_{A_1\cup A_2}^{-}}=\tau_{y_2}<\infty\big)\widetilde{\phi}_{y_2}\\ \overset{(\ref{def_Hv})}{=}&d^{-1} \sum\nolimits_{z\sim y_1:I_{\{z,y_1\}}^{\circ}\cap \mathcal{C}_{A_1\cup A_2}^{-} = \emptyset}  \mathcal{H}_z(A_2 , \mathcal{C}_{A_1\cup A_2}^{-})
		\overset{(\ref{def_hat_H})}{=} 2\widehat{\mathcal{H}}_{y_1}(A_2 , \mathcal{C}_{A_1\cup A_2}^{-}). 
	\end{split}
\end{equation*}
		Combined with (\ref{eq_A1_A2}), it concludes this lemma.
	\end{proof}

\subsection{Loop soup and isomorphism theorem}\label{section_loop_soup}

The loop soup is a useful tool in the study of percolation for the metric graph GFF (see \cite{cai2023one,drewitz2023arm,lupu2016loop}). Specifically, the loop soup on $\widetilde{\mathbb{Z}}^d$ of intensity $\alpha>0$ (denoted by $\widetilde{\mathcal{L}}_{\alpha}$) is a Poisson point process of rooted loops (i.e. continuous paths that start and end at the same point) on $\widetilde{\mathbb{Z}}^d$ with intensity measure $\alpha\widetilde{\mu}$, where $\widetilde{\mu}$ is defined as (see \cite[Equation (3.1)]{drewitz2023arm})
\begin{equation}
	\widetilde{\mu}(\cdot):= \int_{\widetilde{\mathbb{Z}}^d} \mathrm{d}m(v) \int_{0}^{\infty} t^{-1} \widetilde{q}_t(v,v) \widetilde{\mathbb{P}}^t_{v,v}(\cdot)\mathrm{d} t, 
\end{equation}
where $m(\cdot)$ is the Lebesgue measure on $\widetilde{\mathbb{Z}}^d$ and $\widetilde{\mathbb{P}}^t_{v,v}$ is the law of the Brownian bridge on $\widetilde{\mathbb{Z}}^d$ conditioning on $\widetilde{S}_0=\widetilde{S}_t=v$. By forgetting the roots of all rooted loops, $\widetilde{\mu}$ can also be considered as a measure on the equivalence classes of rooted loops under time-shift (we maintain this understanding throughout this paper). Readers may refer to \cite[Section 2]{lupu2016loop} or \cite[Section 2.6]{cai2023one} for an equivalent construction of $\widetilde{\mu}$ based on the discrete loop soup. For any $D\subset \widetilde{\mathbb{Z}}^d$, let $\widetilde{\mu}^{D}$ be the restriction of $\widetilde{\mu}$ to the space of loops contained in $\widetilde{\mathbb{Z}}^d\setminus D$. Note that $\widetilde{\mathcal{L}}_\alpha^{D}:=\widetilde{\mathcal{L}}_\alpha\cdot \mathbbm{1}_{\widetilde{\ell}\ \text{is contained in}\ \widetilde{\mathbb{Z}}^d\setminus D}$ is a Poisson point process with intensity measure $\alpha\widetilde{\mu}^{D}$.

 By Lupu's isomorphism theorem \cite[Proposition 2.1]{lupu2016loop}, we have the following coupling between the GFF $\{\widetilde{\phi}_v\}_{v\in \widetilde{\mathbb{Z}}^d\setminus D}\sim \mathbb{P}^{D}$ and the occupation field $\{\widehat{\mathcal{L}}^{D,v}_{1/2}\}_{v\in \widetilde{\mathbb{Z}}^d\setminus D}$, where $\widehat{\mathcal{L}}^{D,v}_{1/2}$ is the total local time at $v$ of all loops in $\widetilde{\mathcal{L}}_{1/2}^{D}$. Especially, when $D=\emptyset$, we abbreviate $\widehat{\mathcal{L}}^{v}_{1/2}:=\widehat{\mathcal{L}}^{\emptyset,v}_{1/2}$.

\begin{lemma}\label{lemma_iso}
	For any $D\subset \widetilde{\mathbb{Z}}^d$, there is a coupling between the loop soup $\widetilde{\mathcal{L}}_{1/2}^D$ and the GFF $\{\widetilde{\phi}_v\}_{v\in \widetilde{\mathbb{Z}}^d\setminus D}\sim  \mathbb{P}^{D}$ such that 
	\begin{itemize}
		\item  for any $v\in \widetilde{\mathbb{Z}}^d\setminus D$, $\widehat{\mathcal{L}}^{D,v}_{1/2}=\frac{1}{2}\widetilde{\phi}_v^{2}$;

		\item  the sign clusters of $\widetilde{\phi}_\cdot$ are exactly the clusters composed of loops in $\widetilde{\mathcal{L}}_{1/2}^D$.

	\end{itemize}
	
\end{lemma}

We denote the union of ranges of loops in a point measure $\mathcal{L}$ by $\cup \mathcal{L}$.

\begin{corollary}\label{coro_iso}
	(1) For any $D,U_1,U_2\subset \widetilde{\mathbb{Z}}^d$, 
	\begin{equation}\label{coro2.1_1}
		\mathbb{P}^{D}\big(U_1 \xleftrightarrow{\ge 0} U_2 \big)=  \mathbb{P}^{D}\big(U_1 \xleftrightarrow{\le 0} U_2 \big) =  \tfrac{1}{2}\mathbb{P}\big(U_1 \xleftrightarrow{\cup  \widetilde{\mathcal{L}}_{1/2}^{D}} U_2 \big). 
	\end{equation}
	(2) For any $D_1\subset D_2\subset \widetilde{\mathbb{Z}}^d$, $U\subset \widetilde{\mathbb{Z}}^d$ and $v\in \widetilde{\mathbb{Z}}^d\setminus U$, 
	\begin{equation}\label{coro2.1_2}
	\mathbb{E}^{D_1}\Big(\widetilde{\phi}_v\cdot \mathbbm{1}_{v\xleftrightarrow{\ge 0}U } \Big) \ge 	\mathbb{E}^{D_2}\Big(\widetilde{\phi}_v\cdot \mathbbm{1}_{v\xleftrightarrow{\ge 0}U } \Big). 
	\end{equation}
\end{corollary}
\begin{proof}
	Item (1) is a direct consequence of Lemma \ref{lemma_iso} and the symmetry of $\widetilde{\phi}$. 
	

	For Item (2), by Lemma \ref{lemma_iso} and the symmetry of $\widetilde{\phi}$, we have 
	\begin{equation}\label{add2.15}
		 	\mathbb{E}^{D}\Big(\widetilde{\phi}_v\cdot \mathbbm{1}_{v\xleftrightarrow{\ge 0}U } \Big)= 2^{-\frac{1}{2}}\mathbb{E}\Big(\sqrt{\widehat{\mathcal{L}}_{1/2}^{D,v}}\cdot   \mathbbm{1}_{v\ \text{and}\ U\ \text{are connected by}\ \cup \widetilde{\mathcal{L}}_{1/2}^{D} } \Big). 
	\end{equation}
	Note that the random variable on the right-hand side of (\ref{add2.15}) is increasing with respect to the collection of loops. Thus, since $\widetilde{\mathcal{L}}_{1/2}^{D_1}\ge  \widetilde{\mathcal{L}}_{1/2}^{D_2}$, we obtain Item (2).
\end{proof}

We also review some concepts about loop soups presented in \cite{cai2023one}.

\textbf{Glued loops.} The loops in $\widetilde{\mathcal{L}}_{1/2}$ can be divided into the following types:
\begin{itemize}
	\item  fundamental loop: a loop that visits at least two lattice points;

	\item  point loop: a loop that visits exactly one lattice point;

	\item  edge loop: a loop that is contained by a single interval $I_e$ and visits no lattice point.

\end{itemize} 
Correspondingly, three types of glued loops, each of which is the union of ranges of specific loops, are defined as follows.
\begin{itemize}
	\item  For any connected $A\subset \mathbb{Z}^d$ containing at least two lattice points, the glued fundamental loop supported on $A$ is the union of ranges of fundamental loops in $\widetilde{\mathcal{L}}_{1/2}$ that visit every point in $A$ and do not visit any other lattice point.

	\item  For any $x\in \mathbb{Z}^d$, the glued point loop supported on $x$ is the union of ranges of point loops in $\widetilde{\mathcal{L}}_{1/2}$ including $x$.

	\item  For any $x\sim y\in \mathbb{Z}^d$, the glued edge loop supported on $I_{\{x,y\}}$ is the union of ranges of edge loops in $\widetilde{\mathcal{L}}_{1/2}$ contained in $I_{\{x,y\}}$.

\end{itemize}


\textbf{van den Berg-Kesten-Reimer (BKR) inequality.} For two events $\mathsf{A}$ and $\mathsf{B}$ measurable with respect to $\widetilde{\mathcal{L}}_{1/2}$, let $\mathsf{A} \circ \mathsf{B}$ be the event that there exist two disjoint collections of glued loops such that one collection certifies $\mathsf{A}$, and the other certifies $\mathsf{B}$. Note that in this context, ``two disjoint collections'' implies that the collections do not contain any glued loops with matching types and supports, but it
does not necessarily mean that every glued loop in one collection does not intersect any glued loop in the other collection.

 We say an event $\mathsf{A}$ is a connecting event if there exist two finite subsets $A_1,A_2\subset \mathbb{Z}^d$ such that $\mathsf{A}=\{A_1\xleftrightarrow{\cup \widetilde{\mathcal{L}}_{1/2}}A_2\}$. 

\begin{lemma}[BKR inequality; {\cite[Corollary 3.4]{cai2023one}}]\label{lemma_BKR}
	If events $\mathsf{A}_1,\mathsf{A}_2,...,\mathsf{A}_j$ ($j\ge 2$) are connecting events, then we have
	\begin{equation}
		\mathbb{P}\big(\mathsf{A}_1\circ \mathsf{A}_2\circ ... \circ \mathsf{A}_j\big)\le \prod\nolimits_{1\le i\le j} \mathbb{P}\big(\mathsf{A}_i\big).
	\end{equation}
\end{lemma}

\subsection{Exploration martingale}\label{section_EM}

In this subsection, we review a useful tool called ``exploration martingale'', which in our context went back to \cite{lupu2018random} and was further developed in \cite{ding2020percolation}. We first record some necessary notations as follows.
\begin{itemize}
	\item  For any non-empty $A\subset \mathbb{Z}^d$ and $t\ge 0$, let $\mathcal{I}^{A,+}_t$ (resp. $\mathcal{I}^{A,-}_t$) be the collection of points $v\in \widetilde{\mathbb{Z}}^d$ such that there exists a path $\eta$ of length at most $t$ in $\widetilde{E}^{\ge 0}$ (resp. $\widetilde{E}^{\le 0}$) connecting $v$ and $A$. We denote $\mathcal{I}^{A,\pm}_t:=\mathcal{I}^{A,+}_t\cup \mathcal{I}^{A,-}_t$.


	\item   For any non-empty $A\subset \mathbb{Z}^d$ and $x\in \mathbb{Z}^d\setminus A$, we consider the martingales
	\begin{equation}
		\mathcal{M}_{x,t}^{A,\zeta}= \mathbb{E}\big(\widetilde{\phi}_x\mid  \mathcal{F}_{\mathcal{I}^{A,\zeta}_t}\big), \ \ \forall \zeta\in \{+,-,\pm\}. 
	\end{equation}
	We denote the quadratic variation of $\mathcal{M}_{x,t}^{A,\zeta}$ by $\langle \mathcal{M}_x^{A,\zeta} \rangle_t$.
	

\end{itemize}

The martingale $\mathcal{M}_{x,t}^{A,+}$ was discussed in detail in \cite{ding2020percolation,lupu2018random}. For any $\zeta\in \{+,-,\pm\}$, the quadratic variation $\langle\mathcal{M}_x^{A,\zeta} \rangle_t$ can be written as 
\begin{equation}\label{new_2.17}
	\begin{split}
		\langle \mathcal{M}_x^{A,\zeta} \rangle_t  =&  \widetilde{G}(x,x)-\widetilde{G}_{\mathcal{I}^{A,\zeta}_t}(x,x)\\
		=& 	\sum\nolimits_{v\in \mathcal{I}^{A,\zeta}_t} \widetilde{\mathbb{P}}_x\big(\tau_{\mathcal{I}^{A,\zeta}_t}=\tau_v<\infty\big)\widetilde{G}(x,v). 
	\end{split}
\end{equation}
In addition, the following process is a Brownian motion stopped at time $\langle\mathcal{M}_x^{A,\zeta}\rangle_\infty$: 
\begin{equation}\label{new_2.18}
	W_{x,t}^{A,\zeta}:=\left\{\begin{array}{ll}
		\mathcal{M}_{x,T_t}^{A,\zeta}- \mathcal{M}_{x,0}^{A,\zeta}   &\   \ \forall 0\le t<\langle\mathcal{M}_x^{A,\zeta} \rangle_\infty; \\
		\mathcal{M}_{x,\infty}^{A,\zeta}- \mathcal{M}_{x,0}^{A,\zeta}  &\ \  \forall t\ge\langle\mathcal{M}_x^{A,\zeta}\rangle_\infty,
	\end{array}
	\right. 
\end{equation}
where $T_{t}:= \inf\{s\ge 0:\langle \mathcal{M}_x^{A,\zeta} \rangle_s>t \}$. (\ref{new_2.17}) and (\ref{new_2.18}) were proved for $\zeta=+$ in \cite[Section 2]{ding2020percolation}, and the analogues for $\zeta\in \{-,\pm\}$ follow similarly.

%
﻿ 

\begin{lemma}\label{lemma_EM}
	For any $\zeta\in \{+,-,\pm\}$ and $t,T>0$,  
	\begin{equation}\label{new_2.19}
		\mathbb{P}\Big(\mathcal{M}_{x,\infty}^{A,\zeta}-\mathcal{M}_{x,0}^{A,\zeta} \ge t, \langle\mathcal{M}_x^{A,\zeta}\rangle_\infty\le T \mid \mathcal{F}_{\mathcal{I}^{A,\zeta}_0} \Big)\\
		\le  \mathbb{P}\big(|X|\ge \tfrac{t}{\sqrt{T}}\big),
	\end{equation}
	where $X$ is a standard normal random variable. 
\end{lemma}
\begin{proof}
Let $\{W_s\}_{s\ge 0}$ be a standard Brownian motion. By (\ref{new_2.18}) one has 
\begin{equation}\label{2.20}
	\begin{split}
&	\mathbb{P}\Big(\mathcal{M}_{x,\infty}^{A,\zeta}-\mathcal{M}_{x,0}^{A,\zeta}\ge t, \langle\mathcal{M}_x^{A,\zeta}\rangle_\infty\le T \mid \mathcal{F}_{\mathcal{I}^{A,\zeta}_0} \Big)\\ 
	\le & \mathbb{P}\big(\inf\{s\ge 0:W_s=t \}\le T\big).
	\end{split}
\end{equation}
In addition, by the reflection principle, we have
\begin{equation}\label{2.21}
	\mathbb{P}\big(\inf\{s\ge 0:W_s=t \}\le T\big) = 2\mathbb{P}\big(W_T\ge t\big) = \mathbb{P}\big(|X|\ge \tfrac{t}{\sqrt{T}}\big). 
\end{equation}
Combining (\ref{2.20}) and (\ref{2.21}), we obtain the inequality (\ref{new_2.19}).
\end{proof}

\subsection{Average of GFF values}

The following estimate for the normal distribution will be used multiple times in this papar. Readers may refer to e.g. \cite[Proposition 2.1.2]{vershynin2020high} for a detailed proof.

\begin{lemma}\label{lemma_tail}
	For any $\delta>0$, the normal random variable $X\sim N(0,\delta^2)$ satisfies 
	\begin{equation}
		\mathbb{P}(X\ge t) \le (2\pi)^{-\frac{1}{2}}\delta t^{-1} e^{-\frac{t^2}{2\delta^2}},\ \ \forall t>0.  
	\end{equation}
\end{lemma}

Applying Lemma \ref{lemma_tail}, we derive the following bound for the tail probability of the average of GFF values.


\begin{lemma}\label{lemma_average_phi}
 For any $d\ge 3$, there exists $\cl\label{const_average_2}(d)\in (0,1)$ such that for all $D\subset \widetilde{\mathbb{Z}}^d$, $N\ge 1$ and $s>0$,  
	\begin{equation}\label{newadd_2.30}
	\mathbb{P}^D\Big(|\partial^{\mathrm{e}}\mathcal{B}(N)|^{-1}\sum\nolimits_{x\in [\partial^{\mathrm{e}}\mathcal{B}(N)]\setminus D}\widetilde{\phi}_x\ge sN^{-\frac{d}{2}+1}\Big)\le e^{-\cref{const_average_2}s^2}, 
	\end{equation}
		\begin{equation}\label{newadd_2.30_2}
		\mathbb{P}^D\Big(\sum\nolimits_{x\in [\partial^{\mathrm{e}}\mathcal{B}(N)]\setminus D}\widetilde{\mathbb{P}}_{\bm{0}}\big(\tau_{\partial^{\mathrm{e}}  \mathcal{B}(N)}=\tau_x\big)\widetilde{\phi}_x\ge sN^{-\frac{d}{2}+1}\Big)\le e^{-\cref{const_average_2}s^2}. 
	\end{equation}
\end{lemma}
\begin{proof}
We first prove (\ref{newadd_2.30}). Note that $|\partial^{\mathrm{e}}\mathcal{B}(N)|^{-1}\sum\limits_{x\in [\partial^{\mathrm{e}}\mathcal{B}(N)]\setminus D}\widetilde{\phi}_x$ (where $\widetilde{\phi}_{\cdot}\sim \mathbb{P}^D$) is a mean-zero normal random variable, whose variance is given by 
\begin{equation}\label{def_sigma}
	\begin{split}
		\sigma^2:= |\partial^{\mathrm{e}}\mathcal{B}(N)|^{-2} \sum\nolimits_{x,y\in [\partial^{\mathrm{e}}\mathcal{B}(N)]\setminus D} \widetilde{G}_D(x,y).
	\end{split}
\end{equation}
For any $x,y\in \mathbb{Z}^d$, by (\ref{bound_green}) we have 
\begin{equation}\label{add_2.35}
	\widetilde{G}_D(x,y)\le \widetilde{G}(x,y)\le \Cref{const_green_1} |x-y|^{2-d}.
\end{equation}
Meanwhile, for any $x\in \partial^{\mathrm{e}} \mathcal{B}(N)$ and $l\ge 0$, one has
\begin{equation}\label{add_2.36}
	|\partial^{\mathrm{e}} \mathcal{B}(N)\cap \partial \mathcal{B}_x(l)|\le C(l+1)^{d-2}\cdot \mathbbm{1}_{0\le l\le 4N}. 
\end{equation}
Thus, by (\ref{add_2.35}), (\ref{add_2.36}) and $|\partial^{\mathrm{e}}\mathcal{B}(N)|\asymp N^{d-1}$, we have 
\begin{equation}\label{new_add_2.31}
	\begin{split}
		\sigma^2 \le &  |\partial^{\mathrm{e}}\mathcal{B}(N)|^{-1}\max_{x\in \partial^{\mathrm{e}} \mathcal{B}(N)} \sum\nolimits_{y\in [\partial^{\mathrm{e}}\mathcal{B}(N)]\setminus D} \widetilde{G}_D(x,y)   \\
		\le & CN^{-(d-1)}\sum\nolimits_{0\le l\le 4N}(l+1)^{d-2}\cdot l^{2-d}\le C'N^{2-d}.
	\end{split}
\end{equation}
By Lemma \ref{lemma_tail} and (\ref{new_add_2.31}), the left-hand side of (\ref{newadd_2.30}) is bounded from above by $C_{\ddagger} s^{-1} e^{-c_{\ddagger}s^2}$, where $C_{\ddagger}$ and $c_{\ddagger}$ are constants only depending on $d$. Therefore, when $s\ge C_{\ddagger}$, we obtain (\ref{newadd_2.30}) with $\cref{const_average_2}=c_{\ddagger}\land \frac{1}{2}$. When $0<s<C_{\ddagger}$, since $\mathbb{P}(X\ge t)\le \frac{1}{2}$ holds for all mean-zero random variable $X$ and all non-negative number $t$, we know that (\ref{newadd_2.30}) holds with $\cref{const_average_2}=[\ln(2)C_{\ddagger}^{-2}]\land \frac{1}{2}$. In conclusion, we establish (\ref{newadd_2.30}).



For (\ref{newadd_2.30_2}), by \cite[Lemma 6.3.7]{lawler2010random}, we have 
\begin{equation}\label{2.34}
	\widetilde{\mathbb{P}}_{\bm{0}}\big(\tau_{\partial^{\mathrm{e}}  \mathcal{B}(N)}=\tau_x\big) \asymp  N^{1-d}  \asymp |\partial^{\mathrm{e}}\mathcal{B}(N)|^{-1}, \ \ \forall x\in \partial^{\mathrm{e}}  \mathcal{B}(N). 
\end{equation}
By (\ref{2.34}), we know that the variance of $\sum_{x\in [\partial^{\mathrm{e}}\mathcal{B}(N)]\setminus D}\widetilde{\mathbb{P}}_{\bm{0}}\big(\tau_{\partial^{\mathrm{e}}  \mathcal{B}(N)}=\tau_x\big)\widetilde{\phi}_x$ is of the same order as $\sigma^2$ in (\ref{def_sigma}), and therefore, is $O(N^{2-d})$ (by (\ref{new_add_2.31})). Thus, using the same argument as in the proof of (\ref{newadd_2.30}), we also obtain (\ref{newadd_2.30_2}).
\end{proof}

\section{Bound the crossing probability by one-arm probabilities}\label{Section_crossing}


The aim of this section is to establish the following proposition, which bounds the crossing probability from above by a product of one-arm probabilities. Recall that $\rho_d(n,N)=\mathbb{P}\big[ B(n)\xleftrightarrow{\ge 0} \partial B(N) \big]$ in (\ref{def_crossing}).

\begin{proposition}\label{lemma_bound_crossing}
	For $d\ge 3$, there exists $\Cl\label{const_crossing}(d)> 1$ such that for any $N\ge  n\ge 1$, 
	\begin{equation}\label{ineq_lemma_crossing}
		\rho_d(n,N) \le \Cref{const_crossing}n^{d-2} \theta_d(n)\theta_d(N/4). 
	\end{equation}
\end{proposition}

\begin{remark}\label{remark_thm2}
	The upper bounds in Theorem \ref{thm2} directly follow from Proposition \ref{lemma_bound_crossing}, Theorem \ref{thm1} and the upper bound in (\ref{ineq_high_d}). 
\end{remark}

To prove Proposition \ref{lemma_bound_crossing}, we need some preparations as follows. To begin with, we show that given a point $z$ is connected to a set by $\widetilde{E}^{\ge 0}$, the expected value of $\widetilde{\phi}_z$ is still uniformly bounded.

\begin{lemma}\label{lemma_crossing_prepare}
	For any $d\ge 3$, there exists $\Cl\label{const_crossing_prepare}(d)>0$ such that for any $A\subset \mathbb{Z}^d$ and $z\in\mathbb{Z}^d\setminus A$, we have  
	\begin{equation}\label{ineq_crossing_prepare}
		\mathbb{E}\Big(\widetilde{\phi}_z\cdot \mathbbm{1}_{z\xleftrightarrow{\ge 0}A}  \Big) \le  \Cref{const_crossing_prepare}  \mathbb{P}\big( z\xleftrightarrow{\ge 0}A\big). 
	\end{equation}
\end{lemma}
\begin{proof}
	To derive (\ref{ineq_crossing_prepare}), what we need to prove is indeed a van den Berg-Kesten-Reimer type inequality (note that the FKG inequality yields an inequality in the opposite direction). Usually the idea of establishing such an inequality is to decompose the random variable into two independent (or approximately independent) parts. In general, compared to random fields, achieving such a decomposition is much simpler in the context of Poisson point processes, which motivates us to employ the isomorphism theorem (i.e. Lemma \ref{lemma_iso}) to transfer the problem to that of the loop soup.

	For any $z\in \mathbb{Z}^d$, let $\mathcal{L}_{z}$ be the point measure composed of loops in $\widetilde{\mathcal{L}}_{1/2}$ intersecting $z$. By (\ref{add2.15}), the left-hand side of (\ref{ineq_crossing_prepare}) can be written as 
	\begin{equation}\label{2.12}
		\mathbb{E}\Big(\widetilde{\phi}_z\cdot \mathbbm{1}_{z\xleftrightarrow{\ge 0}A}  \Big) \le 	2^{-\frac{1}{2}} \Big[ \mathbb{E}\Big( \sqrt{\widehat{\mathcal{L}}^z_{1/2}} \cdot \mathbbm{1}_{\mathsf{A}^{(1)}_z}\Big)+\mathbb{E}\Big( \sqrt{\widehat{\mathcal{L}}^z_{1/2}} \cdot \mathbbm{1}_{\mathsf{A}^{(2)}_z}\Big) \Big],  
	\end{equation}
	where $\mathsf{A}^{(1)}_z:=\cup_{y\sim z}\big\{y \xleftrightarrow{\cup  (\widetilde{\mathcal{L}}_{1/2}-\mathcal{L}_{z})} A\big\}$ and $\mathsf{A}^{(2)}_z=\big\{z \xleftrightarrow{\cup  \widetilde{\mathcal{L}}_{1/2}} A\big\}\cap \big[\mathsf{A}^{(1)}_z\big]^c$. Since $\widehat{\mathcal{L}}^z_{1/2}$ is measurable with respect to $\mathcal{L}_{z}$ and hence is independent of the event $\mathsf{A}^{(1)}_z$,
		\begin{equation}\label{adding_3.4}
		\begin{split}
			\mathbb{E}\Big(\sqrt{\widehat{\mathcal{L}}^v_{1/2}} \cdot \mathbbm{1}_{\mathsf{A}^{(1)}_z}\Big)=& \mathbb{E}\Big(\sqrt{\widehat{\mathcal{L}}^v_{1/2}} \Big)  \mathbb{P} \big[ \mathsf{A}^{(1)}_z\big]\\
			\overset{(\text{Lemma}\ \ref{lemma_iso})}{=} &2^{-\frac{1}{2}}\mathbb{E}\big(|\widetilde{\phi}_z| \big)   \mathbb{P} \big[ \mathsf{A}^{(1)}_z\big]
			\le C\sum\nolimits_{y\sim z}  \mathbb{P} \big(y \xleftrightarrow{\cup  \widetilde{\mathcal{L}}_{1/2}} A\big). 
		\end{split}
	\end{equation}
		Moreover, for each $y\sim z$, by the FKG inequality one has  
	\begin{equation}\label{adding_3.5}
		\begin{split}
			\mathbb{P} \big(z \xleftrightarrow{\cup  \widetilde{\mathcal{L}}_{1/2}} A\big) \ge& \mathbb{P} \big( y \xleftrightarrow{\cup  \widetilde{\mathcal{L}}_{1/2}} A, y\xleftrightarrow{\cup  \widetilde{\mathcal{L}}_{1/2}}z \big)\\
			\ge& \mathbb{P} \big( y \xleftrightarrow{\cup  \widetilde{\mathcal{L}}_{1/2}} A \big) \mathbb{P} \big( y\xleftrightarrow{\cup  \widetilde{\mathcal{L}}_{1/2}}z \big)\ge c \mathbb{P} \big( y \xleftrightarrow{\cup  \widetilde{\mathcal{L}}_{1/2}} A \big). 
		\end{split}
	\end{equation} 
	Combining (\ref{adding_3.4}) and (\ref{adding_3.5}), we get
	\begin{equation}\label{2.14}
		\mathbb{E}\Big(\sqrt{\widehat{\mathcal{L}}^z_{1/2}} \cdot \mathbbm{1}_{\mathsf{A}^{(1)}_z}\Big) \le C \mathbb{P} \big(z \xleftrightarrow{\cup  \widetilde{\mathcal{L}}_{1/2}} A\big)\overset{(\ref{coro2.1_1})}{=}2C\mathbb{P} \big(z \xleftrightarrow{\ge 0} A\big).
	\end{equation}



	Next, we estimate the term $	\mathbb{E}\Big(\sqrt{\widehat{\mathcal{L}}^z_{1/2}} \cdot \mathbbm{1}_{\mathsf{A}^{(2)}_z}\Big)$. First, we present a decomposition for loops in $\mathcal{L}_{z}':= \mathcal{L}_{z}\cdot \mathbbm{1}_{\widetilde{\ell}\ \text{intersects}\ z\ \text{and}\ \widetilde{\partial} \widetilde{B}_z(1)}$ (note that $\widetilde{\partial} \widetilde{B}_z(1)=\{w\in \mathbb{Z}^d:w\sim z\}$ by (\ref{def_continuous_box})). Specifically, for any $\widetilde{\ell}\in \mathcal{L}_{z}'$ (recall that we consider $\widetilde{\ell}$ as an equivalence class of rooted loops under time-shift), arbitrarily take a rooted loop $\widetilde{\varrho}:[0,T ]\to \widetilde{\mathbb{Z}}^d$ in $\widetilde{\ell}$ satisfying the following:
	\begin{itemize}
		\item  $\widetilde{\varrho}(0)=z$;

		\item  $\exists t\in (0,T)$ such that $\widetilde{\varrho}(t)\in \widetilde{\partial} \widetilde{B}_z(1)$ and $\widetilde{\varrho}(t')\neq z$ for all $t'\in (t,T)$. 
		
	\end{itemize}
	Then we define a sequence of stopping times for $\widetilde{\varrho}$ as follows: 
	\begin{enumerate}
		\item $\widetilde{\tau}_0:=0$;

		\item for integer $k\ge 0$, $\widetilde{\tau}_{2k+1}:= \inf\{t>\widetilde{\tau}_{2k}:\widetilde{\varrho}(t)\in \widetilde{\partial} \widetilde{B}_z(1) \}$;

		\item for integer $k\ge 0$, $\widetilde{\tau}_{2k+2}:= \inf\{t>\widetilde{\tau}_{2k+1}:\widetilde{\varrho}(t)=z \}$.

	\end{enumerate}
	Let $\kappa$ be the unique integer such that $\widetilde{\tau}_{2\kappa}=T$. For each $1\le i\le \kappa$, we define the $i$-th forward crossing path of $\widetilde{\varrho}$ as the sub-path $\widetilde{\eta}_i^{\mathrm{F}}:[0,\widetilde{\tau}_{2i-1}-\widetilde{\tau}_{2i-2}]\to \widetilde{\mathbb{Z}}^d$ with $\widetilde{\eta}_i^{\mathrm{F}}(s)= \widetilde{\varrho}(\widetilde{\tau}_{2i-2}+s)$, and define the $i$-th backward crossing path of $\widetilde{\varrho}$ as the sub-path $\widetilde{\eta}_i^{\mathrm{B}}:[0,\widetilde{\tau}_{2i}-\widetilde{\tau}_{2i-1}]\to \widetilde{\mathbb{Z}}^d$ with $\widetilde{\eta}_i^{\mathrm{B}}(s)= \widetilde{\varrho}(\widetilde{\tau}_{2i-1}+s)$. Note that $\kappa$, $\{\widetilde{\eta}_i^{\mathrm{F}}\}_{i=1}^{\kappa}$ and $\{\widetilde{\eta}_i^{\mathrm{B}}\}_{i=1}^{\kappa}$ do not depend on the way of choosing $\widetilde{\varrho}\in \widetilde{\ell}$. A general version of this decomposition can be found in \cite[Section 2.6.3]{cai2023one}. We denote by $\mathfrak{L}_z^{\mathrm{F}}$ (resp. $\mathfrak{L}_z^{\mathrm{B}}$) the collection of forward (resp. backward) crossing paths of all loops in $\mathcal{L}_{z}'$. Here are some useful properties of $\mathfrak{L}_z^{\mathrm{F}}$ and $\mathfrak{L}_z^{\mathrm{B}}$: 
	\begin{enumerate}
		\item[(a)]  Every path in $\mathfrak{L}_z^{\mathrm{F}}$ is contained in $\widetilde{B}_z(1)$;

		\item[(b)]  $\widehat{\mathcal{L}}^z_{1/2}$ is measurable with respect to $\mathfrak{L}_z^{\mathrm{F}}$ and $\mathcal{L}_z-\mathcal{L}_z'$ since every path in $\mathfrak{L}_z^{\mathrm{B}}$ has zero local time at $z$;

		\item[(c)]  By the spatial Markov property of loop soups (see \cite[Lemma 2.3]{cai2023one}), conditioning on $\mathfrak{L}_z^{\mathrm{F}}=\{\widetilde{\eta}_i\}_{i=1}^{\hat{\kappa}}$ (let $y_i$ be the last point of $\widetilde{\eta}_i$), the backward crossing paths in $\mathfrak{L}_z^{\mathrm{B}}$ are distributed as independent Brownian motions on $\widetilde{\mathbb{Z}}^d$ with law $\widetilde{\mathbb{P}}_{y_i}\big(\{\widetilde{S}_t\}_{0\le t\le \tau_{z}}\in \cdot \mid \tau_{z}<\infty\big)$ for $1\le i\le \hat{\kappa}(\mathfrak{L}_z^{\mathrm{F}})$.


	\end{enumerate}

	It follows from Property (b) that
	\begin{equation}\label{add_3.7}
			\mathbb{E}\Big(\sqrt{\widehat{\mathcal{L}}^z_{1/2}}\cdot  \mathbbm{1}_{\mathsf{A}^{(2)}_z}\Big) = \mathbb{E}\Big(\sqrt{\widehat{\mathcal{L}}^z_{1/2}}\cdot \mathbb{P}\big[ \mathsf{A}^{(2)}_z\mid \mathfrak{L}_z^{\mathrm{F}},\mathcal{L}_z-\mathcal{L}_z'\big]\Big).
	\end{equation}
		Moreover, on the event $\mathsf{A}^{(2)}_z$, there exists $\widetilde{\eta}'\in \mathfrak{L}_z^{\mathrm{B}}$ such that $\widetilde{\eta}'$ is connected to $A$ by $\cup  (\widetilde{\mathcal{L}}_{1/2}-\mathcal{L}_{z})$. This is because on $\big[\mathsf{A}^{(1)}_z\big]^c$ we have that $\widetilde{B}_z(1)$ is not connected to $A$ by $\cup  (\widetilde{\mathcal{L}}_{1/2}-\mathcal{L}_{z})$, and in addition all paths in $\mathfrak{L}_z^{\mathrm{F}}$ and loops in $\mathcal{L}_z-\mathcal{L}_z'$ are contained in $\widetilde{B}_z(1)$ (by Property (a)). As a result, 
		\begin{equation}\label{add3.8}
			\mathbb{P}\big[ \mathsf{A}^{(2)}_z\mid \mathfrak{L}_z^{\mathrm{F}},\mathcal{L}_z-\mathcal{L}_z'\big] \le \mathbb{P}\big(\exists\ \widetilde{\eta}'\in \mathfrak{L}_z^{\mathrm{B}}\ \text{with}\ \widetilde{\eta}' \xleftrightarrow{\cup  (\widetilde{\mathcal{L}}_{1/2}-\mathcal{L}_{z})} A \mid \mathfrak{L}_z^{\mathrm{F}}\big),
		\end{equation}
	where we also used the fact that $\mathfrak{L}_z^{\mathrm{B}}$ and $\widetilde{\mathcal{L}}_{1/2}-\mathcal{L}_{z}$ are both independent of $\mathcal{L}_z-\mathcal{L}_z'$. For any non-empty configuration of $\mathfrak{L}_z^{\mathrm{F}}$, say $\{\widetilde{\eta}_i\}_{i=1}^{\hat{\kappa}}$ (recall that $y_i$ denotes the last point of $\widetilde{\eta}_i$), let $\widetilde{\eta}_i'$ be the backward crossing path right after $\widetilde{\eta}_i$. By the union bound and Property (c), we have 
	\begin{equation}\label{newadd_3.6}
		\begin{split}
			&\mathbb{P}\big(\exists\ \widetilde{\eta}'\in \mathfrak{L}_z^{\mathrm{B}}\ \text{with}\ \widetilde{\eta}' \xleftrightarrow{\cup  (\widetilde{\mathcal{L}}_{1/2}-\mathcal{L}_{z})} A \mid \mathfrak{L}_z^{\mathrm{F}}=\{\widetilde{\eta}_i\}_{i=1}^{\hat{\kappa}}\big)\\
			\le &\sum\nolimits_{1\le i\le \hat{\kappa}} \mathbb{P}\big( \widetilde{\eta}'_i \xleftrightarrow{\cup  (\widetilde{\mathcal{L}}_{1/2}-\mathcal{L}_{z})} A \mid \mathfrak{L}_z^{\mathrm{F}}=\{\widetilde{\eta}_i\}_{i=1}^{\hat{\kappa}}\big)\\
			= & \sum\nolimits_{1\le i\le \hat{\kappa}}\mathbb{P}\big( \widetilde{\eta}''_i \xleftrightarrow{\cup  (\widetilde{\mathcal{L}}_{1/2}-\mathcal{L}_{z})} A \big),
		\end{split}
	\end{equation}
	where every $\widetilde{\eta}_i''\sim \widetilde{\mathbb{P}}_{y_i}\big(\{\widetilde{S}_t\}_{0\le t\le \tau_{z}}\in \cdot \mid \tau_{z}<\infty\big)$ is a Brownian motion independent of $\widetilde{\mathcal{L}}_{1/2}$. For each $1\le i\le \hat{\kappa}$, we denote $\mathcal{L}_{z,y_i}:=\mathcal{L}_{z}\cdot \mathbbm{1}_{\widetilde{\ell}\ \text{intersects}\ y_i}$. Since $\mathcal{L}_{z,y_i}$, $ \widetilde{\eta}_i''$ and $\widetilde{\mathcal{L}}_{1/2}-\mathcal{L}_{z}$ are independent, one has 
	\begin{equation}\label{newadd_3.7}
		\begin{split}
			&\mathbb{P}\big( \widetilde{\eta}''_i \xleftrightarrow{\cup  (\widetilde{\mathcal{L}}_{1/2}-\mathcal{L}_{z})} A,\mathcal{L}_{z,y_i}\neq \emptyset \big)\\
			=& \mathbb{P}\big( \widetilde{\eta}''_i \xleftrightarrow{\cup  (\widetilde{\mathcal{L}}_{1/2}-\mathcal{L}_{z})} A \big)\mathbb{P}\big( \mathcal{L}_{z,y_i}\neq \emptyset \big) \ge c\mathbb{P}\big( \widetilde{\eta}''_i \xleftrightarrow{\cup  (\widetilde{\mathcal{L}}_{1/2}-\mathcal{L}_{z})} A \big). 
		\end{split}
	\end{equation}
	Meanwhile, since $ \widetilde{\eta}_i''$ is stochastically dominated by $\cup \mathcal{L}_{z,y_i}$ when $\mathcal{L}_{z,y_i}\neq \emptyset$ happens, 
	\begin{equation}\label{newadd_3.7_2}
		\begin{split}
			\mathbb{P}\big( \widetilde{\eta}''_i \xleftrightarrow{\cup  (\widetilde{\mathcal{L}}_{1/2}-\mathcal{L}_{z})} A,\mathcal{L}_{z,y_i}\neq \emptyset \big)\le   \mathbb{P}\big(\cup\mathcal{L}_{z,y_i} \xleftrightarrow{\cup  (\widetilde{\mathcal{L}}_{1/2}-\mathcal{L}_{z})} A \big) \le \mathbb{P} \big(z \xleftrightarrow{\cup  \widetilde{\mathcal{L}}_{1/2}} A\big). 
		\end{split}
	\end{equation}
		Combining (\ref{newadd_3.6}), (\ref{newadd_3.7}) and (\ref{newadd_3.7_2}), we have 
	\begin{equation}\label{newadd_3.7_3}
		\mathbb{P}\big(\exists\ \widetilde{\eta}'\in \mathfrak{L}_z^{\mathrm{B}}\ \text{with}\ \widetilde{\eta}' \xleftrightarrow{\cup  (\widetilde{\mathcal{L}}_{1/2}-\mathcal{L}_{z})} A \mid \mathfrak{L}_z^{\mathrm{F}}\big) \le C\hat{\kappa}\mathbb{P} \big(z \xleftrightarrow{\cup  \widetilde{\mathcal{L}}_{1/2}} A\big),
	\end{equation}
	which together with (\ref{add_3.7}) and (\ref{add3.8}) implies that 
\begin{equation}\label{2.15}
	\mathbb{E}\Big(\sqrt{\widehat{\mathcal{L}}^z_{1/2}}\cdot  \mathbbm{1}_{\mathsf{A}^{(2)}_z}\Big) \le C\mathbb{P} \big(z \xleftrightarrow{\cup  \widetilde{\mathcal{L}}_{1/2}} A\big) \mathbb{E}\Big(\sqrt{\widehat{\mathcal{L}}^z_{1/2}}\cdot \hat{\kappa}\Big). 
\end{equation}

	The spatial Markov property of loop soups implies that given $\hat{\kappa}$ (i.e. the number of paths in $\mathfrak{L}_z^{\mathrm{F}}$), the paths in $\mathfrak{L}_z^{\mathrm{F}}$ are independent Brownian motions on $\widetilde{\mathbb{Z}}^d$ with law $\widetilde{\mathbb{P}}_{z}\big(\{\widetilde{S}_t\}_{0\le t\le \tau_{\widetilde{\partial} \widetilde{B}_z(1)}}\in \cdot\big)$, whose local times at $z$ are independent exponential random variables with rate $1$. As a result, one has 
	\begin{equation}\label{kappa_2}
		\mathbb{E}\big(\hat{\kappa}^{2}\big)\le C \mathbb{E}\Big[\big(\widehat{\mathcal{L}}^z_{1/2}\big)^2\Big]. 
	\end{equation}
	 Thus, by the Cauchy-Schwarz inequality and the inequality that $(\mathbb{E}|X|)^2 \le \mathbb{E}(X^2)$,  
	 \begin{equation*}
	 	\begin{split}
	 	\mathbb{E}\Big(\sqrt{\widehat{\mathcal{L}}^z_{1/2}}\cdot \hat{\kappa}\Big)\le & \Big(\mathbb{E}\big|\widehat{\mathcal{L}}^z_{1/2}\big|\Big)^{\frac{1}{2}}\Big[\mathbb{E}\big(\hat{\kappa}^{2}\big)\Big]^{\frac{1}{2}}\\
	 	 \le  & \Big\{\mathbb{E}\Big[\big(\widehat{\mathcal{L}}^z_{1/2}\big)^2\Big]\Big\}^{\frac{1}{4}}\Big[\mathbb{E}\big(\hat{\kappa}^{2}\big)\Big]^{\frac{1}{2}}\\
	 	  \overset{(\ref{kappa_2})}{\le } & C\Big\{\mathbb{E}\Big[\big(\widehat{\mathcal{L}}^z_{1/2}\big)^2\Big]\Big\}^{\frac{3}{4}} \overset{(\text{Lemma}\ \ref{lemma_iso})}{=} 2^{-\frac{3}{2}}C\Big[\mathbb{E}\big(\widetilde{\phi}_z^4\big)\Big]^{\frac{3}{4}}< C',
	 	\end{split} 	
	 \end{equation*}
	 	which together with (\ref{2.15}) implies that 
	 \begin{equation}\label{2.16}
	 	\mathbb{E}\Big(\sqrt{\widehat{\mathcal{L}}^z_{1/2}}\cdot  \mathbbm{1}_{\mathsf{A}^{(2)}_z}\Big) \le C\mathbb{P} \big(z \xleftrightarrow{\cup  \widetilde{\mathcal{L}}_{1/2}} A\big)\overset{(\ref{coro2.1_1})}{=}2C\mathbb{P} \big(z \xleftrightarrow{\ge 0} A\big). 
	 \end{equation}
	 	Combining (\ref{2.12}), (\ref{2.14}) and (\ref{2.16}), we conclude this lemma. 
	 \end{proof}

The subsequent lemma, taking Lemma \ref{lemma_crossing_prepare} as an input, shows that the expected values of harmonic averages can be bounded from above by one-arm probabilities.

\begin{lemma}\label{lemma3.4}
	For any $d\ge 3$, there exists $C(d)>0$ such that for any non-empty $A\subset \mathbb{Z}^d$ and $y\in \mathbb{Z}^d\setminus A$,  
	\begin{equation}\label{add_3.8}
		\mathbb{E}\Big[\mathcal{H}_y(A, \mathcal{C}^-_{A})\Big]\le C \mathbb{P}\big(y \xleftrightarrow{\ge 0} A \big). 
	\end{equation}
\end{lemma}

\begin{proof}
	Conditioning on $\mathcal{F}_{\mathcal{C}^{-}_{A}}$, if $y\in ( \mathcal{C}^-_{A})^{\circ}$ happens (i.e. $y \xleftrightarrow{\le 0} A$), then we have $\mathcal{H}_y(A, \mathcal{C}^-_{A})=0$; otherwise (i.e. $\{y \xleftrightarrow{\le 0} A\}^c$), by Lemma \ref{lemma_strong_markov}, we know that the conditional expectation of $\widetilde{\phi}_y$ equals   $\mathcal{H}_y(A,  \mathcal{C}^-_{A})$. Therefore, we have 
	\begin{equation}\label{3.9}
		\mathbb{E}\Big[\mathcal{H}_y(A,  \mathcal{C}^-_{A})\Big] = \mathbb{E}\Big( \widetilde{\phi}_y \cdot \mathbbm{1}_{\{y\xleftrightarrow{\le 0} A\}^c} \Big). 
	\end{equation}
	In addition, by the symmetry of $\widetilde{\phi}$, one has 
	\begin{equation*}
		\begin{split}
			 \mathbb{E}\Big( \widetilde{\phi}_y \cdot \mathbbm{1}_{\{y\xleftrightarrow{\le 0} A\}^c} \Big)
			=\mathbb{E}\big(\widetilde{\phi}_y \big)- \mathbb{E}\Big(\widetilde{\phi}_y \cdot \mathbbm{1}_{y\xleftrightarrow{\le 0} A} \Big)
			= \mathbb{E}\Big(\widetilde{\phi}_y \cdot \mathbbm{1}_{y\xleftrightarrow{\ge 0} A} \Big). 
		\end{split}
	\end{equation*}
	Combined with (\ref{3.9}) and Lemma \ref{lemma_crossing_prepare}, it implies the desired bound (\ref{add_3.8}). 
	\end{proof}

For any $N,n\ge 1$ with $N\ge 100dn$, we denote $\mathcal{C}_{n,N}^{-}:=\mathcal{C}^-_{\partial B(n)}\cup \mathcal{C}^-_{\partial B(N)}$ and 
\begin{equation}\label{def_H_in}
	\overbar{\mathcal{H}}^{\mathrm{in}}= |\partial \mathcal{B}(4dn)|^{-1} \sum\nolimits_{z\in \partial \mathcal{B}(4dn)} \mathcal{H}_z\big(\partial B(n),\mathcal{C}_{n,N}^{-}\big),
\end{equation}
\begin{equation}\label{def_H_out}
	\overbar{\mathcal{H}}^{\mathrm{out}}= |\partial^{\mathrm{e}} \mathcal{B}(5N/8)|^{-1} \sum\nolimits_{z\in \partial^{\mathrm{e}} \mathcal{B}(5N/8)}  \mathcal{H}_z\big(\partial B(N),\mathcal{C}_{n,N}^{-}\big). 
\end{equation}

\begin{lemma}\label{lemma_hat_Hy}
	Recall $\widehat{\mathcal{H}}_{\cdot}$ in (\ref{def_hat_H}). Then we have 
	\begin{equation}\label{3.20_hat_Hy}
		\sum\nolimits_{y\in \partial B(n)} \widehat{\mathcal{H}}_{y}(\partial B(N),\mathcal{C}_{n,N}^{-})\widetilde{\phi}_{y}\le C(d)n^{d-2}\overbar{\mathcal{H}}^{\mathrm{in}}   \overbar{\mathcal{H}}^{\mathrm{out}}. 
	\end{equation}
\end{lemma}
\begin{proof}

		For any $y\in\partial B(n)$, $y'\sim y$ with $I_{\{y,y'\}}^{\circ}\cap \mathcal{C}_{n,N}^{-}=\emptyset$ and for any $z\in \partial B(N)$, by the last-exit decomposition for the Brownian motion (see e.g. \cite[Section 8.2]{morters2010brownian}),
	\begin{equation}\label{3.10}
		\begin{split}
		&\widetilde{\mathbb{P}}_{y'}\big(\tau_{\mathcal{C}_{n,N}^{-}}=\tau_{z}\big)\\
		\le  &C\sum_{z_1\in \partial \mathcal{B}(4dn)}\widetilde{G}_{\mathcal{C}_{n,N}^{-}}(y',z_1) \sum_{z_1'\in \partial^{\mathrm{e}} \mathcal{B}(4dn): z_1' \sim z_1}  \widetilde{\mathbb{P}}_{z_1'}\big(\tau_{\mathcal{C}_{n,N}^{-}\cup \mathcal{B}(4dn)}=\tau_{z}\big).
		\end{split} 
	\end{equation}
	Moreover, by the strong Markov property, we have 
	\begin{equation}\label{3.11}
		\begin{split}
			&\widetilde{\mathbb{P}}_{z_1'}\big(\tau_{\mathcal{C}_{n,N}^{-}\cup \mathcal{B}(4dn)}=\tau_{z}\big)\\
			 \le  &\sum\nolimits_{z_2\in \partial^{\mathrm{e}} \mathcal{B}(5N/8)} \widetilde{\mathbb{P}}_{z_1'}\big(\tau_{\mathcal{B}(4dn)}>\tau_{\partial^{\mathrm{e}}\mathcal{B}(5N/8)}=\tau_{z_2}\big) \widetilde{\mathbb{P}}_{z_2}\big(\tau_{\mathcal{C}_{n,N}^{-}}=\tau_{z}\big).	
		\end{split}
	\end{equation}
	In addition, by the strong Markov property and \cite[Lemmas 6.3.4 and 6.3.7]{lawler2010random} (recalling that the projection of $\widetilde{S}_{\cdot}$ on $\mathbb{Z}^d$ is a simple random walk), we have: for any $z_1\in \partial \mathcal{B}(4dn)$ and $z_1'\in \partial^{\mathrm{e}} \mathcal{B}(4dn)$ with $z_1 \sim z_1'$,
	\begin{equation}\label{3.12}
		\begin{split}
			&\widetilde{\mathbb{P}}_{z_1'}\big(\tau_{\mathcal{B}(4dn)}>\tau_{\partial^{\mathrm{e}} \mathcal{B}(5N/8)}=\tau_{z_2}\big)\\
			\le& \widetilde{\mathbb{P}}_{z_1'}\big(\tau_{\mathcal{B}(4dn)}>\tau_{\partial^{\mathrm{e}} \mathcal{B}(8dn)}\big) \max_{w\in \partial^{\mathrm{e}} \mathcal{B}(8dn)}\widetilde{\mathbb{P}}_{w}\big(\tau_{\partial^{\mathrm{e}} \mathcal{B}(5N/8)}=\tau_{z_2}\big)
			\le  Cn^{-1}N^{-d+1}.  
		\end{split}
	\end{equation}
 By (\ref{3.10}), (\ref{3.11}) and (\ref{3.12}), we get
	\begin{equation*}
	\widetilde{\mathbb{P}}_{y'}\big(\tau_{\mathcal{C}_{n,N}^{-}}=\tau_{z}\big) \le Cn^{-1}N^{-d+1} \sum_{z_1\in \partial \mathcal{B}(4dn)}\widetilde{G}_{\mathcal{C}_{n,N}^{-}}(y',z_1) \sum_{z_2\in \partial^{\mathrm{e}} \mathcal{B}(5N/8) } \widetilde{\mathbb{P}}_{z_2}\big(\tau_{\mathcal{C}_{n,N}^{-}}=\tau_{z}\big).
	\end{equation*}
Combined with the fact that all points in $\mathcal{C}_{n,N}^{-}$ with positive harmonic measures have non-negative GFF values, it implies that 
	\begin{equation}\label{3.13}
	\begin{split}
		&\mathcal{H}_{y'}\big(\partial B(N),\mathcal{C}_{n,N}^{-}\big)\\
		\le &Cn^{-1}N^{-d+1} \sum_{z_1\in  \partial \mathcal{B}(4dn)} \widetilde{G}_{\mathcal{C}_{n,N}^{-}}(y',z_1) \sum_{z_2\in \partial^{\mathrm{e}} \mathcal{B}(5N/8) }\mathcal{H}_{z_2}\big(\partial B(N),\mathcal{C}_{n,N}^{-}\big)\\
		\le & C'n^{-1}\overbar{\mathcal{H}}^{\mathrm{out}} \sum_{z_1\in  \partial \mathcal{B}(4dn)} \widetilde{G}_{\mathcal{C}_{n,N}^{-}}(y',z_1),
	\end{split}	 
\end{equation}
	where in the last inequality we used $|\partial^{\mathrm{e}} \mathcal{B}(5N/8)|\asymp N^{d-1}$ and (\ref{def_H_out}). Meanwhile, since $I_{\{y,y'\}}^{\circ}\cap \mathcal{C}_{n,N}^{-}=\emptyset$, it follows from Lemma \ref{lemma2.1} that
	\begin{equation}\label{newadd_3.17}
		\widetilde{G}_{\mathcal{C}_{n,N}^{-}}(y',y')\le C\widetilde{\mathbb{P}}_{y'}\big(\tau_{\mathcal{C}_{n,N}^{-}}=\tau_{y}\big).
	\end{equation} 
Therefore, by the strong Markov property and (\ref{newadd_3.17}), one has 
	\begin{equation*}
		\begin{split}
			\widetilde{\mathbb{P}}_{z_1}\big(\tau_{\mathcal{C}_{n,N}^{-}}=\tau_{y}\big)\ge&\widetilde{\mathbb{P}}_{z_1}\big(\tau_{\mathcal{C}_{n,N}^{-}}>\tau_{y'}\big)\widetilde{\mathbb{P}}_{y'}\big(\tau_{\mathcal{C}_{n,N}^{-}}=\tau_{y}\big)\\
			=&  \tfrac{\widetilde{\mathbb{P}}_{y'}\big(\tau_{\mathcal{C}_{n,N}^{-}}=\tau_{y}\big)}{\widetilde{G}_{\mathcal{C}_{n,N}^{-}}(y',y')}\cdot \widetilde{G}_{\mathcal{C}_{n,N}^{-}}(y',z_1) \overset{(\ref{newadd_3.17})}{\ge} c\widetilde{G}_{\mathcal{C}_{n,N}^{-}}(y',z_1).
		\end{split}  
	\end{equation*}
	Combined with (\ref{3.13}) (recall $\widehat{\mathcal{H}}_\cdot$ in (\ref{def_hat_H}) and $\overbar{\mathcal{H}}^{\mathrm{in}}$ in (\ref{def_H_in})), it concludes (\ref{3.20_hat_Hy}): 
	\begin{equation}
		\begin{split}
			&	\sum\nolimits_{y\in \partial B(n)} \widehat{\mathcal{H}}_{y}(\partial B(N),\mathcal{C}_{n,N}^{-})\widetilde{\phi}_{y}\\
			\le &Cn^{-1}\overbar{\mathcal{H}}^{\mathrm{out}}  \sum\nolimits_{z_1\in \partial \mathcal{B}(4dn)} \sum\nolimits_{y\in \partial B(n)} \widetilde{\mathbb{P}}_{z_1}\big(\tau_{\mathcal{C}_{n,N}^{-}}=\tau_{y}\big)\widetilde{\phi}_{y} \\
			=&  Cn^{-1}\overbar{\mathcal{H}}^{\mathrm{out}}  \sum\nolimits_{z_1\in \partial \mathcal{B}(4dn)}\mathcal{H}_{z_1}\big(\partial B(n),\mathcal{C}_{n,N}^{-}\big) \overset{(|\partial \mathcal{B}(4dn)|\asymp n^{d-1})}{\le}
			 C'n^{d-2}\overbar{\mathcal{H}}^{\mathrm{out}}\overbar{\mathcal{H}}^{\mathrm{in}}.   \qedhere 
		\end{split}
	\end{equation}
	
	\end{proof}


Next, we present a uniform upper bound for the hitting distribution on a ball.

\begin{lemma}\label{lemma_uniform_distribution}
	For $d\ge 3$, there exists $C(d)>0$ such that for any $n\ge 1$, $y\in \partial \mathcal{B}(n)$, any $z\in \mathbb{Z}^d\setminus \mathcal{B}(10n)$, 
	\begin{equation}\label{new_add_3.22}
		\widetilde{\mathbb{P}}_z\big(\tau_{\partial \mathcal{B}(n)}=\tau_{y}<\infty\big) \le Cn^{-1}|z|^{2-d}. 
	\end{equation}
\end{lemma}
\begin{proof}
	   For any $y\in \partial \mathcal{B}(n)$, $y'\in \partial^{\mathrm{e}} \mathcal{B}(n)$ with $y'\sim y$ and for any $z\in \mathbb{Z}^d\setminus \mathcal{B}(10n)$, by the strong Markov property and the symmetry of the Green's function, we have 
	   \begin{equation}\label{new_add_3.18}
	   	\begin{split}
	   			\widetilde{\mathbb{P}}_z\big( \tau_{y'}<\tau_{ \mathcal{B}(n)},\tau_{y'}<\infty\big)=&\widetilde{G}_{\mathcal{B}(n)}(z,y')\big[\widetilde{G}_{\mathcal{B}(n)}(y',y')\big]^{-1}\\
	   		=&\tfrac{\widetilde{G}_{\mathcal{B}(n)}(z,z)}{\widetilde{G}_{\mathcal{B}(n)}(y',y')}\cdot \widetilde{\mathbb{P}}_{y'}\big( \tau_{z}<\tau_{ \mathcal{B}(n)},\tau_{z}<\infty\big)\\
	   		\le &C\widetilde{\mathbb{P}}_{y'}\big( \tau_{z}<\tau_{ \mathcal{B}(n)},\tau_{z}<\infty\big). 
	   	\end{split}
	   \end{equation}
	   Note that the Brownian motion starting from $y'$ must hit $\partial^{\mathrm{e}} \mathcal{B}(2n)$ before $z$. Thus, by the strong Markov property and \cite[Lemmas 6.3.4 and 6.4.2]{lawler2010random} (recalling that the projection of $\widetilde{S}_{\cdot}$ on $\mathbb{Z}^d$ is a simple random walk),  
	   \begin{equation*}
	   	\begin{split}
	   		&\widetilde{\mathbb{P}}_{y'}\big( \tau_{z}<\tau_{ \mathcal{B}(n)},\tau_{z}<\infty\big)\\
	   		\le &\widetilde{\mathbb{P}}_{y'}\big( \tau_{\partial^{\mathrm{e}} \mathcal{B}(2n)}<\tau_{\mathcal{B}(n)}\big)\max_{w\in \partial^{\mathrm{e}} \mathcal{B}(2n) }\widetilde{\mathbb{P}}_{w}\big( \tau_{z}<\infty\big)\le Cn^{-1}|z|^{2-d}. 
	   	\end{split}  
	   \end{equation*}	
	    Combined with (\ref{new_add_3.18}), this implies that 
	   \begin{equation}\label{new_add_3.19}
	   		\widetilde{\mathbb{P}}_z\big( \tau_{y'}<\tau_{ \mathcal{B}(n)},\tau_{y'}<\infty\big) \le Cn^{-1}|z|^{2-d}. 
	   \end{equation}
	   Note that if the Brownian motion starts from $z$ and first hits $\partial \mathcal{B}(n)$ at $y$, then it must reach some $y'\in \partial^{\mathrm{e}} \mathcal{B}(n)$ with $y'\sim y$ before $\mathcal{B}(n)$. Therefore, by (\ref{new_add_3.19}), we obtain the desired bound:
	\begin{equation*}
		\widetilde{\mathbb{P}}_z\big(\tau_{\partial \mathcal{B}(n)}=\tau_{y}<\infty\big) \le \sum_{y'\in \partial^{\mathrm{e}} \mathcal{B}(n):y'\sim y}\widetilde{\mathbb{P}}_z\big( \tau_{y'}<\tau_{ \mathcal{B}(n)},\tau_{y'}<\infty\big) \le Cn^{-1}|z|^{2-d}. \qedhere
	\end{equation*}
	\end{proof}

The following lemma is crucial for the proof of Proposition \ref{lemma_bound_crossing}, which helps us handle the correlation between $\overbar{\mathcal{H}}^{\mathrm{in}}$ and $\overbar{\mathcal{H}}^{\mathrm{out}}$. We denote 
\begin{equation}\label{new_add_3.17}
	\widecheck{\mathcal{H}}^{\mathrm{in}}:= |\partial \mathcal{B}(4dn)|^{-1} \sum\nolimits_{z\in \partial \mathcal{B}(4dn)} \mathcal{H}_z\big(\partial B(n),\mathcal{C}_{\partial B(n)}^{-}\big).  
\end{equation}
Note that $\widecheck{\mathcal{H}}^{\mathrm{in}}$ can be obtained from $\overbar{\mathcal{H}}^{\mathrm{in}}$ by replacing $\mathcal{C}_{n,N}^{-}$ with $\mathcal{C}_{\partial B(n)}^{-}$. Moreover, since $\mathcal{C}_{\partial B(n)}^{-}\subset \mathcal{C}_{n,N}^{-}$ and $\widetilde{\phi}_z\ge 0$ for all $z\in \partial B(n)\cap \widetilde{\partial}\mathcal{C}_{\partial B(n)}^{-}$, we have $\overbar{\mathcal{H}}^{\mathrm{in}} \le \widecheck{\mathcal{H}}^{\mathrm{in}}$. Furthermore, we derive from Lemma \ref{lemma3.4} that 
\begin{equation}\label{ineq_coro_lemma3.4}
	\begin{split}
			\mathbb{E}\big(\widecheck{\mathcal{H}}^{\mathrm{in}}\big)\le& \max_{z\in \partial \mathcal{B}(4dn)} \mathbb{E}\Big[\mathcal{H}_z\big(\partial B(n),\mathcal{C}_{\partial B(n)}^{-}\big)\Big] \le   C\theta_d(n). 
	\end{split}
\end{equation}

%
%

\begin{lemma}\label{lemma_3.6}
	For any $z\in \partial^{\mathrm{e}} \mathcal{B}(5N/8)$, we have 
	\begin{equation}\label{add_3.18}
		\mathbb{E}\Big[\mathcal{H}_z\big(\partial B(N),\mathcal{C}_{n,N}^{-}\big) \mid \mathcal{F}_{\mathcal{C}^-_{B(n)}} \Big]\le C(d)\Big[ \theta_d(N/4)+(n/N)^{d-2}\widecheck{\mathcal{H}}^{\mathrm{in}}\Big].  
	\end{equation}
\end{lemma}
\begin{proof}

   When $z\in \mathcal{C}^-_{B(n)}$, we have $\mathcal{H}_z\big(\partial B(N),\mathcal{C}_{n,N}^{-}\big)=0$ and thus (\ref{add_3.18}) holds.

   Next, we consider the case when $z\notin \mathcal{C}^-_{B(n)}$. Similar to (\ref{3.9}), we have
   \begin{equation}\label{add3.31}
 	\mathbb{E}\Big[\mathcal{H}_z\big(\mathcal{C}_{n,N}^{-}\big) \mid \mathcal{F}_{\mathcal{C}^-_{n,N}} \Big]  =	\mathbb{E}\Big[\widetilde{\phi}_z \cdot \mathbbm{1}_{\{z\xleftrightarrow{\le 0} \partial B(n)\}^c\cap \{z\xleftrightarrow{\le 0} \partial B(N)\}^c} \mid \mathcal{F}_{\mathcal{C}^-_{n,N}} \Big]. 
   \end{equation}
   By taking the integral on both sides of (\ref{add3.31}) (with respect to $\mathcal{C}^-_{B(N)}$ given $\mathcal{C}^-_{B(n)}$) and using $\mathcal{H}_z\big(\partial B(N),\mathcal{C}_{n,N}^{-}\big)\le \mathcal{H}_z\big(\mathcal{C}_{n,N}^{-}\big)$, we have 
 \begin{equation}\label{3.19}
 	\begin{split}
 		\mathbb{E}\Big[\mathcal{H}_z\big(\partial B(N),\mathcal{C}_{n,N}^{-}\big) \mid \mathcal{F}_{\mathcal{C}^-_{B(n)}} \Big] \le \mathbb{E}\Big[\widetilde{\phi}_z \cdot \mathbbm{1}_{\{z\xleftrightarrow{\le 0} \partial B(N)\}^c} \mid \mathcal{F}_{\mathcal{C}^-_{B(n)}} \Big].
 	\end{split}
 \end{equation}
    By Lemma \ref{lemma_strong_markov}, conditioning on $\mathcal{F}_{\mathcal{C}^-_{B(n)}}$, we have that $\{\widetilde{\phi}_v\}_{v\in \widetilde{B}(N)\setminus \mathcal{C}^-_{B(n)}}$ is distributed as $\{\widetilde{\phi}_v'+\mathcal{H}_v'\}_{v\in \widetilde{B}(N)\setminus \mathcal{C}^-_{B(n)}}$, where $\widetilde{\phi}_{\cdot}'\sim \mathbb{P}^{\mathcal{C}^-_{B(n)}}$ and $\mathcal{H}_v':=\mathcal{H}_v\big(\partial B(n),\mathcal{C}_{\partial B(n)}^{-}\big)$. Thus, noting that $\widetilde{\phi}_v\le 0$ is equivalent to $\widetilde{\phi}_v'\le -\mathcal{H}'_{v}$ under this coupling, the right-hand side of (\ref{3.19}) can be written as 
    \begin{equation}\label{3.20}
   	\begin{split}
   		&\mathbb{E}^{\mathcal{C}_{\partial B(n)}^{-}}\Big[\big(\widetilde{\phi}_z'+\mathcal{H}_z'\big) \cdot \mathbbm{1}_{\big\{z\xleftrightarrow{\le -\mathcal{H}'_{\cdot }} \partial B(N)\big\}^{c}}  \Big]\\
   		=&\mathbb{E}^{\mathcal{C}_{\partial B(n)}^{-}}\big(\widetilde{\phi}_z'\big)+\mathcal{H}_z' -	\mathbb{E}^{\mathcal{C}_{\partial B(n)}^{-}}\Big[\big(\widetilde{\phi}_z'+\mathcal{H}_z'\big) \cdot \mathbbm{1}_{z\xleftrightarrow{\le -\mathcal{H}'_{\cdot }} \partial B(N)}  \Big]\\
   	=&\mathcal{H}_z' +   \mathbb{E}^{\mathcal{C}_{\partial B(n)}^{-}}\Big[\big(\widetilde{\phi}_z'-\mathcal{H}_z'\big) \cdot \mathbbm{1}_{z\xleftrightarrow{\ge \mathcal{H}'_{\cdot }} \partial B(N)} \Big], 
   	\end{split}
   \end{equation}
   where $z \xleftrightarrow{\le -\mathcal{H}'_{\cdot }}\partial B(N)$ (resp. $z \xleftrightarrow{\ge \mathcal{H}'_{\cdot }} \partial B(N)$) means that there exists a path connecting $z$ and $\partial B(N)$ on which the values of $\widetilde{\phi}'_{\cdot}$ is at most $-\mathcal{H}'_{\cdot }$ (resp. at least $\mathcal{H}'_{\cdot }$). Note that we used the symmetry of $\widetilde{\phi}'_{\cdot}$ in the last equality of (\ref{3.20}). Moreover, since $\mathcal{H}'_{v}\ge 0$ for all $v\in  \widetilde{B}(N)\setminus \mathcal{C}^-_{B(n)}$, we have 
   \begin{equation}\label{new_add_3.21}
   	\begin{split}
   		&\mathbb{E}^{\mathcal{C}_{\partial B(n)}^{-}}\Big[\big(\widetilde{\phi}_z'-\mathcal{H}_z'\big) \cdot \mathbbm{1}_{z\xleftrightarrow{\ge \mathcal{H}'_{\cdot }} \partial B(N)} \Big] \\
   		\le & \mathbb{E}^{\mathcal{C}_{\partial B(n)}^{-}}\Big[\widetilde{\phi}_z' \cdot \mathbbm{1}_{z\xleftrightarrow{\ge 0} \partial B(N)} \Big] \\
   		\overset{(\ref{coro2.1_2})}{\le } &\mathbb{E}\Big[\widetilde{\phi}_z' \cdot \mathbbm{1}_{z\xleftrightarrow{\ge 0} \partial B(N)} \Big]
   		\overset{(\text{Lemma}\ \ref{lemma_crossing_prepare})}{\le }  C \mathbb{P}\big[z\xleftrightarrow{\ge 0}\partial B(N) \big]  \le C\theta_d(N/4),
   	\end{split}
   \end{equation}
   where in the last inequality we used $z\in \partial^{\mathrm{e}} \mathcal{B}(5N/8)$. By (\ref{3.19}), (\ref{3.20}) and (\ref{new_add_3.21}), 
   \begin{equation}\label{3.21}
   	\begin{split}
   			\mathbb{E}\Big[\mathcal{H}_z\big(\partial B(N),\mathcal{C}_{n,N}^{-}\big) \mid \mathcal{F}_{\mathcal{C}^-_{B(n)}} \Big]
   			\le  &\mathcal{H}_z'+ C \theta_d(N/4).
   	\end{split}
   \end{equation}
   Recall $\widecheck{\mathcal{H}}^{\mathrm{in}}$ in (\ref{new_add_3.17}). For $z\in\partial^{\mathrm{e}} \mathcal{B}(5N/8)$, by the strong Markov property we have  
   \begin{equation*}\label{3.22}
   	\begin{split}
   		\mathcal{H}_z'\le& \sum\nolimits_{y\in \partial \mathcal{B}(4dn)} \widetilde{\mathbb{P}}_z\big(\tau_{\partial \mathcal{B}(4dn)}=\tau_{y}<\infty\big) \mathcal{H}_y'\\
   		\overset{(\text{Lemma}\ \ref{lemma_uniform_distribution})}{\le } &Cn^{-1}N^{2-d}  \sum\nolimits_{y\in \partial \mathcal{B}(4dn)}  \mathcal{H}_y'  \overset{(|\partial \mathcal{B}(4dn)|\asymp n^{d-1})}{\le }C'(n/N)^{d-2}  \widecheck{\mathcal{H}}^{\mathrm{in}}.
   	\end{split}
   \end{equation*}
   Combined with (\ref{3.21}), it concludes this lemma. 
   \end{proof}

Recall $\overbar{\mathcal{H}}^{\mathrm{out}}$ in (\ref{def_H_out}). Then it follows from Lemma \ref{lemma_3.6} that 
 \begin{equation}\label{bound_H_out}
 	\overbar{\mathcal{H}}^{\mathrm{out}} \le C\Big[ \theta_d(N/4)+(n/N)^{d-2}\widecheck{\mathcal{H}}^{\mathrm{in}}\Big]. 
 \end{equation}
With these preparations, we are ready to prove Proposition \ref{lemma_bound_crossing}.

\begin{proof}[Proof of Proposition \ref{lemma_bound_crossing}]
Recall that \cite[Theorem 5]{ding2020percolation} shows 
\begin{equation}\label{ineq_thm5}
\theta_d(k)\ge ck^{-\frac{d}{2}+1}, \ \ \forall k\ge 1. 
\end{equation}
By (\ref{ineq_thm5}), we have $n^{d-2} \theta_d(n)\theta_d(N/4) \ge c(n/N)^{\frac{d}{2}-1}$. Thus, it suffices to prove this lemma in the case when $N/n$ is sufficiently large (otherwise, we can take a large enough $\Cref{const_crossing}$ such that the right-hand side of (\ref{ineq_lemma_crossing}) exceeds $1$).


	By Lemmas \ref{lemma_connecting_hm} and \ref{lemma_hat_Hy}, and the inequality that $1-e^{-a}\le a\land 1$ for all $a>0$, 
	\begin{equation}\label{3.18}
		\begin{split}
		\rho_d(n,N)= \mathbb{E}\Big[1- e^{-2\sum\nolimits_{y\in \partial B(n)} \widehat{\mathcal{H}}_{y}(\partial B(N),\mathcal{C}_{n,N}^{-})\widetilde{\phi}_{y} } \Big]
		\le  C \mathbb{E}\Big[\big(n^{d-2}\overbar{\mathcal{H}}^{\mathrm{out}}\overbar{\mathcal{H}}^{\mathrm{in}}\big)\land 1 \Big]. 
		\end{split}	
	\end{equation}
		Recall that $\overbar{\mathcal{H}}^{\mathrm{in}}\le \widecheck{\mathcal{H}}^{\mathrm{in}}$ below (\ref{new_add_3.17}). Also note that $\widecheck{\mathcal{H}}^{\mathrm{in}}$ is measurable with respect to $\mathcal{F}_{\mathcal{C}^-_{B(n)}}$. Thus, by Jensen's inequality (i.e. $\mathbb{E}(Y\land 1)\le \mathbb{E}(Y)\land 1$) and (\ref{bound_H_out}),
	\begin{equation}\label{3.25}
		\begin{split}
	&	\mathbb{E}\Big[\big(n^{d-2}\overbar{\mathcal{H}}^{\mathrm{out}}\overbar{\mathcal{H}}^{\mathrm{in}}\big)\land 1 \Big]\\
\overset{(\overbar{\mathcal{H}}^{\mathrm{in}}\le \widecheck{\mathcal{H}}^{\mathrm{in}})}{\le } &\mathbb{E}\Big\{\mathbb{E}\Big[\big(n^{d-2}\overbar{\mathcal{H}}^{\mathrm{out}}\widecheck{\mathcal{H}}^{\mathrm{in}}\big)\land 1 \mid \mathcal{F}_{\mathcal{C}^-_{B(n)}} \Big]\Big\}  \\
	\overset{(\text{Jensen's ineq})}{\le } &      \mathbb{E}\Big\{\Big(\mathbb{E}\big[\overbar{\mathcal{H}}^{\mathrm{out}} \mid \mathcal{F}_{\mathcal{C}^-_{B(n)}} \big]n^{d-2}\widecheck{\mathcal{H}}^{\mathrm{in}}\Big)\land 1\Big\}  \\
		\overset{(\ref{bound_H_out})}{\le } &C \mathbb{E}\Big\{\Big(\big[\theta_d(N/4)+ (n/N)^{d-2}\widecheck{\mathcal{H}}^{\mathrm{in}}\big]n^{d-2}\widecheck{\mathcal{H}}^{\mathrm{in}}\Big) \land 1\Big\}.
		\end{split}
	\end{equation}
	Note that for any $a,a',b>0$, one has 
	\begin{equation}\label{ineq_a_a'_b}
		[(a+a')b]\land 1\le ab+ [(a'b)\land 1]\le ab+ \sqrt{a'b}. 
	\end{equation}
	Applying (\ref{ineq_a_a'_b}), we can bound the right-hand side of (\ref{3.25}) from above by 
\begin{equation}\label{new_add_3.26}
		\begin{split}
	Cn^{d-2} \Big[ \theta_d(N/4)\mathbb{E}\big(\widecheck{\mathcal{H}}^{\mathrm{in}}\big)  + N^{-\frac{d}{2}+1} \mathbb{E}\big(\widecheck{\mathcal{H}}^{\mathrm{in}}\big) \Big]  \overset{(\ref{ineq_thm5})}{\le}  C'n^{d-2} \theta_d(N/4)\mathbb{E}\big(\widecheck{\mathcal{H}}^{\mathrm{in}}\big). 
		\end{split}
	\end{equation}
	It follows from (\ref{ineq_coro_lemma3.4}), (\ref{3.25}) and (\ref{new_add_3.26}) that 
	\begin{equation*}
		\begin{split}	\mathbb{E}\Big[\big(n^{d-2}\overbar{\mathcal{H}}^{\mathrm{out}}\overbar{\mathcal{H}}^{\mathrm{in}}\big)\land 1 \Big]\overset{(\ref{3.25}),(\ref{new_add_3.26})}{\le } Cn^{d-2}\theta_d(N/4) \mathbb{E}\big(\widecheck{\mathcal{H}}^{\mathrm{in}}\big)\overset{(\ref{ineq_coro_lemma3.4})}{\le}C'n^{d-2}\theta(n)\theta_d(N/4).
		\end{split}
	\end{equation*}
  Combined with (\ref{3.18}), it completes the proof of this proposition. 
\end{proof}


\section{Proof of Theorem \ref{thm1}}\label{section_proof_thm1}


In this section, we present the proof of Theorem \ref{thm1}. Before showing the proof details, we first provide an overview as follows. As explained in Section \ref{section_1.2.1}, our proof strategy is based on proof by contradiction. As described in Definitions \ref{def_es} and \ref{def_af} below, we take a sufficiently large function $\lambda$ and consider the first scale $N_*$ such that $\theta_d(N_*)>\lambda(N_*)N_*^{-\frac{d}{2}+1}$. Assuming that $N_*<\infty$, we establish in Lemma \ref{lemma_k_diamond} that there exists a scale $k_*\ge 1$ such that for any $v\in \mathbb{Z}^d$, with probability at least $2^{-k_*}$ the harmonic average $\mathcal{H}^*_v$ (see (\ref{4.9})) satisfies $\mathcal{H}^*_v\ge \lambda(N_*)N_*^{-\frac{d}{2}+1}k_*^{-C}2^{k_*}$. Based on this property, we prove in Propsition \ref{prop_1} that with probability at least $2^{-k_*^{40d}}$, the average of harmonic averages on $\partial^{\mathrm{e}} \mathcal{B}(\frac{3N_*}{16})$, denoted by $\overbar{\mathcal{H}}_{*}$ (see (\ref{4.23})), is at least $e^{k_*^{5}}\lambda_*N_*^{-\frac{d}{2}+1}$. The proof of Proposition \ref{prop_1}, which is the core of this paper, will be implemented in Section \ref{section_block}. Meanwhile, by analyzing the conditional distribution of the average of $\widetilde{\phi}_\cdot$ on $\partial^{\mathrm{e}} \mathcal{B}(\frac{3N_*}{16})$, in Proposition \ref{prop_2} we show that the probability for $\overbar{\mathcal{H}}_{*}$ to reach $e^{k_*^{5}}\lambda_*N_*^{-\frac{d}{2}+1}$ is at most $Ce^{-c[\lambda(N_*)]^2e^{2k_*^5}}$, which then causes a contradiction with Proposition \ref{prop_1} and thus concludes Theorem \ref{thm1}.

\subsection{Exceeding size}\label{section4.1}


To upper-bound $\theta_d(\cdot)$, we introduce the concept of \textit{exceeding size}, which is the first scale such that $\theta_d$ exceeds a predetermined threshold.

We take a sufficiently small $\cl\label{const_small}(d)\in (0,(10d)^{-100d})$ such that for $k\ge \log_2(1/\cref{const_small})$, 
	\begin{equation}\label{require_constant1}
	(2\pi)^{-\frac{1}{2}}[G(\bm{0},\bm{0})]^{\frac{1}{2}}k^{-1} e^{-\frac{k^2}{2G(\bm{0},\bm{0})}} < 2^{-k},
	\end{equation}
		\begin{equation}\label{require_constant2.3}
		\tfrac{1}{2}e^{k^5} > \Cref{const_G_2}   \ \ (\text{defined below Equation}\ (\ref{new_4.34})),
	\end{equation}
		\begin{equation}\label{require_constant1.1}
	     k^{-1}\le \min\{c_\dagger^{(1)},c_\dagger^{(2)},c_\dagger^{(3)}\}  \ \ (\text{defined in the proof of Lemma}\ \ref{lemma_excellent_box}),
	\end{equation}
	\begin{equation}\label{require_constant1_section5}
		e^{k^{5.5}-k^{5.4}}\ge \Cref{const_lemma_6.3.7} \ \ (\text{defined in Inequality}\ (\ref{add5.74})).
	\end{equation}
	We also take a sufficiently large constant $\Cl\label{const_section_block}(d)>0$ satisfying that 
	\begin{equation}\label{equation_def_C9}
		\Cref{const_section_block}\ge \mathrm{exp}\big((\cref{const_average_2}^{-1}\cref{const_small}^{-1}\cref{const_hit}^{-1}\Cref{const_green_1}\Cref{const_crossing}\Cref{const_ldp})^{100d}e^{10^6d^2}\big)
	\end{equation}
	(where the constants $\cref{const_hit}$ and $\Cref{const_ldp}$ are defined in Lemma \ref{lemma_green} and (\ref{ineq_ldp}) respectively) and that for any $N\ge \Cref{const_section_block}^{1/d}$,  
		\begin{equation}\label{require_constant2}
		\theta_d(\tfrac{N}{2d})< (10d)^{-1}  \ \ (\text{recall that}\ \lim\nolimits_{N\to \infty}\theta_d(N)= 0),
	\end{equation}
	\begin{equation}\label{require_constant2.2}
		\Cref{const_k_prepare}e^{-\cref{const_k_prepare_1}\ln^{2}(N)}<\tfrac{1}{8} N^{-\frac{d}{2}+1}\ \ (\Cref{const_k_prepare}\ \text{and}\ \cref{const_k_prepare_1}\ \text{are defined in Lemma}\ \ref{lemma_k_diamond_prepare}). 
	\end{equation}


%


	

%
%
%

\begin{definition}[exceeding size]\label{def_es}
	For any $d\ge 3$ and any non-decreasing function $\lambda: (0,\infty)\to (1,\infty)$, we define 
	\begin{equation}
		N_*=N_*(d,\lambda):= \min\big\{\text{integer}\ N\ge \Cref{const_section_block}^{1/d}: \theta_d(N)> \lambda(N)N^{-\frac{d}{2}+1}\big\},
	\end{equation}
	where we set $\min \emptyset =+\infty$ for completeness. 
	
\end{definition}

We also consider functions that have a finite exceeding size.

\begin{definition}[admissible function]\label{def_af}
	For any $d\ge 3$, let $\mathfrak{X}_d$ be the collection of non-decreasing functions $\lambda: (0,\infty)\to (1,\infty)$ such that $N_*(d,\lambda)<\infty$. Moreover, for any $\lambda\in \mathfrak{X}_d$, we denote $\lambda_*:=\lambda(N_*)$.
\end{definition}


\begin{remark}[A sufficient condition for Theorem \ref{thm1}]\label{remark_sufficient}
For any non-decreasing function $\lambda(\cdot)$ with $\lambda\ge \Cref{const_section_block}$, since $\theta_{d}(N)\le 1< \lambda N^{-\frac{d}{2}+1}$ holds for all $N\le \Cref{const_section_block}^{1/d}$, it follows from Definitions \ref{def_es} and \ref{def_af} that 
\begin{equation}\label{add4.10}
	\theta_{d}(N)\le \lambda(N) N^{-\frac{d}{2}+1}, \ \ \forall 1\le N<N_*, 
\end{equation}
where $N_*$ may be infinity. Thus, to establish Theorem \ref{thm1}, it suffices to prove
\begin{enumerate}
	\item For any $d\in \{3,4,5\}$, $\mathfrak{X}_d$ does not contain the constant function $\lambda(N)=\Cref{const_section_block}$.

	\item For $d=6$, $\mathfrak{X}_6$ does not contain the function $\lambda(N)= \Cref{const_section_block}\mathrm{exp}\big(\ln^{\frac{1}{2}}(N)\ln\ln(N)\big)$ (for convenience, we set $\ln(a):=0$ for $a<1$ throughout this paper). 
	
\end{enumerate}





\end{remark}

\begin{remark}\label{remark_lambda}
	In the case when $d=6$, the reason why we choose the function $\lambda(N)=\Cref{const_section_block}\mathrm{exp}\big(\ln^{\frac{1}{2}}(N)\ln\ln(N)\big)$($=\Cref{const_section_block}\mathrm{exp}\big(\ln^{\frac{1}{2}+\frac{\ln\ln\ln(N )}{\ln\ln(N)}}(N)\big)$) in Proposition \ref{prop_1} is that $\lambda(N)$ satisfies the following inequality. Let $K_1=\frac{1}{d}$ and $K_2=2000d$. For any $N\ge \Cref{const_section_block}^{1/d}$ and $M\in [1,N^d]$ (where $M$ may depend on $N$), one has 	
	\begin{equation}\label{4.29}
		\frac{\lambda(N)}{\lambda\big(N[\lambda(N)M]^{-K_1}\big)} \ge  \ln^{K_2}\big(\lambda(N)M\big). 
	\end{equation}
	Next, we prove (\ref{4.29}) separately in two cases. When $N[\lambda(N)M]^{-K_1}\le 100$, to get (\ref{4.29}), it is sufficient to have 
	\begin{equation*}
	\begin{split}
			&\lambda(N)\ge \lambda(100) \ln^{K_2}\big(\lambda(N)M\big)\\
		\overset{(\lambda(N)\le \Cref{const_section_block}N\le N^2)}{\Leftarrow}	 &\mathrm{exp}\big(\ln^{\frac{1}{2}}(N)\ln\ln(N)\big)\ge \mathrm{exp}\big(\ln^{\frac{1}{2}}(100)\ln\ln(100)\big) \ln^{K_2}(N^{d+2}), 
	\end{split}
	\end{equation*}
		which holds true for all $N\ge \Cref{const_section_block}^{1/d} > \mathrm{exp}(e^{10^6d^2})$. When $N[\lambda(N)M]^{-K_1}\ge 100$, let $\xi(N):=\frac{\ln\ln\ln(N )}{\ln\ln(N)}$ and $t_N:=\ln(N)$. For (\ref{4.29}), it suffices to have that 
	\begin{equation*}
		\begin{split}
			&t_N^{\frac{1}{2}+\xi(N) }-\big[t_N-K_1\big(t_N^{\frac{1}{2}+\xi(N)}+\ln(\Cref{const_section_block}M)\big)\big]^{\frac{1}{2}+\xi(N)}\ge  K_2\ln\big(t_N^{\frac{1}{2}+\xi(N)}+\ln(\Cref{const_section_block}M)\big) \\
			\Leftarrow\ &t_N^{\frac{1}{2}+\xi(N) } \Big\{1- \Big[1- K_1\Big(t_N^{-\frac{1}{2}+\xi(N)}+t_N^{-1}\ln(\Cref{const_section_block}M)\Big)\Big]^{\frac{1}{2}+\xi(N)}\Big\} \ge  K_2\ln\big(t_N^{\frac{1}{2}+\xi(N)}+\ln(\Cref{const_section_block}M)\big)\\
			\Leftarrow\ & \tfrac{1}{4}K_1\Big[t_N^{2\xi(N) }+ t_N^{-1}\ln(\Cref{const_section_block}M)\Big]    \ge  K_2\big[2\ln\ln(N) +t_N^{-2}\ln(\Cref{const_section_block}M)\big]\\
			\Leftarrow\ &  t_N^{2\xi(N)}  \ge  8K_1^{-1}K_2\ln\ln(N)\ \text{and}\ \ln(N)\ge 4K_1^{-1}K_2 \\
			\Leftarrow\  &  \xi(N)\ge \tfrac{\ln\ln\ln(N)+\ln(8K_1^{-1}K_2)}{2\ln\ln(N)}\ \text{and}\ \ln(N)\ge 4K_1^{-1}K_2\\
			\Leftarrow\  & N \ge \max\big\{ \mathrm{exp}\big(e^{8K_1^{-1}K_2}\big),e^{ 4K_1^{-1}K_2}\big\},
		\end{split}
	\end{equation*} 
	which also holds for all $N\ge \Cref{const_section_block}^{1/d} > \mathrm{exp}(e^{10^6d^2})$. In conclusion, we establish (\ref{4.29}).
\end{remark}

\subsection{Effective scale}

The key idea for the proof of Theorem \ref{thm1} relies on a coarse-graining method, which shows that with significant probability, the negative cluster $\mathcal{C}^-_{\partial^{\mathrm{e}} \mathcal{B}(N_*/4)}$ is so dense that an independent Brownian motion starting from $\bm{0}$ will hit $\mathcal{C}^-_{\partial^{\mathrm{e}} \mathcal{B}(N_*/4)}$ before reaching $\partial \mathcal{B}(0.01N_*)$ with high probability (see Lemma \ref{lemma_F}).  As a notable difference from most of the classical coarse-graining methods used in the study of percolation models, in our proof the involved scale is not predetermined, but chosen in a rather implicit manner depending on certain intrinsic properties of the model. Specifically, we present the following lemma to ensure the existence of such an effective scale (that is, a scale $k_*$ such that (\ref{ineq_lemma_k_diamond}) holds).


For any $\lambda\in \mathfrak{X}_d$ and $v\in \mathbb{Z}^d$, we define (recalling $\mathcal{H}_v$ in Definition \ref{def_harmonic_average})
\begin{equation}\label{4.9}
	\mathcal{H}^*_v= \mathcal{H}^*_v(d,\lambda):= \mathcal{H}_v\big( \partial^{\mathrm{e}} \mathcal{B}_v(N_*/2) , \mathcal{C}^{-}_{\partial^{\mathrm{e}} \mathcal{B}_v(N_*/2)}\big). 
\end{equation}

\begin{lemma}\label{lemma_k_diamond}
	For $d\ge 3$ and $\lambda\in \mathfrak{X}_d$, there exists an integer $k_*=k_*(d,\lambda)>0$ with 
	\begin{equation}\label{range_k_*}
			\lambda_* N_*^{-\frac{d}{2}+1}\ln^{-10}(N_*) \le 2^{-k_*}\le \cref{const_small}
	\end{equation}
	such that for any $v\in \mathbb{Z}^d$,  
	\begin{equation}\label{ineq_lemma_k_diamond}
		\mathbb{P}\big(\mathcal{H}^*_v\ge \cref{const_small}^2\lambda_* N_*^{-\frac{d}{2}+1} k_*^{-3}2^{k_*}\big)\ge 2^{-k_*}. 
	\end{equation}
\end{lemma}

Note that the upper bound in (\ref{range_k_*}) implies that $k_*\ge \log_2(1/\cref{const_small})$. Before proving Lemma \ref{lemma_k_diamond}, we need some estimates for $\widetilde{\phi}_{\bm{0}} \mathcal{H}^*_{\bm{0}}$ as follows. 
\begin{lemma}\label{lemma_k_diamond_prepare}
	For any $d\ge 3$, there exist $\Cl\label{const_k_prepare}(d),\cl\label{const_k_prepare_1}(d)>0$ such that for all $\lambda\in \mathfrak{X}_d$, 
	\begin{equation}\label{ineq_lemma_k_diamond_prepare}
		\mathbb{P}\big[\widetilde{\phi}_{\bm{0}} \mathcal{H}^*_{\bm{0}} \ge \ln^{5}(N_*)\big] \le  \Cref{const_k_prepare} e^{-\cref{const_k_prepare_1}\ln^{2}(N_*)}. 
	\end{equation}
\end{lemma}
\begin{proof}
	Since $\mathcal{H}^*_{\bm{0}}\ge 0$, we have 
	\begin{equation}\label{ineq_lemma_k_diamond_prepare_1}
		\mathbb{P}\big[\widetilde{\phi}_{\bm{0}} \mathcal{H}^*_{\bm{0}} \ge \ln^{5}(N_*)\big]\le \mathbb{P}\big[\widetilde{\phi}_{\bm{0}}  \ge \ln^{2}(N_*)\big]+\mathbb{P}\big[ \mathcal{H}^*_{\bm{0}} \ge \ln^{3}(N_*)\big]. 
	\end{equation}
	For the first term on the right-hand side of (\ref{ineq_lemma_k_diamond_prepare_1}), by Lemma \ref{lemma_tail} one has  
	\begin{equation}\label{ineq_phi0_N_0.01}
		\mathbb{P}\big[\widetilde{\phi}_{\bm{0}}  \ge \ln^2(N_*)\big] \le (2\pi)^{-\frac{1}{2}}[G(\bm{0},\bm{0})]^{\frac{1}{2}}\ln^{-2}(N_*)e^{-\frac{\ln^4(N_*)}{2G(\bm{0},\bm{0})}}.
	\end{equation}
	For the second term, since $\mathcal{H}^*_{\bm{0}}\le \max_{y\in \partial^{\mathrm{e}} \mathcal{B}(N_*/2)} \widetilde{\phi}_y$, we have 
	\begin{equation}
		\mathbb{P}\big[ \mathcal{H}^*_{\bm{0}} \ge \ln^{3}(N_*)\big] \le \mathbb{P}\big[\max\nolimits_{y\in \partial^{\mathrm{e}} \mathcal{B}(N_*/2)} \widetilde{\phi}_y \ge \ln^{3}(N_*)\big]. 
	\end{equation}
	In addition, it is known that $\max_{y\in \partial^{\mathrm{e}} \mathcal{B}(N_*/2)} \widetilde{\phi}_y $ concentrates in the order of $\ln(N_*)$. Specifically, by \cite[Theorem1.1]{chiarini2016extremes} we have
	\begin{equation}\label{ineq_lemma_k_diamond_prepare_4}
		\mathbb{P}\big[ \max\nolimits_{y\in \partial^{\mathrm{e}} \mathcal{B}(N_*/2)} \widetilde{\phi}_y \ge \ln^{3}(N_*)\big] \le Ce^{-c\ln^2(N_*)}. 
	\end{equation}
	Combining (\ref{ineq_lemma_k_diamond_prepare_1})-(\ref{ineq_lemma_k_diamond_prepare_4}), we obtain the desired bound (\ref{ineq_lemma_k_diamond_prepare}).  
\end{proof}


\begin{lemma}\label{lemma_4.5}
	For any $d\ge 3$ and $\lambda\in \mathfrak{X}_d$, we have 
	\begin{equation}\label{4.14}
	\mathbb{E}\Big[\big(1-e^{-2\widetilde{\phi}_{\bm{0}}\mathcal{H}_{\bm{0}}^{*}} \big) \cdot \mathbbm{1}_{\widetilde{\phi}_{\bm{0}}\ge 0}\Big] \ge \tfrac{1}{2} \theta_d(N_*). 
	\end{equation}
\end{lemma}
\begin{proof}
	Recall that $\widetilde{\partial} \widetilde{B}(1)=\{y\in \mathbb{Z}^d:y\sim \bm{0}\}$. By the symmetry of $\widetilde{\phi}$ and (\ref{require_constant2}),
	\begin{equation}\label{ineq_c6}
	\mathbb{P}\big[\widetilde{B}(1)\xleftrightarrow{\le 0 } \partial^{\mathrm{e}} \mathcal{B}(N_*/2) \big]\le 2d  \theta_d(\tfrac{N_*}{2d}) <\tfrac{1}{2}. 
	\end{equation}
Thus, by the FKG inequality and Lemma \ref{lemma_connecting_hm} (recalling $\widehat{\mathcal{H}}_\cdot$ in (\ref{def_hat_H})), 
	\begin{equation}\label{4.15}
\begin{split}
	\tfrac{1}{2} \theta_d(N_*) \overset{(\mathrm{FKG})}{\le } &\mathbb{P}\Big[ \bm{0}\xleftrightarrow{\ge 0}\partial^{\mathrm{e}} \mathcal{B}(N_*/2), \big\{\widetilde{B}(1)\xleftrightarrow{\le 0 } \partial^{\mathrm{e}} \mathcal{B}(N_*/2)\big\}^c\Big]\\
	\overset{(\mathrm{Lemma}\ \ref{lemma_connecting_hm})}{=}& \mathbb{E}\Big\{\big[1-e^{-2\widetilde{\phi}_{\bm{0}}\widehat{\mathcal{H}}_{\bm{0}}(\partial^{\mathrm{e}} \mathcal{B}(\tfrac{N_*}{2}), \mathcal{C}^-_{\{\bm{0}\}\cup \partial^{\mathrm{e}} \mathcal{B}(\frac{N_*}{2})})} \big]  \mathbbm{1}_{\widetilde{\phi}_{\bm{0}}\ge 0, \mathcal{C}^-_{\partial^{\mathrm{e}} \mathcal{B}(\frac{N_*}{2})}\cap \widetilde{B}(1)=\emptyset}\Big\}.
\end{split}	
	\end{equation}
	On the event $\big\{\widetilde{\phi}_{\bm{0}}\ge 0, \mathcal{C}^-_{\partial^{\mathrm{e}} \mathcal{B}(N_*/2)}\cap \widetilde{B}(1)=\emptyset \big\}$, one has 
	\begin{equation}\label{4.16}
		\begin{split}
				&\widehat{\mathcal{H}}_{\bm{0}}\big(\partial^{\mathrm{e}} \mathcal{B}(N_*/2), \mathcal{C}^-_{\{\bm{0}\}\cup \partial^{\mathrm{e}} \mathcal{B}(N_*/2)}\big)\\
				 \overset{(\ref{def_hat_H})}{=}& (2d)^{-1} \sum\nolimits_{z\sim \bm{0}} \mathcal{H}_{z}\big(\partial^{\mathrm{e}} \mathcal{B}(N_*/2), \mathcal{C}^-_{\{\bm{0}\}\cup \partial^{\mathrm{e}} \mathcal{B}(N_*/2)}\big)\\
				\le &(2d)^{-1} \sum\nolimits_{z\sim \bm{0}} \mathcal{H}_{z}\big(\partial^{\mathrm{e}} \mathcal{B}(N_*/2), \mathcal{C}^-_{\partial^{\mathrm{e}} \mathcal{B}(N_*/2)}\big) = \mathcal{H}_{\bm{0}}^{*}. 
		\end{split}
	\end{equation}
	Combining (\ref{4.15}) and (\ref{4.16}), we obtain (\ref{4.14}). 
\end{proof}

Now we are ready to prove Lemma \ref{lemma_k_diamond}. 

\begin{proof}[Proof of Lemma \ref{lemma_k_diamond}]
	Without loss of generality, we take $v=\bm{0}$.

	We claim that there exists $\epsilon \in \big(2\lambda_* N_*^{-\frac{d}{2}+1}\ln^{-10}(N_*) ,\cref{const_small}\big)$ such that 
	\begin{equation}\label{ineq_claim_phi_h}
		\mathbb{P}\big[\widetilde{\phi}_{\bm{0}} \mathcal{H}^*_{\bm{0}} \ge 2\cref{const_small}^2\lambda_* N_*^{-\frac{d}{2}+1} \epsilon^{-1}\log^{-2}_2(1/\epsilon) \big] \ge 2\epsilon. 
	\end{equation}
	We first prove the lemma assuming (\ref{ineq_claim_phi_h}). By Lemma \ref{lemma_tail} we have
	\begin{equation}\label{lower_bound_phi_0}
		\mathbb{P}\big[ \widetilde{\phi}_{\bm{0}}\ge\log_2(1/\epsilon)\big] \le (2\pi)^{-\frac{1}{2}}[G(\bm{0},\bm{0})]^{\frac{1}{2}}\log_2^{-1}(1/\epsilon) e^{-\frac{\log_2^2(1/\epsilon)}{2G(\bm{0},\bm{0})}} \overset{(\ref{require_constant1})}{\le} \epsilon. 
	\end{equation}
	By $\mathcal{H}^*_{\bm{0}}\ge 0$, (\ref{ineq_claim_phi_h}) and (\ref{lower_bound_phi_0}), we get
	\begin{equation*}
		\begin{split}
			&\mathbb{P}\big[ \mathcal{H}^*_{\bm{0}} \ge 2\cref{const_small}^2\lambda_* N_*^{-\frac{d}{2}+1} \epsilon^{-1}\log^{-3}_2(1/\epsilon) \big] \\
			\overset{(\mathcal{H}^*_{\bm{0}}\ge 0)}{\ge }& \mathbb{P}\big[ \widetilde{\phi}_{\bm{0}}\mathcal{H}^*_{\bm{0}} \ge  2\cref{const_small}^2\lambda_* N_*^{-\frac{d}{2}+1} \epsilon^{-1}\log^{-2}_2(1/\epsilon) \big]- \mathbb{P}\big[ \widetilde{\phi}_{\bm{0}}\ge\log_2(1/\epsilon)\big] \overset{(\ref{ineq_claim_phi_h}),(\ref{lower_bound_phi_0})}{\ge } \epsilon.
		\end{split}
	\end{equation*}
	Thus, by taking $k_*:=\lceil \log_2(1/\epsilon) \rceil$, we conclude this lemma.

	In what follows, we prove (\ref{ineq_claim_phi_h}) by contradiction. To this end, we suppose that (\ref{ineq_claim_phi_h}) does not hold. I.e., we suppose that for any $\epsilon \in \big(2\lambda_* N_*^{-\frac{d}{2}+1}\ln^{-10}(N_*) ,\cref{const_small}\big)$, 
	\begin{equation}\label{ineq_claim_false}
		\mathbb{P}\big[\widetilde{\phi}_{\bm{0}} \mathcal{H}^*_{\bm{0}} \ge f(\epsilon) \big]
		:=	\mathbb{P}\big[\widetilde{\phi}_{\bm{0}} \mathcal{H}^*_{\bm{0}} \ge 2\cref{const_small}^2\lambda_* N_*^{-\frac{d}{2}+1} \epsilon^{-1}\log^{-2}_2(1/\epsilon) \big]< 2\epsilon.
	\end{equation}
	Note that $f(\cdot)$ is strictly decreasing on $(0,e^{-2})\supset (0,\cref{const_small})$, and satisfies that 
	$$
	f(\cref{const_small})\overset{(\cref{const_small}\le 0.01)}{\le }  \tfrac{1}{8}\lambda_* N_*^{-\frac{d}{2}+1}\ \ \text{and}\ \ f\big(2\lambda_* N_*^{-\frac{d}{2}+1}\ln^{-10}(N_*) \big)\overset{(\ref{equation_def_C9})}{\ge}  2\ln^{5}(N_*). 
	$$
	As a result, for any $t\in \big[\frac{1}{8}\lambda_* N_*^{-\frac{d}{2}+1},2\ln^{5}(N_*) \big]$, there exists a unique number $\nu(t)$ in $\big(2\lambda_* N_*^{-\frac{d}{2}+0.9},\cref{const_small}\big)$ such that $f(\nu(t))=t$. Therefore, since $\nu(t)\le \cref{const_small}<\frac{1}{2}$, we have 
	$	\nu(t)=2\cref{const_small}^2\lambda_* N_*^{-\frac{d}{2}+1} t^{-1}\log^{-2}_2(1/\nu(t)) \le 2\cref{const_small}^2\lambda_* N_*^{-\frac{d}{2}+1} t^{-1}$, which implies that
	\begin{equation}\label{newadd_4.23}
		\nu(t)\le 2\cref{const_small}^2\lambda_* N_*^{-\frac{d}{2}+1} t^{-1}\log^{-2}_2(\tfrac{1}{2}\cref{const_small}^{-2}\lambda_*^{-1} N_*^{\frac{d}{2}-1} t). 
	\end{equation}
	Let $k>0$ be the integer such that $t\in [2^{k-3}\lambda_* N_*^{-\frac{d}{2}+1},2^{k-2}\lambda_* N_*^{-\frac{d}{2}+1}]$. Hence, we have $2\cref{const_small}^2\lambda_* N_*^{-\frac{d}{2}+1} t^{-1} \le  2^{-k-2} $ (by $\cref{const_small}<0.01$) and thus $\log_2(\tfrac{1}{2}\cref{const_small}^{-2}\lambda_*^{-1} N_*^{\frac{d}{2}-1} t)\ge k+2$. Combined with (\ref{newadd_4.23}), it yields that $\nu(t) \le (k+2)^{-2}2^{-k-2}$ and thus, 
	\begin{equation}\label{ineq_nu_t}
	\mathbb{P}\big(\widetilde{\phi}_{\bm{0}} \mathcal{H}_{\bm{0}}^* \ge t \big) \overset{(\ref{ineq_claim_false})}{<} 2	\nu(t) \le   (k+2)^{-2}2^{-k-1}.
	\end{equation}
	As a result, $\mathbb{E}\Big(\widetilde{\phi}_{\bm{0}} \mathcal{H}^*_{\bm{0}}\cdot \mathbbm{1}_{\widetilde{\phi}_{\bm{0}} \mathcal{H}^*_{\bm{0}}\in \big[\frac{1}{8}\lambda_* N_*^{-\frac{d}{2}+1},\ln^{5}(N_*)\big]}\Big)$ is bounded from above by 
	\begin{equation}\label{newadd_4.25}
		\begin{split}
			&\int_{\frac{1}{8}\lambda_* N_*^{-\frac{d}{2}+1} }^{\ln^{5}(N_*)} \mathbb{P}\big(\widetilde{\phi}_{\bm{0}} \mathcal{H}_{\bm{0}}^* \ge t \big)dt\\
			\le &\sum_{k\ge 0:2^{k-3}\lambda_* N_*^{-\frac{d}{2}+1}\le \ln^{5}(N_*) }\int_{2^{k-3}\lambda_* N_*^{-\frac{d}{2}+1}}^{2^{k-2}\lambda_* N_*^{-\frac{d}{2}+1}} \mathbb{P}\big(\widetilde{\phi}_{\bm{0}} \mathcal{H}_{\bm{0}}^* \ge t \big)dt  \\
			\overset{(\ref{ineq_nu_t})}{\le }  & \sum\nolimits_{k\ge 0} 2^{k-3}\lambda_* N_*^{-\frac{d}{2}+1} \cdot    (k+2)^{-2}2^{-k-1} < \tfrac{1}{8}\lambda_* N_*^{-\frac{d}{2}+1}.  
		\end{split}
	\end{equation}
	By $\mathcal{H}^*_{\bm{0}}\ge 0$, (\ref{newadd_4.25}), Lemma \ref{lemma_k_diamond_prepare} and (\ref{require_constant2.2}), we obtain
	\begin{equation*}\label{4.21}
		\begin{split}
			&\mathbb{E}\Big[\big(1-e^{-2\widetilde{\phi}_{\bm{0}} \mathcal{H}^*_{\bm{0}}}\big)\cdot \mathbbm{1}_{\widetilde{\phi}_{\bm{0}}\ge 0}\Big]\\
			\overset{(\mathcal{H}^*_{\bm{0}}\ge 0)}{\le } &\tfrac{1}{4}\lambda_* N_*^{-\frac{d}{2}+1} +  \mathbb{E}\Big(\widetilde{\phi}_{\bm{0}} \mathcal{H}^*_{\bm{0}}\cdot \mathbbm{1}_{\widetilde{\phi}_{\bm{0}} \mathcal{H}^*_{\bm{0}}\in [\frac{1}{8}\lambda_* N_*^{-\frac{d}{2}+1},\ln^{5}(N_*)]}\Big)+\mathbb{P}\big[\widetilde{\phi}_{\bm{0}} \mathcal{H}^*_{\bm{0}} \ge \ln^{5}(N_*)\big]\\
			\overset{(\ref{newadd_4.25}),\text{Lemma}\ \ref{lemma_k_diamond_prepare}}{\le} & \tfrac{3}{8}\lambda_* N_*^{-\frac{d}{2}+1}   +\Cref{const_k_prepare} e^{-\cref{const_k_prepare_1}\ln^{2}(N_*)}  \overset{(\ref{require_constant2.2})}{<} \tfrac{1}{2}\lambda_* N_*^{-\frac{d}{2}+1},                
		\end{split}
	\end{equation*}
	which is contradictory with Lemma \ref{lemma_4.5} since $\theta_d(N_*)\ge \lambda_*N_*^{-\frac{d}{2}+1}$. This completes the proof of (\ref{ineq_claim_phi_h}) and thus concludes this lemma.
\end{proof}

\subsection{Strategy of proving Theorem \ref{thm1}} \label{section_strategy}

In this subsection, we present the two main ingredients for the proof of Theorem \ref{thm1}, and then demonstrate how they imply Theorem \ref{thm1}. The first ingredient shows that the harmonic average may reach an unexpectedly high level with a significant probability.

For any $d\ge 3$ and $\lambda\in \mathfrak{X}_d$, we define 
\begin{equation}\label{4.23}
	\overbar{\mathcal{H}}_{*}=\overbar{\mathcal{H}}_{*}(d,\lambda):= \big|\partial^{\mathrm{e}} \mathcal{B}(\tfrac{3N_*}{16})\big|^{-1} \sum\nolimits_{x\in \partial^{\mathrm{e}} \mathcal{B}(\tfrac{3N_*}{16})} \mathcal{H}_{x}\big(\partial^{\mathrm{e}} \mathcal{B}(N_*/4), \mathcal{C}^-_{\partial^{\mathrm{e}} \mathcal{B}(N_*/4)}\big). 
\end{equation}


\begin{proposition}\label{prop_1}
	 Recall $\Cref{const_section_block}$ and $k_*$ in Section \ref{section4.1} and Lemma \ref{lemma_k_diamond} respectively. For $3\le d \le 5$, if $\mathfrak{X}_d$ contains the constant function $\lambda(N)=\Cref{const_section_block}$, then  
	\begin{equation}\label{ineq_prop_1}
		\mathbb{P}\Big(\overbar{\mathcal{H}}_{*}\ge e^{k_*^{5}}\lambda_*N_*^{-\frac{d}{2}+1}\Big) \ge 2^{-k_*^{40d}}. 
	\end{equation}
	For $d=6$, (\ref{ineq_prop_1}) holds if $\mathfrak{X}_d$ contains $\lambda(N)= \Cref{const_section_block}\mathrm{exp}\big(\ln^{\frac{1}{2}}(N)\ln\ln(N)\big)$.
\end{proposition}

To prove Proposition \ref{prop_1}, we first construct a specific event $\mathsf{F}$ (see (\ref{def_F})) with $\mathbb{P}(\mathsf{F})\ge 2^{-k_*^{30d}}$ (see (\ref{ineq_PF})) on which the harmonic average $\mathcal{H}_{\bm{0}}^*$ (recalling (\ref{4.9})) is at least of order $\lambda_*N_*^{-\frac{d}{2}+1}$. Subsequently, we establish that on $\mathsf{F}$, with high probability an independent Brownian motion on $\widetilde{\mathbb{Z}}^d$ starting from $\bm{0}$ will hit $\mathcal{C}^-_{\partial^{\mathrm{e}} \mathcal{B}(N_*/4)}$ before reaching $\partial \mathcal{B}(0.01N_*)$ (see Lemma \ref{lemma_F}). Combined with the harmonicity of the harmonic average, this indeed implies that typically $\overbar{\mathcal{H}}_{*}$ is much larger than $\mathcal{H}_{\bm{0}}^*$ and thus takes an exceedingly large value, as presented in Proposition \ref{prop_1}. The proof of Proposition \ref{prop_1} is postponed to Section \ref{section_block}.


For the second ingredient, through an analysis of the conditional distribution of the average of $\widetilde{\phi}_\cdot$ on $\partial^{\mathrm{e}} \mathcal{B}(\frac{3N_*}{16})$, we establish the following upper bound for the probability that $\overbar{\mathcal{H}}_{*}$ takes a large value.  


\begin{proposition}\label{prop_2}
	Recall $\cref{const_average_2}$ in Lemma \ref{lemma_average_phi}. For any $d\ge 3$ and $\lambda\in \mathfrak{X}_d$, 
\begin{equation}\label{ineq_prop_2}
	\mathbb{P}\Big(\overbar{\mathcal{H}}_{*}\ge e^{k_*^{5}}\lambda_*N_*^{-\frac{d}{2}+1}\Big) \le 4\mathrm{exp}\big( -2^{-4d}\cref{const_average_2}\lambda_*^2e^{2k_*^{5}}\big). 
\end{equation}
\end{proposition}

The proof of Proposition \ref{prop_2} will be presented in Section \ref{subsection_proof_prop_2}.

\begin{proof}[Proof of Theorem \ref{thm1} assuming Propositions \ref{prop_1} and \ref{prop_2}]
	It follows from (\ref{equation_def_C9}) that 
	\begin{equation}\label{4.26}
	4\mathrm{exp}\big( -2^{-4d}\cref{const_average_2}\Cref{const_section_block}^2e^{2k_*^{5}}\big)<4\mathrm{exp}\big( -e^{2k_*^{5}}\big) < 2^{-k_*^{40d}}. 
	\end{equation}
		Combining Propositions \ref{prop_1} and \ref{prop_2} with (\ref{4.26}), we confirm the sufficient condition
	in Remark \ref{remark_sufficient} for Theorem \ref{thm1}, and thus conclude Theorem \ref{thm1}.
	\end{proof}
	


%
%
%
%
%
%

\subsection{Proof of Proposition \ref{prop_2}}\label{subsection_proof_prop_2}
For any $d\ge 3$ and $\lambda\in \mathfrak{X}_d$, we define 
\begin{equation}
\Phi_*:= \big|\partial^{\mathrm{e}} \mathcal{B}(\tfrac{3N_*}{16})\big|^{-1} \sum\nolimits_{x\in \partial^{\mathrm{e}} \mathcal{B}\big(\tfrac{3N_*}{16}\big)} \widetilde{\phi}_x, 
\end{equation}
\begin{equation}
	\Phi_*^{-}:= \big|\partial^{\mathrm{e}} \mathcal{B}(\tfrac{3N_*}{16})\big|^{-1} \sum\nolimits_{x\in \partial^{\mathrm{e}} \mathcal{B}\big(\tfrac{3N_*}{16}\big)} \widetilde{\phi}_x\cdot \mathrm{1}_{x\xleftrightarrow{\le 0} \partial^{\mathrm{e}} \mathcal{B}\big(\tfrac{N_*}{4}\big)}.  
\end{equation}
 Note that for any $x\in \partial^{\mathrm{e}} \mathcal{B}\big(\tfrac{3N_*}{16}\big)$, 
\begin{equation}\label{newadd_4.34}
	\big\{ x\xleftrightarrow{\le 0} \partial^{\mathrm{e}} \mathcal{B}(\tfrac{N_*}{4})   \big\} \subset  \big\{ \widetilde{\phi}_x\le 0  \big\}.
\end{equation}
Recall the notation $\overbar{\mathcal{H}}_{*}$ in (\ref{4.23}). We denote that 
\begin{equation}\label{new_4.34}
	\mathsf{G}_1:= \big\{ \overbar{\mathcal{H}}_{*}\ge e^{k_*^{5}}\lambda_*N_*^{-\frac{d}{2}+1} \big\} \ \ \text{and} \ \ 	\mathsf{G}_2:=\big\{ \Phi_*^{-} \ge -\Cref{const_G_2}\lambda_*N_*^{-\frac{d}{2}+1} \big\},
\end{equation}
where $\Cl\label{const_G_2}(d):=2^{\frac{5}{2}d-4}d^{\frac{d}{2}-1}\Cref{const_crossing_prepare}(d)$ (recalling $\Cref{const_crossing_prepare}$ in Lemma \ref{lemma_crossing_prepare}). Note that $\mathsf{G}_1$ is an increasing event since $\overbar{\mathcal{H}}_{*}$ is increasing with respect to $\widetilde{\phi}$. Meanwhile, $\mathsf{G}_2$ is also an increasing event because of the following obervations:
\begin{itemize}
	\item   If we increase the value of $\widetilde{\phi}_{\cdot}$ on $\widetilde{\mathbb{Z}}^d\setminus \mathcal{C}^-_{\partial^{\mathrm{e}} \mathcal{B}(N_*/4)}$, then $\Phi_*^{-}$ remains unchanged.

	\item   If we increase the value of $\widetilde{\phi}_{\cdot}$ on $\mathcal{C}^-_{\partial^{\mathrm{e}} \mathcal{B}(N_*/4)}$, then we have 
	\begin{itemize}
		\item  for any $x\in  \partial^{\mathrm{e}} \mathcal{B}(\tfrac{3N_*}{16})$, $\widetilde{\phi}_{x}$ can only  increase;

		\item  $\mathcal{C}^-_{\partial^{\mathrm{e}} \mathcal{B}(N_*/4)}$ can only shrink and hence, $\mathrm{1}_{x\xleftrightarrow{\le 0} \partial^{\mathrm{e}} \mathcal{B}\big(\tfrac{N_*}{4}\big)}$ may vanish for some $x\in \partial^{\mathrm{e}} \mathcal{B}(\tfrac{3N_*}{16})$, which will only enlarge $\Phi_*^{-}$ according to (\ref{newadd_4.34}).

	\end{itemize}

\end{itemize}

For any $x\in \partial^{\mathrm{e}} \mathcal{B}(\tfrac{3N_*}{16})$, by the symmetry of $\widetilde{\phi}$, Lemma \ref{lemma_crossing_prepare} and (\ref{add4.10}), we have 
\begin{equation*}\label{4.30}
	\begin{split}
			\mathbb{E}\Big[\widetilde{\phi}_x\cdot  \mathrm{1}_{x\xleftrightarrow{\le 0} \partial^{\mathrm{e}} \mathcal{B}(N_*/4)} \Big]\overset{(\text{symmetry})}{=}& -\mathbb{E}\Big[\widetilde{\phi}_x\cdot  \mathrm{1}_{x\xleftrightarrow{\ge 0} \partial^{\mathrm{e}} \mathcal{B}(N_*/4)} \Big]\\
			\overset{(\text{Lemma}\ \ref{lemma_crossing_prepare})}{\ge}& -\Cref{const_crossing_prepare}\theta_d(\tfrac{N_*}{32d})\overset{(\ref{add4.10})}{\ge} -\tfrac{1}{2}\Cref{const_G_2} \lambda_*N_*^{-\frac{d}{2}+1}, 
	\end{split}
\end{equation*}
which implies that $\mathbb{E}(\Phi_*^{-})\ge -\tfrac{1}{2}\Cref{const_G_2} \lambda_*N_*^{-\frac{d}{2}+1}$. Therefore, since $\Phi_*^{-}\le 0$ (according to (\ref{newadd_4.34})), by Markov's inequality (i.e. $\mathbb{P}(X\le -a)\le \frac{\mathbb{E}(X)}{-a}$ for any random variable $X\le 0$ and real number $a>0$) we have 
\begin{equation}\label{newadd_4.36}
	\mathbb{P}\Big(\Phi_*^{-}\le -\Cref{const_G_2}\lambda_*N_*^{-\frac{d}{2}+1}\Big)\le  \frac{\mathbb{E}(\Phi_*^{-})}{-\Cref{const_G_2}\lambda_*N_*^{-\frac{d}{2}+1}}\le \frac{1}{2}. 
\end{equation}
It follows from (\ref{newadd_4.36}) that $\mathbb{P}(\mathsf{G}_2)\ge \frac{1}{2}$. As a result, by the FKG inequality, 
\begin{equation}\label{4.31}
		\mathbb{P}(\mathsf{G}_1\cap \mathsf{G}_2)\ge\tfrac{1}{2} \mathbb{P}(\mathsf{G}_1). 
\end{equation}

Note that $\Phi_*^{-}$ is measurable with respect to $\mathcal{F}_{\mathcal{C}^-_{\partial^{\mathrm{e}} \mathcal{B}(N_*/4)}}$. In addition, by Lemma \ref{lemma_strong_markov}, conditioning on $\mathcal{F}_{\mathcal{C}^-_{\partial^{\mathrm{e}} \mathcal{B}(N_*/4)}}$, for each $x\in  \partial^{\mathrm{e}} \mathcal{B}(\tfrac{3N_*}{16})$ with $\big\{x\xleftrightarrow{\le 0} \partial^{\mathrm{e}} \mathcal{B}(\tfrac{N_*}{4})\big\}^c$, $\widetilde{\phi}_x$ is distributed as a normal random variable with mean $\mathcal{H}_{x}\big(\partial^{\mathrm{e}} \mathcal{B}(N_*/4), \mathcal{C}^-_{\partial^{\mathrm{e}} \mathcal{B}(N_*/4)}\big)$, and thus $\Phi_*$ follows a normal distribution with mean $\Phi_*^{-}+\overbar{\mathcal{H}}_{*}$. Therefore, since $\Phi_*^{-}+\overbar{\mathcal{H}}_{*}\ge \big(e^{k_*^{5}}-\Cref{const_G_2} \big)\lambda_*N_*^{-\frac{d}{2}+1} \overset{(\ref{require_constant2.3})}{\ge } \tfrac{1}{2}e^{k_*^{5}}\lambda_*N_*^{-\frac{d}{2}+1}$ holds on the event $\mathsf{G}_1\cap \mathsf{G}_2$, 
\begin{equation}\label{4.34}
	\begin{split}
		\mathbb{P}\Big(\Phi_*\ge  \tfrac{1}{2}e^{k_*^{5}}\lambda_*N_*^{-\frac{d}{2}+1} \Big)  \ge &\mathbb{E}\Big[ \mathbb{P}\big(\Phi_*\ge \tfrac{1}{2}e^{k_*^{5}}\lambda_*N_*^{-\frac{d}{2}+1}  \mid  \mathcal{F}_{\mathcal{C}^-_{\partial^{\mathrm{e}} \mathcal{B}(N_*/4)}}\big)\cdot \mathbbm{1}_{\mathsf{G}_1\cap \mathsf{G}_2} \Big]\\
		\ge  & \tfrac{1}{2}\mathbb{P}(\mathsf{G}_1\cap \mathsf{G}_2)\overset{(\ref{4.31})}{\ge} \tfrac{1}{4} \mathbb{P}(\mathsf{G}_1). 
	\end{split}
\end{equation}
Moreover, it follows from (\ref{newadd_2.30}) (taking $D=\emptyset$) that 
\begin{equation}\label{4.37}
	\begin{split}
		\mathbb{P}\big(\Phi_*\ge \tfrac{1}{2}e^{k_*^{5}}\lambda_*N_*^{-\frac{d}{2}+1} \big) \le \mathrm{exp}\big( -2^{-4d}\cref{const_average_2}\lambda_*^2e^{2k_*^{5}}\big). 
	\end{split}
\end{equation}
	Combining (\ref{4.34}) and (\ref{4.37}), we conclude Proposition \ref{prop_2}.   \qed

\section{Block the Brownian motion by negative clusters}\label{section_block}

The aim of this section is to establish Proposition \ref{prop_1}, which is achieved through the following three subsections.
\begin{itemize}
	\item \textbf{Section \ref{subsection_con_F}.} We introduce several types of points and boxes concerning harmonic averages and negative clusters of the GFF $\widetilde{\phi}$, and then establish some useful properties of them. Based on these definitions and properties, we construct a crucial event $\mathsf{F}$ and estimate its probability.

	\item  \textbf{Section \ref{subsection_stu_F}.} We analyze the Brownian motion on the event $\mathsf{F}$. Specifically, we estimate how likely an independent Brownian motion can reach a certain distance without hitting certain negative clusters of $\widetilde{\phi}$.

	\item   \textbf{Section \ref{section_prove_prop1}.} We prove Proposition \ref{prop_1} using properties of the event $\mathsf{F}$.

\end{itemize} 


Recall the constant $\Cref{const_section_block}$ in Section \ref{section4.1}. For any $3\le d\le 6$, throughout this section we assume that $\lambda\in \mathfrak{X}_d$ and that 
\begin{itemize}
	\item    when $d\in \{3,4,5\}$, $\lambda(N)= \Cref{const_section_block}$;

	\item    when $d=6$, $\lambda(N)= \Cref{const_section_block} \mathrm{exp}\big(\ln^{\frac{1}{2}}(N)\ln\ln(N)\big)$.

\end{itemize}

\subsection{Construction of the event $\mathsf{F}$}\label{subsection_con_F}

As a key component of the proof, we first introduce the definitions of good points and good boxes, and then we present some intuitions and properties for these concepts.

Recall $\mathcal{H}^*_\cdot $ and $k_*$ in (\ref{4.9}) and Lemma \ref{lemma_k_diamond} respectively. For any $3\le d\le 6$, we denote $\widehat{N}_*=\widehat{N}_*(d,\lambda):=k_*^{-100}N_*$ and $\widecheck{N}_*=\widecheck{N}_*(d,\lambda):=k_*^{100d}2^{-0.5k_*}N_*$. For any $j\ge 0$, let $\widecheck{N}_j=\widecheck{N}_j(d,\lambda) :=2^{-j}\widecheck{N}_*$. We define (recall $\Cref{const_crossing}$ in Proposition \ref{lemma_bound_crossing})
\begin{equation}\label{5.1}
	j_*= j_*(d,\lambda):= \min\big\{ j\in \mathbb{N}:\Cref{const_crossing} \lambda(\widecheck{N}_j)\lambda(N_*) \big(8d\widecheck{N}_jN_*^{-1}\big)^{\frac{d}{2}-1}\le  2^{-k_*-1} \big\}.
\end{equation}
We now establish the following bounds for later use. For any $0\le j<j_*$, since $j$ does not satisfy the condition in (\ref{5.1}), we have 
\begin{equation*}
	\Cref{const_crossing} \lambda(\widecheck{N}_{j})\lambda(N_*) \big(8d\widecheck{N}_{j}N_*^{-1}\big)^{\frac{d}{2}-1}>  2^{-k_*-1}. 
\end{equation*}
This implies that (using $\widecheck{N}_{j}=k_*^{100d}2^{-j-0.5k_*}N_*$)
\begin{equation}\label{cal1}
	2^{j}<\Cref{const_crossing}^2k_*^{200d}2^{(\frac{6-d}{2(d-2)})k_*}\big[\lambda(\widecheck{N}_{j})\lambda(N_*)\big]^{\frac{2}{d-2}},
\end{equation}  
and thus (recalling that $\Cref{const_crossing}>1$, $\lambda(\widecheck{N}_{j})\le \lambda(N_*)$ and $\lambda_*:=\lambda(N_*)$), 
\begin{equation}\label{cal2}
	j\le \Cref{const_crossing}k_*^2\log_2(\lambda_*). 
\end{equation}


\begin{definition}[Good point]\label{def_good_point}
	For any $3\le d\le 6$ and $y\in \mathbb{Z}^d$, we say $y$ is a good point if $\mathcal{H}^*_y \ge \cref{const_small}^2\lambda_* N_*^{-\frac{d}{2}+1} k_*^{-3}2^{k_*}$ and $\big\{\mathcal{B}_y(\widecheck{N}_{j_*})\xleftrightarrow{\le 0} \partial^{\mathrm{e}} \mathcal{B}_y(\frac{N_*}{2})\big\}^c$ both hold. 
\end{definition}
When $\widecheck{N}_{j_*}<1$, we have $\mathcal{B}_y(\widecheck{N}_{j_*})=\{y\}$ and hence, $$\big\{\mathcal{B}_y(\widecheck{N}_{j_*})\xleftrightarrow{\le 0} \partial^{\mathrm{e}} \mathcal{B}_y(\tfrac{N_*}{2})\big\}^c=\big\{y\xleftrightarrow{\le 0} \partial^{\mathrm{e}} \mathcal{B}_y(\tfrac{N_*}{2})\big\}^c,$$ 
which is implied by (i.e. contains) the event $\big\{\mathcal{H}^*_y \ge \cref{const_small}^2\lambda_* N_*^{-\frac{d}{2}+1} k_*^{-3}2^{k_*}\big\}$. Thus, when $j_*$ is excessively large, in Definition \ref{def_good_point} $\big\{\mathcal{B}_y(\widecheck{N}_{j_*})\xleftrightarrow{\le 0} \partial^{\mathrm{e}} \mathcal{B}_y(\frac{N_*}{2})\big\}^c$ does not provide any extra restriction. As a result, by Lemma \ref{lemma_k_diamond} we have: for any $y\in \mathbb{Z}^d$,  
\begin{equation*}
	\mathbb{P}(y\ \text{is a good point})= \mathbb{P}\big(\mathcal{H}^*_y \ge \cref{const_small}^2\lambda_* N_*^{-\frac{d}{2}+1} k_*^{-3}2^{k_*}\big)\overset{(\text{Lemma}\ \ref{lemma_k_diamond})}{\ge} 2^{-k_*}.
\end{equation*}
Otherwise (i.e.\ $\widecheck{N}_{j_*}\ge1$), by Lemma \ref{lemma_k_diamond}, (\ref{inclusion_box}), Proposition \ref{lemma_bound_crossing}, (\ref{add4.10}) and (\ref{5.1}),
\begin{equation*}
	\begin{split}
		&\mathbb{P}(y\ \text{is a good point})\\
	\overset{(\text{Lemma}\ \ref{lemma_k_diamond})}{\ge} &2^{-k_*}- \mathbb{P}\big[\mathcal{B}_y(\widecheck{N}_{j_*})\xleftrightarrow{\le 0} \partial^{\mathrm{e}}  \mathcal{B}_y(\tfrac{N_*}{2})\big]\\
		\overset{(\ref{inclusion_box})}{\ge}  & 2^{-k_*}- \mathbb{P}\big[B(\widecheck{N}_{j_*})\xleftrightarrow{\le 0} \partial  B(\tfrac{N_*}{2d})\big]          \\
			\overset{(\text{Proposition}\ \ref{lemma_bound_crossing})}{\ge} & 2^{-k_*}- \Cref{const_crossing}\widecheck{N}_{j_*}^{d-2}\theta_{d}(\widecheck{N}_{j_*})\theta_d(\tfrac{N_*}{8d})\\
			\overset{(\ref{add4.10})}{\ge } &  2^{-k_*}- \Cref{const_crossing} \lambda(\widecheck{N}_{j_*})\lambda(N_*) \big(8d\widecheck{N}_{j_*}N_*^{-1}\big)^{\frac{d}{2}-1} \overset{(\ref{5.1})}{\ge }  2^{-k_*-1}.
	\end{split}
\end{equation*}
To sum up, in both cases we have
\begin{equation}\label{5.2}
	\mathbb{P}(y\ \text{is a good point}) \ge 2^{-k_*-1}, \ \ \forall y\in \mathbb{Z}^d.
\end{equation}

%
%
%
%

\begin{definition}[Good box]\label{def_good_box}
	For any $x\in \mathbb{Z}^d$, we say $B_x(\widehat{N}_*)$ is a good box if the proportion of good points in $B_x(\widehat{N}_*)$ is at least $2^{-k_*-2}$. 
\end{definition}

For any $x\in \mathbb{Z}^d$, it follows from (\ref{5.2}) that
\begin{equation*}
	\begin{split}
		2^{-k_*-1}|B_x(\widehat{N}_*)| \le& \mathbb{E}\Big(\big|\{y\in B_x(\widehat{N}_*):y\ \text{is a good point}  \}\big|\Big)\\
		\le & \mathbb{P}\big[B_x(\widehat{N}_*)\ \text{is a good box}\big]\cdot \big|B_x(\widehat{N}_*)\big| +  2^{-k_*-2}\big|B_x(\widehat{N}_*)\big|,
	\end{split}
\end{equation*}
which implies that 
\begin{equation}\label{5.3}
	\mathbb{P}\big[B_x(\widehat{N}_*)\ \text{is a good box}\big] \ge 2^{-k_*-2}. 
\end{equation}
Note that for any $x\in \mathbb{Z}^d$, $\{x\ \text{is a good point}\}$ and $\big\{B_x(\widehat{N}_*)\ \text{is a good box}\big\}$ are both increasing events.

\begin{remark}\label{remark_good_points}
	For a good point $y\in \mathbb{Z}^d$, the harmonic average at $y$ is required to be large. In fact, according to the exploration martingale (see (\ref{new_2.17}) and Lemma \ref{lemma_EM}), this requirement implies that with high probability, the Green's function at $y$ will be decreased significantly if a killing boundary condition is posed on the negative cluster $\mathcal{C}^-_{\partial^{\mathrm{e}}\mathcal{B}_v(N_*/2)}$. In other words, the Brownian motion starting from $y$ will hit $\mathcal{C}^-_{\partial^{\mathrm{e}}\mathcal{B}_v(N_*/2)}$ with a considerable probability. In addition, for convenience of our further estimates, we need the second requirement $\big\{\mathcal{B}_y(\widecheck{N}_{j_*})\xleftrightarrow{\le 0} \partial^{\mathrm{e}} \mathcal{B}_y(N_*/2)\big\}^c$ to avoid the scenario that $\mathcal{C}^-_{\partial^{\mathrm{e}}\mathcal{B}_y(N_*/2)}$ gets too close to $y$ (because the Green's function at $y$ is too sensitive to exploration near $y$).

In an earlier version of the notion of good points (during the course of research), we directly required that $\mathcal{C}^-_{\partial^{\mathrm{e}}\mathcal{B}_y(N_*/2)}$ is sufficiently large such that the Brownian motion starting from $y$ has a considerable probability of hitting it. However, such a requirement constitutes a decreasing event. As a result, we cannot use it to construct a monotonic event $\mathsf{F}$ (we require monotonicity for the FKG inequality) where the harmonic average at some point is large, since the harmonic average is an increasing random variable. To address this problem, we maintain the requirement on harmonic averages in defining good points, and additionally, introduce the concept of suitable points and suitable boxes (see Definitions \ref{def_suitable_point} and \ref{def_suitable_box}) to achieve the transformation from high harmonic averages to high hitting probabilities for the negative clusters.

\end{remark}

The following lemma includes a key observation: Given a sufficient number of good points within a box, it is feasible to select a sub-box of considerable size such that starting from any point in this sub-box, a Brownian motion on $\widetilde{\mathbb{Z}}^d$ (or equivalently, a simple random walk) will visit numerous good points with a significant probability before escaping faraway from this sub-box.

\begin{definition}[Excellent box]\label{def_excellent_box}
	We say a box $B$ is excellent if the subset $A:=\{y\in B: y\ \text{is good}\}$ satisfies that for any $A'$ obtained from $A$ by removing at most $2^{-k_*^{60d}}\big|B(\widehat{N}_*)\big|$ points, we have 
	\begin{equation}
		\widetilde{\mathbb{P}}_z\big(\nu\ge k_*^{-3}2^{-k_*}\widehat{N}_*^2\big)\ge k_*^{-3},\ \ \forall z\in B,
	\end{equation}
	where $\nu$ is the number of points in $A'$ visited by the Brownian motion before exiting $20B$ (which is a concentric box of $B$ whose size is $20$ times that of $B$). Note that the removal of points in this definition will be echoed in Definition \ref{def_suitable_box}.
\end{definition}

\begin{remark}\label{remark_translation}
	Recalling (\ref{4.9}) and the translation invariance of the GFF, the collection of good points is translation invariant in distribution. As a result, the probability for a box to be good or excellent is also translation invariant.
\end{remark}

\begin{lemma}\label{lemma_excellent_box}
	On the event $\{B(\widehat{N}_*)\ \text{is a good box}\}$, there exists a (random) integer $l\ge 0$ with $2^{-l}\ge k_*^{-2}2^{-0.5k_*}$ and a (random) point $w\in (2^{-l-10}\widehat{N}_*\cdot \mathbb{Z}^d)\cap B(\widehat{N}_*)$ such that the box $B_w(2^{-l}\widehat{N}_*)$ is excellent. 
\end{lemma}
\begin{proof}
	The proof is based on the second moment method, with a careful choice for $l$ and $w$. For illustration, let us consider an extreme case where all good points form a sub-box $B' \subset B(\widehat{N}_*)$. Then naturally (for the second moment method to work) we should choose $B'$ as the desired sub-box (which is excellent). A moment of thinking leads to a choice of $w$ such that a Brownian motion from $w$ will have the maximal expected number of visits to good points, and a choice of $l$ such that this expectation has a significant contribution from the good points in $B_w(2^{-l} \widehat{N}_*)$. We next carry out the proof where the preceding heuristic will be slightly modified.


Let $A$ be the collection of good points in $B(\widehat{N}_*)$. Since $B(\widehat{N}_*)$ is a good box,
\begin{equation}\label{5.5}
	|A|\ge 2^{-k_*-2}\big|B(\widehat{N}_*)\big| \ge  2^{-k_*}\widehat{N}_*^d.
\end{equation}
Let $x_\dagger$ be the point $x\in A$ maxmizing $\sum_{y \in A}|x-y|^{2-d}$ (recall that we set $0^{-a}=1$ for $a>0$). Since $A\subset B_{x_\dagger}(2\widehat{N}_*)= \cup_{l\ge 1}[B_{x_\dagger}(2^{2-l}\widehat{N}_*)\setminus B_{x_\dagger}(2^{1-l}\widehat{N}_*)]\cup \{x_\dagger\}$ and $\sum_{l\ge 1}l^{-2}<2$, there exists an integer $l_\dagger\ge 1$ such that 
\begin{equation}\label{5.6}
	\begin{split}
		\sum\nolimits_{y \in A_\dagger}|x_\dagger-y|^{2-d}
		\ge \tfrac{1}{2}l_{\dagger}^{-2}\sum\nolimits_{y \in A\setminus \{x_\dagger\}}|x_\dagger-y|^{2-d}, 
	\end{split}
\end{equation}
where $A_\dagger:= A\cap [B_{x_\dagger}(2^{2-l_\dagger}\widehat{N}_*)\setminus B_{x_\dagger}(2^{1-l_\dagger}\widehat{N}_*)]$. In fact, $l_\dagger$ cannot be too large. To see this, on the one hand, by (\ref{5.5}) and $A\subset B_{x_\dagger}(2\widehat{N}_*)\subset \mathcal{B}_{x_\dagger}(2d^{\frac{1}{2}}\widehat{N}_*)$, 
	\begin{equation}\label{5.7}
	\sum\nolimits_{y \in A\setminus \{x_\dagger\}}|x_\dagger-y|^{2-d} \ge   (2d^{\frac{1}{2}}\widehat{N}_*)^{2-d}(|A|-1) \ge d^{-2d}2^{-k_{*}}\widehat{N}_*^2.  
\end{equation}
On the other hand, $A_\dagger \subset B_{x_\dagger}(2^{2-l_\dagger}\widehat{N}_*)\setminus B_{x_\dagger}(2^{1-l_\dagger}\widehat{N}_*)$ implies that $|A_\dagger|\le 3^d(2^{2-l_\dagger}\widehat{N}_*)^d$ and that $|x_\dagger-y|\ge 2^{1-l_\dagger}\widehat{N}_*$ for all $y\in A_{\dagger}$. As a result, we have 
	\begin{equation}\label{5.8}
	\begin{split}
		\sum\nolimits_{y \in A_\dagger}|x_\dagger-y|^{2-d}\le 	3^d(2^{2-l_\dagger}\widehat{N}_*)^d( 2^{1-l_\dagger}\widehat{N}_*)^{2-d}
		\le  d^{4d}2^{-2l_\dagger}\widehat{N}_*^2.
	\end{split}
\end{equation}
By (\ref{5.6}), (\ref{5.7}) and (\ref{5.8}), we have $2^{-k_*}\le d^{6d}l_{\dagger}^{2}2^{1-2l_\dagger}\le 10d^{6d}2^{-1.5l_\dagger}$ (since $t^22^{-2t}<5\cdot 2^{-1.5t}$ for all $t\ge 1$), which together with $k_*\ge \log_2(1/\cref{const_small})\ge 100d\log_2(10d)$ implies that $l_{\dagger}\le k_*$. Combined with $2^{-k_*}\le d^{6d}l_{\dagger}^{2}2^{1-2l_\dagger}$, it yields that
\begin{equation}\label{new_5.9}
  2^{-l_\dagger}\ge d^{-4d} k_*^{-1}2^{-0.5k_*}.
\end{equation}
Meanwhile, by (\ref{5.6}), (\ref{5.7}) and $l_{\dagger}\le k_*$, we get 
\begin{equation}\label{new_5.10}
	\sum\nolimits_{y \in A_\dagger}|x_\dagger-y|^{2-d} \ge \tfrac{1}{2}d^{-2d}k_*^{-2}2^{-k_{*}}\widehat{N}_*^2. 
\end{equation}


Let $A_\dagger'$ be an arbitrary set obtained from $A_\dagger$ by removing at most $2^{-k_*^{60d}}\big|B(\widehat{N}_*)\big|$ points. Since $|A_\dagger\setminus A_\dagger'|\le 2^{-k_*^{60d}}\big|B(\widehat{N}_*)\big|\le  2^{-k_*^{60d}}3^d\widehat{N}_*^d$ and $A_\dagger\subset \mathbb{Z}^d\setminus B_{x_\dagger}(2^{1-l_\dagger}\widehat{N}_*)$, 
\begin{equation*}\label{new_5.11}
	\sum\nolimits_{y\in A_\dagger\setminus A_\dagger'} |x_\dagger-y|^{2-d}\le 2^{-k_*^{60d}}3^d\widehat{N}_*^d (2^{1-l_\dagger}\widehat{N}_*)^{2-d} \overset{(l_\dagger\le k_*)}{\le } 2^{-k_*^{55d}} \widehat{N}_*^{2}. 
\end{equation*}
Combined with (\ref{new_5.10}), it yields that 
\begin{equation}\label{new_add_5.11}
	\begin{split}
			\sum\nolimits_{y\in A_\dagger'} |x_\dagger-y|^{2-d} \ge& \sum\nolimits_{y\in A_\dagger} |x_\dagger-y|^{2-d}- 2^{-k_*^{55d}} \widehat{N}_*^{2}\\
			\overset{(\ref{new_5.10})}{\ge} &\big(1- 2^{-k_*^{50d}}\big)  \sum\nolimits_{y\in A_\dagger} |x_\dagger-y|^{2-d}\\
			 \ge & \frac{1}{2}\sum\nolimits_{y\in A_\dagger} |x_\dagger-y|^{2-d}. 
	\end{split}
\end{equation}
For any $y\in A_\dagger'$, we denote $\mathsf{A}_y:=\big\{\tau_y<\tau_{\partial B_{x_\dagger}(2^{4-l_\dagger}\widehat{N}_*)}\big\}$. Let $\mathbf{X}:=\sum_{y\in A_\dagger'}\mathbbm{1}_{\mathsf{A}_y}$. For any $x, y\in B_{x_\dagger}(2^{2-l_\dagger}\widehat{N}_*)$, by \cite[Proposition 1.5.10]{lawler2013intersections} one has
\begin{equation}\label{5.9}
	\widetilde{\mathbb{P}}_x\big(\mathsf{A}_y\big)\asymp |x-y|^{2-d}. 
\end{equation}
Combining (\ref{new_5.10}), (\ref{new_add_5.11}) and (\ref{5.9}), we get 
\begin{equation}\label{new_5.14}
	\begin{split}
			\widetilde{\mathbb{E}}_{x_\dagger}(\mathbf{X})\overset{(\ref{5.9})}{\ge} &c\sum\nolimits_{y\in A_\dagger'} |x_\dagger-y|^{2-d} \\
			 \overset{(\ref{new_add_5.11})}{\ge}&  \frac{c}{2}\sum\nolimits_{y\in A_\dagger} |x_\dagger-y|^{2-d}  
			 \overset{(\ref{new_5.10})}{\ge}  c'k_*^{-2}2^{-k_{*}}\widehat{N}_*^2. 
	\end{split}
\end{equation}
In addition, by (\ref{range_k_*}), the right-hand side of (\ref{new_5.14}) is bounded from below by 
\begin{equation}\label{add_5.16}
\mathbb{I}_d:= ck_*^{-202}\lambda_*N_*^{\frac{6-d}{2}}\ln^{-10}(N_*). 
\end{equation}
In fact, we have $\mathbb{I}_d\ge  c\ln^{10}(N_*)$ for $3\le d\le 6$. To see this, note that (\ref{range_k_*}) implies $\ln(N_*)\ge ck_*$. For $3\le d\le 5$, by $\ln(N_*)\ge ck_*$, $\lambda_*\ge 1$ and $N_*^{\frac{6-d}{2}}\ge \ln^{300}(N_*)$, we have $\mathbb{I}_d\ge  c\ln^{10}(N_*)$. For $d=6$,  since $\lambda_*=\Cref{const_section_block} \mathrm{exp}\big(\ln^{\frac{1}{2}}(N_*)\ln\ln(N_*)\big)\ge \ln^{300}(N_*)$, we also have $\mathbb{I}_6\ge  c\ln^{10}(N_*)$. In conclusion, 
\begin{equation}\label{add_5.17}
	\widetilde{\mathbb{E}}_{x_\dagger}(\mathbf{X}) \ge  c\ln^{10}(N_*). 
\end{equation}


Now we estimate $\widetilde{\mathbb{E}}_{x_\dagger}(\mathbf{X}^2)$. For any $y_1,y_2\in A_\dagger'$, by the strong Markov property and (\ref{5.9}), we have 
\begin{equation}\label{5.10}
	\begin{split}
		\widetilde{\mathbb{P}}_{x_\dagger}\big(\mathsf{A}_{y_1}\cap \mathsf{A}_{y_2}\big)
		\le& \sum\nolimits_{i\in \{1,2\}} \widetilde{\mathbb{P}}_{x_\dagger}\big(\mathsf{A}_{y_i}\big) \widetilde{\mathbb{P}}_{y_i}\big(\mathsf{A}_{y_{2-i}}\big)\\
		\le &C|y_1-y_2|^{2-d}\sum\nolimits_{i\in \{1,2\}}\widetilde{\mathbb{P}}_{x_\dagger}\big(\mathsf{A}_{y_i}\big). 
	\end{split}
\end{equation}
	Therefore, we obtain the following estimate for $\widetilde{\mathbb{E}}_{x_\dagger}(\mathbf{X}^2)$: 
	\begin{equation}\label{new5.16}
		\begin{split}
				\widetilde{\mathbb{E}}_{x_\dagger}(\mathbf{X}^2)=&	  \sum\nolimits_{y_1, y_2\in A_\dagger'}\widetilde{\mathbb{P}}_{x_\dagger}\big(\mathsf{A}_{y_1}\cap \mathsf{A}_{y_2}\big) \\
			\overset{(\ref{5.10})}{\le} & 2C\sum\nolimits_{y_1\in A_\dagger'}\widetilde{\mathbb{P}}_{x_\dagger}\big(\mathsf{A}_{y_1}\big) \sum\nolimits_{y_2\in A_\dagger'}|y_1-y_2|^{2-d} \\
			\overset{(\text{maximality of}\ x_\dagger)}{\le } &2C\widetilde{\mathbb{E}}_{x_\dagger}(\mathbf{X})  \Big(\sum\nolimits_{y \in A\setminus \{x_\dagger\}}|x_\dagger-y|^{2-d}+1\Big)\\
			\overset{(\ref{5.6}),l_{\dagger}\le k_*}{\le } & 4Ck_*^2 \widetilde{\mathbb{E}}_{x_\dagger}(\mathbf{X}) \Big(\sum\nolimits_{y \in A_\dagger}|x_\dagger-y|^{2-d}+1\Big)\\
			\overset{(\ref{new_add_5.11})}{\le }  & 8Ck_*^2\widetilde{\mathbb{E}}_{x_\dagger}(\mathbf{X})  \Big(\sum\nolimits_{y \in A_\dagger'}|x_\dagger-y|^{2-d}+1 \Big)\\
			\overset{(\ref{5.9})}{\le } &  C'k_*^2 \Big[ \widetilde{\mathbb{E}}_{x_\dagger}(\mathbf{X}) +1 \Big]^{2}\overset{(\ref{add_5.17})}{\le }C''k_*^2\Big[\widetilde{\mathbb{E}}_{x_\dagger}\big(\mathbf{X}\big)\Big]^2.
		\end{split}
	\end{equation}
		 By (\ref{new_5.14}), (\ref{new5.16}) and the Paley–Zygmund inequality, one has 
		 \begin{equation}\label{5.15}
		 	\begin{split}
		 		\widetilde{\mathbb{P}}_{x_\dagger}\big(\mathbf{X}\ge ck_*^{-2}2^{-k_*}\widehat{N}_*^2\big) \overset{(\ref{new_5.14})}{\ge } & \widetilde{\mathbb{P}}_{x_\dagger}\big(\mathbf{X}\ge  \tfrac{1}{2}\widetilde{\mathbb{E}}_{x_\dagger}\mathbf{X}\big)\\
		 	 \overset{(\text{Paley–Zygmund ineq})}{\ge }& 4^{-1} \big[\widetilde{\mathbb{E}}_{x_\dagger}\big(\mathbf{X}\big)\big]^2/ \widetilde{\mathbb{E}}_{x_\dagger}\big(\mathbf{X}^2\big)	\overset{(\ref{new5.16})}{\ge } c'k_*^{-2}. 
		 	\end{split}
		 \end{equation}

	When $1\le l_\dagger\le 3$, we take $l=0$ and $w=\bm{0}$. Otherwise (i.e. $l_\dagger\ge 4$), we take $l=l_\dagger-3$ and let $w$ be the closest point in $(2^{-l-10}\widehat{N}_*\cdot \mathbb{Z}^d)\cap B(\widehat{N}_*)$ to $x_\dagger$ (we break the tie in some predetermined manner). For any $z\in B_w(2^{-l}\widehat{N}_*)$, by the invariance principle, we have
	\begin{equation}\label{5.16}
		\widetilde{\mathbb{P}}_z\big[\tau_{B_{x_\dagger}(2^{-l_\dagger}\widehat{N}_*)}<\tau_{\partial B_w(2^{-l+1}\widehat{N}_*)}\big]\ge c\in (0,1).
	\end{equation}
	Meanwhile, by Harnack's inequality (see e.g. \cite[Theorem 6.3.9]{lawler2010random}), we have 
	\begin{equation}\label{5.17}
		\widetilde{\mathbb{P}}_{z'}\big(\mathbf{X}\ge c k_*^{-2}2^{-k_*}\widehat{N}_*^2\big)\asymp \widetilde{\mathbb{P}}_{x_\dagger}\big(\mathbf{X}\ge c k_*^{-2}2^{-k_*}\widehat{N}_*^2\big),\ \forall z'\in B_{x_\dagger}(2^{-l_\dagger}\widehat{N}_*). 
	\end{equation}
	Let $\nu_{w,l}'$ be the number of points in $A_\dagger'$ visited by the Browinian motion before exiting $B_w(20\cdot 2^{-l}\widehat{N}_*)$. By the strong Markov property, (\ref{5.15}), (\ref{5.16}) and (\ref{5.17}), we get: for any $z\in B_w(2^{-l}\widehat{N}_*)$, 
	\begin{equation}\label{5.18}
		\begin{split}
			&\widetilde{\mathbb{P}}_{z}\big(\nu_{w,l}'\ge ck_*^{-2}2^{-k_*}\widehat{N}_*^2\big)\\
			\overset{(\ref{5.16})}{\ge } &c' \cdot  \min\nolimits_{z'\in B_{x_\dagger}(2^{-l_\dagger}\widehat{N}_*)}\widetilde{\mathbb{P}}_{z'}\big(\mathbf{X}\ge c k_*^{-2}2^{-k_*}\widehat{N}_*^2\big)\\
			\overset{(\ref{5.17})}{\ge } &c'' \widetilde{\mathbb{P}}_{x_\dagger}\big(\mathbf{X}\ge c k_*^{-2}2^{-k_*}\widehat{N}_*^2\big)	\overset{(\ref{5.15})}{\ge }  c'''k_*^{-2}. 
		\end{split}
	\end{equation}
	In conclusion, by (\ref{new_5.9}) and (\ref{5.18}), there exist constants $c_\dagger^{(1)},c_\dagger^{(2)},c_\dagger^{(3)}>0$ such that $2^{-l}\ge c_\dagger^{(1)}k_*^{-1} 2^{-0.5k_*}$ and $\widetilde{\mathbb{P}}_{z}\big(\nu'_{w,l}\ge c_\dagger^{(2)}k_*^{-2} 2^{-k_*}\widehat{N}_*^2\big)\ge c_\dagger^{(3)}k_*^{-2}$ for all $z\in B_w(2^{-l}\widehat{N}_*)$. Thus, since $k_*^{-1}\le \min\{c_\dagger^{(1)},c_\dagger^{(2)},c_\dagger^{(3)}\}$ (which follows from (\ref{require_constant1.1})), we obtain that $	2^{-l}\ge  k_*^{-2}2^{-0.5k_*}$, and that $B_w(2^{-l}\widehat{N}_*)$ is an excellent box.   
	\end{proof}

	

	By applying the union bound, we derive from Lemma \ref{lemma_excellent_box} and (\ref{5.3}) that 
	\begin{equation*}
		\begin{split}
		&	\sum_{l\ge 0: 2^{-l}\ge k_*^{-2}2^{-0.5k_*}}\sum_{w \in (2^{-l-10}\widehat{N}_*\cdot \mathbb{Z}^d)\cap B(\widehat{N}_*)}  \mathbb{P}\big[B_w(2^{-l}\widehat{N}_*)\ \text{is excellent}\big]\\
			\overset{(\text{Lemma}\ \ref{lemma_excellent_box})}{\ge } &\mathbb{P}\big[B(\widehat{N}_*)\ \text{is good}\big]	\overset{(\ref{5.3})}{\ge }   2^{-k_*-2}. 
		\end{split}
	\end{equation*}
	Combined with $l\le k_*$ and $\big|(2^{-l-10}\widehat{N}_*\cdot \mathbb{Z}^d)\cap B(\widehat{N}_*)\big|\le 2^{dk_*}$ (both of which follow from $2^{-l}\ge k_*^{-2}2^{-0.5k_*}$), it implies that there exists a (deterministic) integer $l_\diamond\ge 0$ with $2^{-l_\diamond}\ge k_*^{-2}2^{-0.5k_*}$ and a (deterministic) point $w_\diamond\in \mathbb{Z}^d$ such that
	\begin{equation}\label{5.22}
		\mathbb{P}\big[B_{w_\diamond}(2^{-l_\diamond}\widehat{N}_*)\ \text{is excellent}\big] \ge 2^{-2dk_*}. 
	\end{equation}
	Let $N_\diamond:= 2^{-l_\diamond}\widehat{N}_*$. By (\ref{5.22}) and the translation invariance of the excellent box in distribution (see Remark \ref{remark_translation}), we have 
\begin{equation}\label{new_5.18}
	\mathbb{P}(\mathsf{D}_y):= \mathbb{P}\big[B_y(N_\diamond)\ \text{is excellent}\big] \ge 2^{-2dk_*},  \ \ \forall y\in \mathbb{Z}^d. 
\end{equation}
	Note that $\mathsf{D}_y$ is an increasing event. We also define the event 
	\begin{equation}
		\mathsf{D}:= \bigcap_{y\in [-k_*^{10},k_*^{10}]^d\cap \mathbb{Z}^d}  \mathsf{D}_{N_\diamond \cdot y }  \cap  \{\bm{0}\ \text{is a good point}\}  .
	\end{equation}
	By the FKG inequality, (\ref{5.2}) and (\ref{new_5.18}), we have 
	\begin{equation}\label{5.20}
		\mathbb{P}(\mathsf{D})\ge 2^{-k_*^{20d}}. 
	\end{equation}

We next construct the event $\mathsf{F}$ by posing some additional restrictions to the event $\mathsf{D}$. To this end, as mentioned in Remark \ref{remark_good_points} we introduce the definitions of suitable points and suitable boxes to facilitate the proof that the Brownian motion is likely to be blocked by negative clusters when crossing excellent boxes.

For any $x\in \mathbb{Z}^d$, we denote 
\begin{equation}
	\mathcal{K}_x^*= \mathcal{K}_x^*(d,\lambda):= \sum\nolimits_{z\in \partial^{\mathrm{e}}  \mathcal{B}_x(N_*/2)} \widetilde{\mathbb{P}}_x\big(\tau_{\partial^{\mathrm{e}}  \mathcal{B}_x(N_*/2)}=\tau_z\big) \widetilde{\phi}_z. 
\end{equation}
Recall $\mathcal{M}_{x,t}^{A,-}$ and $\langle \mathcal{M}_x^{A,-} \rangle_t$ in Section \ref{section_EM}, and recall $\mathcal{H}_{\cdot}^{*}$ in (\ref{4.9}). Let
\begin{equation*}
	\mathcal{M}_{x,t}^{*,-}:=\mathcal{M}_{x,t}^{\partial^{\mathrm{e}}  \mathcal{B}_x(N_*/2),-} \ \text{and}\ \langle \mathcal{M}_x^{*,-} \rangle_t:= \langle \mathcal{M}_x^{\partial^{\mathrm{e}}  \mathcal{B}_x(N_*/2),-} \rangle_t.
\end{equation*}

\begin{definition}[suitable point]\label{def_suitable_point}
For $x\in \mathbb{Z}^d$, we say $x$ is unsuitable if the event $\mathsf{A}^{\mathcal{K}}_x\cup \mathsf{A}^{\mathcal{H}}_x$ happens, where 
	\begin{equation}
		\mathsf{A}^{\mathcal{K}}_x:= \Big\{\mathcal{K}_x^*\ge \tfrac{1}{2}\cref{const_small}^2\lambda_* N_*^{-\frac{d}{2}+1} k_*^{-3}2^{k_*}\Big\},
	\end{equation}
	\begin{equation}
		\mathsf{A}^{\mathcal{H}}_x := \Big\{\mathcal{H}_{x}^{*}\ge \cref{const_small}^2\lambda_* N_*^{-\frac{d}{2}+1} k_*^{-3}2^{k_*},\langle \mathcal{M}_x^{*,-} \rangle_\infty \le \cref{const_small}^4\lambda_*^2 N_*^{-d+2} k_*^{-40d}2^{2k_*}  \Big\}. 
	\end{equation}
	We also say $x$ is suitable if it is not unsuitable. 
\end{definition}


Recall that our assumption on $\lambda$ implies that $\lambda_*\ge \Cref{const_section_block}$. Therefore, by (\ref{newadd_2.30_2}) and (\ref{equation_def_C9}), we have: for any $x\in \mathbb{Z}^d$,
\begin{equation}\label{5.26}
	\mathbb{P}\big(\mathsf{A}^{\mathcal{K}}_x\big)\overset{(\ref{newadd_2.30_2})}{\le} e^{-\cref{const_average_2}2^{-d}\cref{const_small}^4\lambda_*^2k_*^{-6}2^{2k_*}}\overset{(\ref{equation_def_C9})}{\le } e^{-k_*^{-6}2^{2k_*}} \le  2^{-k_*^{70d}}.
\end{equation}
Since  $\mathcal{M}_{x,0}^{*,-}=\mathcal{K}_x^*$ and $\{\mathcal{H}_x^*>0\}\subset \{x\xleftrightarrow{\le 0} \partial^{\mathrm{e}}B_x(N_*/2)\}^c\subset \{\mathcal{M}_{x,\infty}^{*,-}=\mathcal{H}_x^*\}$, 
\begin{equation}\label{5.27}
	\begin{split}
		\mathbb{P}\Big[\mathsf{A}^{\mathcal{H}}_x\cap \big( \mathsf{A}^{\mathcal{K}}_x\big)^c\Big]\le &	\mathbb{P}\Big[\mathcal{M}_{x,\infty}^{*,-}-\mathcal{M}_{x,0}^{*,-}\ge \tfrac{1}{2}\cref{const_small}^2\lambda_* N_*^{-\frac{d}{2}+1} k_*^{-3}2^{k_*},\\
		&\ \ \ \ \ \langle \mathcal{M}_x^{*,-} \rangle_\infty \le \cref{const_small}^4\lambda_*^2 N_*^{-d+2} k_*^{-40d}2^{2k_*} \Big] \le 2^{-k_*^{70d}},
	\end{split}
\end{equation}
where we used Lemma \ref{lemma_EM} in the last inequality. By (\ref{5.26}) and (\ref{5.27}), we obtain
\begin{equation}\label{5.28}
	\mathbb{P}\big(x\ \text{is unsuitable}\big)\le 2^{-k_*^{70d}+1}.
\end{equation}

\begin{definition}[Suitable box]\label{def_suitable_box}
For any $x\in \mathbb{Z}^d$, we say a box $B_x(N_\diamond)$ is suitable if the proportion of unsuitable points in $B_x(N_\diamond)$ is at most $2^{-k_*^{60d}}$. Note that $2^{-k_*^{60d}}$ is also the maximal proportion of removed points in Definition \ref{def_excellent_box}.
\end{definition}

For any $x\in \mathbb{Z}^d$, it follows from (\ref{5.28}) that
\begin{equation}\label{5.29}
	\begin{split}
		\mathbb{P}\big[B_x(N_\diamond)\ \text{is not suitable}\big] \le& \frac{\mathbb{E}\big(\big|\{y\in B_x(N_\diamond):y\ \text{is unsuitable}  \}\big|\big)}{2^{-k_*^{60d}}|B_x(N_\diamond)|}\\
		\le &\frac{2^{-k_*^{70d}+1}|B_x(N_\diamond)|}{2^{-k_*^{60d}}|B_x(N_\diamond)|}\le 2^{-k_*^{50d}}. 
	\end{split}
\end{equation}

\begin{definition}[Nice box]
For any $x\in \mathbb{Z}^d$, we say the box $B_x(N_\diamond)$ is nice if it is excellent and suitable.
\end{definition}

As promised at the beginning, we define the event 
\begin{equation}\label{def_F}
	\mathsf{F}:= \bigcap_{y\in [-k_*^{10},k_*^{10}]^d\cap \mathbb{Z}^d} \big\{ B_{N_\diamond\cdot y}(N_\diamond)\ \text{is nice}\big\}\cap  \{\bm{0}\ \text{is a good point}\}. 
\end{equation}
Therefore, by (\ref{5.20}) and (\ref{5.29}), we have 
\begin{equation}\label{ineq_PF}
	\mathbb{P}\big( \mathsf{F}\big)\ge 2^{-k_*^{20d}}- \big|[-k_*^{10},k_*^{10}]^d\cap \mathbb{Z}^d\big|\cdot 2^{-k_*^{50d}}\ge  2^{-k_*^{30d}}. 
\end{equation}

\subsection{Brownian motions on the event $\mathsf{F}$}\label{subsection_stu_F}

The core of proving Proposition \ref{prop_1} is to show that on the event $\mathsf{F}$, with high probability, an independent Brownian motion on $\widetilde{\mathbb{Z}}^d$ starting from $\bm{0}$ will hit $\mathcal{C}^{-}_{\partial \mathcal{B}(N_*/4)}$ before exiting $\mathcal{B}(0.01N_*)$ (see Lemma \ref{lemma_F}). To achieve this, we need to show that when crossing a nice box, the
Brownian motion hits the negative clusters with a significant probability.

\begin{lemma}\label{lemma_escape}
	For any $3\le d\le 6$, suppose that $\lambda\in\mathfrak{X}_d$ satisfies the conditions given at the beginning of this section. Then on the event $\{B(N_\diamond)\ \text{is nice}\}$,
\begin{equation}\label{ineq_escape}
	\widetilde{\mathbb{P}}_z\Big[\tau_{\partial^{\mathrm{e}}B(40N_\diamond)} > \tau_{\mathcal{C}^{-}_{\partial^{\mathrm{e}}  \mathcal{B}(3N_*/8)}}\Big] \ge  k_*^{-4}, \ \ \forall z\in B(N_\diamond).
\end{equation}
\end{lemma}

Before proving Lemma \ref{lemma_escape}, we need some preparations as follows. Recall that $\widecheck{N}_*=k_*^{100d}2^{-0.5k_*}N_*$ and $\widecheck{N}_{j}= 2^{-j}\widecheck{N}_*$ for $j\ge 0$. Also recall $j_*$ in (\ref{5.1}).

For each good and suitable point $x$, it follows from Definitions \ref{def_good_point} and \ref{def_suitable_point} that the following events happen: 
\begin{equation}\label{5.33}
	\big\{\mathcal{B}_x(\widecheck{N}_{j_*})\xleftrightarrow{\le 0} \partial^{\mathrm{e}} \mathcal{B}_x(N_*/2)\big\}^c,
\end{equation}
\begin{equation}\label{5.34}
	\begin{split}
		\langle \mathcal{M}_x^{*,-} \rangle_\infty\overset{(\ref{new_2.17})}{=}&\sum\nolimits_{v\in \mathcal{C}^{-}_{\partial^{\mathrm{e}}  \mathcal{B}_x(N_*/2)} } \widetilde{\mathbb{P}}_x\Big[\tau_{\mathcal{C}^{-}_{\partial^{\mathrm{e}}  \mathcal{B}_x(N_*/2)}}=\tau_{v}<\infty \Big] \widetilde{G}(v,x)      \\
		\ge & \cref{const_small}^4\lambda_*^2 N_*^{-d+2} k_*^{-40d}2^{2k_*}. 
	\end{split}	 
\end{equation}
  In fact, $\langle \mathcal{M}_x^{*,-} \rangle_\infty$ can be decomposed as follows: 
\begin{equation}\label{5.35}
	\begin{split}
		\langle \mathcal{M}_x^{*,-} \rangle_\infty= \sum\nolimits_{j\ge 0} \langle \mathcal{M}_x^{*,-} \rangle_\infty^{j} ,
	\end{split}
\end{equation}
	where we define 
\begin{equation*}
	\langle \mathcal{M}_x^{*,-} \rangle_\infty^{0}:=\sum_{v\in \mathcal{C}^{-}_{\partial^{\mathrm{e}}  \mathcal{B}_x(N_*/2)}\setminus \widetilde{B}_x(\widecheck{N}_*) } \widetilde{\mathbb{P}}_x\Big[\tau_{\mathcal{C}^{-}_{\partial^{\mathrm{e}}  \mathcal{B}_x(N_*/2)}}=\tau_{v}<\infty\Big] \widetilde{G}(v,x),
\end{equation*}
and for each $j\ge 1$, 
\begin{equation*}
	\langle \mathcal{M}_x^{*,-} \rangle_\infty^{j}:= \sum_{v\in \mathcal{C}^{-}_{\partial^{\mathrm{e}}  \mathcal{B}_x(N_*/2)}\cap [\widetilde{B}_x(\widecheck{N}_{j-1})\setminus \widetilde{B}_x(\widecheck{N}_{j})] } \widetilde{\mathbb{P}}_x\Big[\tau_{\mathcal{C}^{-}_{\partial^{\mathrm{e}}  \mathcal{B}_x(N_*/2)}}=\tau_{v}<\infty\Big] \widetilde{G}(v,x). 
\end{equation*}
Note that (\ref{5.33}) implies that 
\begin{equation}\label{new_add_5.38}
		\langle \mathcal{M}_x^{*,-} \rangle_\infty^{j} = 0,\ \ \  \forall j\ge j_*+1.  
\end{equation}

\begin{lemma}\label{lemma_5.10}
	When $3\le d\le 6$, for any good and suitable point $x\in \mathbb{Z}^d$, there exists a (random) integer $j\in [1,j_*]$ such that 
	\begin{equation}\label{5.38}
		\begin{split}
			\langle \mathcal{M}_x^{*,-} \rangle_\infty^{j}
			\ge \tfrac{1}{4}j^{-2}\cref{const_small}^4\lambda_*^2 N_*^{-d+2} k_*^{-40d}2^{2k_*}. 
		\end{split}
	\end{equation}
\end{lemma}
\begin{proof}
By (\ref{bound_green}), $\lambda_*\ge \Cref{const_section_block}$, (\ref{equation_def_C9}) and $d\le 6$, we have 
		\begin{equation*}\label{5.36}
		\begin{split}
			\langle \mathcal{M}_x^{*,-} \rangle_\infty^{0}\overset{(\ref{bound_green})}{\le } 	\Cref{const_green_1}(k_*^{100d}2^{-0.5k_*}N_*)^{2-d}\overset{(\ref{equation_def_C9})}{\le } \tfrac{1}{2}\cref{const_small}^4\lambda_*^2 N_*^{-d+2} k_*^{-40d}2^{2k_*}. 
		\end{split}
	\end{equation*}
	Combined with (\ref{5.34}) and (\ref{new_add_5.38}), it yields that 
\begin{equation*}
	\sum\nolimits_{1\le j\le j_*} \langle \mathcal{M}_x^{*,-} \rangle_\infty^{j} \ge \tfrac{1}{2}\cref{const_small}^4\lambda_*^2 N_*^{-d+2} k_*^{-40d}2^{2k_*}.
\end{equation*}
		As a result, by $\sum_{n\ge 1}n^{-2}<2$ there exists $j\in [1,j_*]$ such that (\ref{5.38}) holds.
\end{proof}

We denote the event 
\begin{equation*}
	\mathsf{G}:= \Big\{\widetilde{S}_\cdot\ \text{visits at least}\ k_*^{-3}2^{-k_*}\widehat{N}_*^2\ \text{good and suitable points before}\ \tau_{\partial^{\mathrm{e}}B(20N_\diamond)} \Big\}. 
\end{equation*}
Recalling Definitions \ref{def_excellent_box} and \ref{def_suitable_box}, we know that on $\{B(N_\diamond)\ \text{is a nice box}\}$, 
\begin{equation}\label{ineq_PG}
\widetilde{\mathbb{P}}_z\big(\mathsf{G}\big)\ge k_*^{-3}, \ \ \forall z\in B(N_\diamond).
\end{equation}
 We say $x$ is $j$-nice if (\ref{5.38}) holds. By Lemma \ref{lemma_5.10}, each good and suitable point is $j$-nice for some $j\in [1,j_*]$. As a result, we have 
 \begin{equation}\label{new_5.39}
	\mathsf{G} \subset \cup_{1\le j\le j_*}\mathsf{G}_j, 
\end{equation}
where we define 
\begin{equation}\label{event_Gj}
	\mathsf{G}_j := \Big\{\widetilde{S}_\cdot\ \text{visits at least}\ \tfrac{1}{2}j^{-2}k_*^{-3}2^{-k_*}\widehat{N}_*^2\ \  j\text{-nice points before}\ \tau_{\partial^{\mathrm{e}}B(20N_\diamond)} \Big\}. 
\end{equation}
By (\ref{ineq_PG}) and (\ref{new_5.39}), to get (\ref{ineq_escape}), it remains to upper-bound the probabilities
 $\widetilde{\mathbb{P}}_z\big(\mathsf{G}_j,\tau_{\partial^{\mathrm{e}}B(40N_\diamond)}< \tau_{\mathcal{C}^{-}_{\partial^{\mathrm{e}}  \mathcal{B}(3N_*/8)}}\big)$ for all $z\in B(N_\diamond) $ and $1\le j\le j_*$. To achieve this, we need the following lemma.

\begin{lemma}\label{lemma_green}
	For any $d\ge 3$, there exists $\cl\label{const_hit}(d)\in (0,1)$ such that for any $R\ge 1$ and $D\subset \widetilde{B}(2R)\setminus \widetilde{B}(R)$, 
	\begin{equation}\label{new_add_5.42}
	\widetilde{\mathbb{P}}_{\bm{0}}\big[\tau_{D}< \tau_{\partial B(4R)}\big]\ge \cref{const_hit} \widetilde{\mathbb{P}}_{\bm{0}}\big[\tau_{D}<\infty\big]. 
	\end{equation}
\end{lemma}
\begin{proof}
	By the strong Markov property, we have 
	\begin{equation}\label{add_5.46}
		\begin{split}
			&\widetilde{\mathbb{P}}_{\bm{0}}\big[\tau_{D}< \tau_{\partial B(4R)}\big] \\ \ge & \widetilde{\mathbb{P}}_{\bm{0}}\big[\tau_{D}> \tau_{\partial B(3R)}\big] \min_{y\in \partial B(3R)} \widetilde{\mathbb{P}}_{y}\big[\tau_{D}< \tau_{\partial B(4R)}\big]+  \widetilde{\mathbb{P}}_{\bm{0}}\big[\tau_{D}< \tau_{\partial B(3R)}\big],
		\end{split}
	\end{equation}
		\begin{equation}\label{add_5.47}
		\begin{split}
			\widetilde{\mathbb{P}}_{\bm{0}}\big[\tau_{D}< \infty\big]\le  \widetilde{\mathbb{P}}_{\bm{0}}\big[\tau_{D}> \tau_{\partial B(3R)}\big] \max_{y\in \partial B(3R)} \widetilde{\mathbb{P}}_{y}\big[\tau_{D}< \infty\big]+  \widetilde{\mathbb{P}}_{\bm{0}}\big[\tau_{D}< \tau_{\partial B(3R)}\big]. 
		\end{split}
	\end{equation}
	According to the potential theory of random walk (see e.g. \cite[Proposition 6.5.1]{lawler2010random}), we know that for any $y\in \partial B(3R)$, both $\widetilde{\mathbb{P}}_{y}\big[\tau_{D}< \tau_{\partial B(4R)}\big]$ and $\widetilde{\mathbb{P}}_{y}\big[\tau_{D}< \infty\big]$ are of the same order as the product of $R^{2-d}$ and the capacity of $D$. Therefore, there exists a constant $c(d)\in (0,1)$ such that 
	\begin{equation*}
		\min_{y\in \partial B(3R)} \widetilde{\mathbb{P}}_{y}\big[\tau_{D}< \tau_{\partial B(4R)}\big] \ge c\max_{y\in \partial B(3R)} \widetilde{\mathbb{P}}_{y}\big[\tau_{D}< \infty\big]. 
	\end{equation*}	
Combined with (\ref{add_5.46}) and (\ref{add_5.47}), it concludes this lemma. 
\end{proof}


We present further preparations as follows. Suppose that $x\in B(N_\diamond)$ is a $j$-nice point for some $j\in [1,j_*]$. We abbreviate 
\begin{equation}
	\widecheck{\mathcal{C}}_x^j:= \mathcal{C}^{-}_{\partial^{\mathrm{e}}  \mathcal{B}_x(N_*/2)}\cap [\widetilde{B}_x(\widecheck{N}_{j-1})\setminus \widetilde{B}_x(\widecheck{N}_{j})]. 
\end{equation}
By (\ref{bound_green}) and Lemma \ref{lemma_green}, one has 
\begin{equation*}
	\begin{split}
	\langle	\mathcal{M}_x^{*,-} \rangle_\infty^{j} \le& \Cref{const_green_1}\widecheck{N}_{j}^{2-d}\widetilde{\mathbb{P}}_x\big[\tau_{\widecheck{\mathcal{C}}_x^j}<\infty\big]
	\le \Cref{const_green_1} \cref{const_hit}^{-1}\widecheck{N}_{j}^{2-d}\widetilde{\mathbb{P}}_x\big[\tau_{\widecheck{\mathcal{C}}_x^j}<\tau_{\partial B_x(2\widecheck{N}_{j-1})}\big].
	\end{split}
\end{equation*}  
Combined with (\ref{5.38}), it yields that (recalling that $\widecheck{N}_{j}=k_*^{100d}2^{-j-0.5k_*}N_*$)
\begin{equation}\label{newadd_5.49}
	\widetilde{\mathbb{P}}_x\big[\tau_{\widecheck{\mathcal{C}}_x^j}<\tau_{\partial B_x(2\widecheck{N}_{j-1})}\big] \ge  \widecheck{\mathbb{I}}_j:=\cref{const_section5_1} j^{-2}2^{-j(d-2)} \lambda_*^{2}k_*^{60d}2^{(\frac{6-d}{2}) k_*}, 
\end{equation}
where $\cl\label{const_section5_1}:=2^{-d}\Cref{const_green_1}^{-1}  \cref{const_small}^4\cref{const_hit} $. For sufficiently small $j\ge 1$, $\widecheck{\mathbb{I}}_j$ could potentially be larger than $1$, thereby violating (\ref{newadd_5.49}). However, in the context of our proof, such a ``contradiction'' will provide further restriction on $j$ and thus facilitate the proof. Precisely, it follows from $\widecheck{\mathbb{I}}_j\le 1$ that $j^{2}2^{j(d-2)}\ge \cref{const_section5_1}\lambda_*^{2}k_*^{60d}2^{(\frac{6-d}{2}) k_*}$ and thus,
\begin{equation}\label{bound_check_Nj}
	\widecheck{N}_j=2^{-j}k_*^{100d}2^{-0.5k_*}N_*\le (\lambda_*^{-1}2^{-k_*})^{1/d}N_*.
\end{equation}
In light of (\ref{newadd_5.49}), it would be useful to bound the exit time from a box. According to \cite[Proposition 2.4.5]{lawler2010random}, there exists a constant $\Cl\label{const_ldp}(d)>1$ such that
\begin{equation}\label{ineq_ldp}
	\widetilde{\mathbb{P}}_x\big(\tau_{\partial B_x(2\widecheck{N}_{j-1})}\ge  \Cref{const_ldp}j^2k_*^2\widecheck{N}_{j}^2\big) \le e^{-j^2k_*^2}.  
\end{equation}

The following lemma is crucial for this section.

\begin{lemma}\label{lemma5.12}
	Recall the event $\mathsf{G}_j$ in (\ref{event_Gj}). For any $z\in B(N_\diamond)$, 
	\begin{equation}
		\sum\nolimits_{1\le j\le j_*}\widetilde{\mathbb{P}}_z\Big[\mathsf{G}_j,\tau_{\partial^{\mathrm{e}}B(40N_\diamond)}< \tau_{\mathcal{C}^{-}_{\partial^{\mathrm{e}}  \mathcal{B}(3N_*/8)}}\Big]\le e^{-k_*}. 
	\end{equation}
\end{lemma}
\begin{proof}
	We arbitrarily take $j\in [1,j_*]$. Let $\{\widetilde{S}_t\}_{t\ge 0}$ be a Brownian motion on $\widetilde{\mathbb{Z}}^d$, which starts from $z$ and is independent of the GFF $\widetilde{\phi}$. We define a sequence of stopping times as follows. We set $\tau_0^{+}:=0$. For any $l\ge 1$, we define
	\begin{itemize}
		\item  $\tau_l^{-}:= \inf\big\{t\ge \tau_{l-1}^{+}: \widetilde{S}_t\ \text{is}\ j \text{-nice} \big\}$ and  let $x_l:= \widetilde{S}_{\tau_l^{-}}$;

		\item  $\tau_l^{+}:=  \inf\big\{t\ge \tau_l^{-}:  \widetilde{S}_t \in \partial B_{x_l}(2\widecheck{N}_{j-1})  \big\}$. 
		
	\end{itemize}
	Let $l_\ddagger$ be the smallest $l\ge 1$ such that $\tau_l^{-}\le \tau_{\widecheck{\mathcal{C}}_{x_l}^j} \le \tau_l^{+}$. We claim that 
	\begin{equation}\label{claim_5.58}
		\mathsf{G}_j\cap \big\{\tau_{\partial^{\mathrm{e}}B(40N_\diamond)} < \tau_{\mathcal{C}^{-}_{\partial^{\mathrm{e}}  \mathcal{B}(3N_*/8)}}\big\} \subset \big\{l_\ddagger\ge \mathcal{J}_* \big\} \cup \mathsf{A}_*,
	\end{equation}
	where $\mathcal{J}_*:=\Cref{const_ldp}^{-1}j^{-4}2^{2j}k_*^{-300d}$ and $\mathsf{A}_*:=\bigcup_{l=1}^{\mathcal{J}_*}\big\{\tau_l^{+}-\tau_l^{-}\ge \Cref{const_ldp}j^2k_*^2\widecheck{N}_{j}^2   \big\}$. In fact, on $\mathsf{G}_j\cap \big\{\tau_{\partial^{\mathrm{e}}B(40N_\diamond)} < \tau_{\mathcal{C}^{-}_{\partial^{\mathrm{e}}  \mathcal{B}(3N_*/8)}}\big\}\cap \mathsf{A}_*^c$, the number of $j$-nice points visited by $\widetilde{S}_t$ for $0\le t\le \tau_{\mathcal{J}_*}^{+}$ is at most (recall that $\widecheck{N}_{j}=k_*^{100d}2^{-j-0.5k_*}N_*$ and $\widehat{N}_*=k_*^{-100}N_*$)
	\begin{equation}\label{cal}
		\begin{split}
				 \sum\nolimits_{1\le l\le \mathcal{J}_*}(\tau_l^{+}-\tau_l^{-})\le  \Cref{const_ldp}^{-1}j^{-4}2^{2j}k_*^{-300d}\cdot \Cref{const_ldp}j^2k_*^2\widecheck{N}_{j}^2 <\tfrac{1}{2}j^{-2}k_*^{-3}2^{-k_*}\widehat{N}_*^2.  
		\end{split}
	\end{equation}
	However, since $\mathsf{G}_j\cap \big\{\tau_{\partial^{\mathrm{e}}B(40N_\diamond)} < \tau_{\mathcal{C}^{-}_{\partial^{\mathrm{e}}  \mathcal{B}(3N_*/8)}}\big\}$ happens, we know that $\widetilde{S}_t$ will visit at least  $\tfrac{1}{2}j^{-2}k_*^{-3}2^{-k_*}\widehat{N}_*^2$ $j$-nice points before hitting $\mathcal{C}^{-}_{\partial^{\mathrm{e}}  \mathcal{B}(3N_*/8)}$. Combined with (\ref{cal}), this implies that for every $1\le l\le \mathcal{J}_*$, $\widetilde{S}_t$ for $\tau_l^{-}\le t\le \tau_l^{+}$ does not intersect $\mathcal{C}^{-}_{\partial^{\mathrm{e}}  \mathcal{B}(3N_*/8)}$, and thus does not intersect $\tau_{\widecheck{\mathcal{C}}_{x_l}^j}$ (otherwise, since $B_{x_l}(\widecheck{N}_{j-1})\subset \mathcal{B}(3N_*/8)$ and $\partial^{\mathrm{e}}  \mathcal{B}_{x_l}(N_*/2)\subset \mathbb{Z}^d\setminus \mathcal{B}(3N_*/8)$, both of which follow from $x_l\in B(N_\diamond)$, we have $\widecheck{\mathcal{C}}_{x_l}^j\subset \mathcal{C}^{-}_{\partial^{\mathrm{e}}  \mathcal{B}_{x_l}(N_*/2)}\cap \widetilde{B}_{x_l}(\widecheck{N}_{j-1})\subset\mathcal{C}^{-}_{\partial^{\mathrm{e}}  \mathcal{B}(3N_*/8)} $ and hence $\widetilde{S}_t$ must intersect $\mathcal{C}^{-}_{\partial^{\mathrm{e}}  \mathcal{B}(3N_*/8)}$). In conclusion, we obtain that 
	\begin{equation}
		\mathsf{G}_j\cap \big\{\tau_{\partial^{\mathrm{e}}B(40N_\diamond)} < \tau_{\mathcal{C}^{-}_{\partial^{\mathrm{e}}  \mathcal{B}(3N_*/8)}}\big\}\cap \mathsf{A}_*^c \subset \big\{l_\ddagger\ge \mathcal{J}_* \big\},
	\end{equation}
	which implies our claim (\ref{claim_5.58}).


	By (\ref{ineq_ldp}), (\ref{claim_5.58}) and $\mathcal{J}_*\le 2^{2j}<e^{\frac{1}{2}j^2k_*^2}$, we have 
	\begin{equation}\label{5.56}
		\begin{split}
			\widetilde{\mathbb{P}}_z\big(\mathsf{G}_j,\tau_{\partial^{\mathrm{e}}B(40N_\diamond)}< \tau_{\mathcal{C}^{-}_{\partial^{\mathrm{e}}  \mathcal{B}(3N_*/8)}}\big)\le \widetilde{\mathbb{P}}_z(l_\ddagger\ge \mathcal{J}_* )+e^{-\frac{1}{2}j^2k_*^2}. 
		\end{split}
	\end{equation}
	Moreover, on the event $\{l_\ddagger\ge \mathcal{J}_*\}$, for any $1\le l\le \mathcal{J}_*$, we know that the Brownian motion $\widetilde{S}_t$ starting from $x_l$ does not hit $\widecheck{\mathcal{C}}_{x_l}^j$ before reaching $\partial B_{x_l}(2\widecheck{N}_{j-1})$. Thus, by the strong Markov property and (\ref{newadd_5.49}), we have: for any $1\le j\le j_*$,
	\begin{equation}\label{5.57}
		\widetilde{\mathbb{P}}_z(l_\ddagger\ge \mathcal{J}_* ) \le \big(1-\widecheck{\mathbb{I}}_{j}\big)^{\mathcal{J}_*} \le e^{-\widecheck{\mathbb{I}}_{j}\mathcal{J}_*},
	\end{equation}
	 where $\widecheck{\mathbb{I}}_{j}\mathcal{J}_*=\cref{const_section5_1}\Cref{const_ldp}^{-1} j^{-6}2^{j(4-d)}k_*^{-240d}2^{(\frac{6-d}{2}) k_*}\lambda_*^2$. See Figure \ref{fig_block_rw} for an illustration.

	 	\begin{figure}[h!]
	 	\centering
	 	\includegraphics[width=0.37\textwidth]{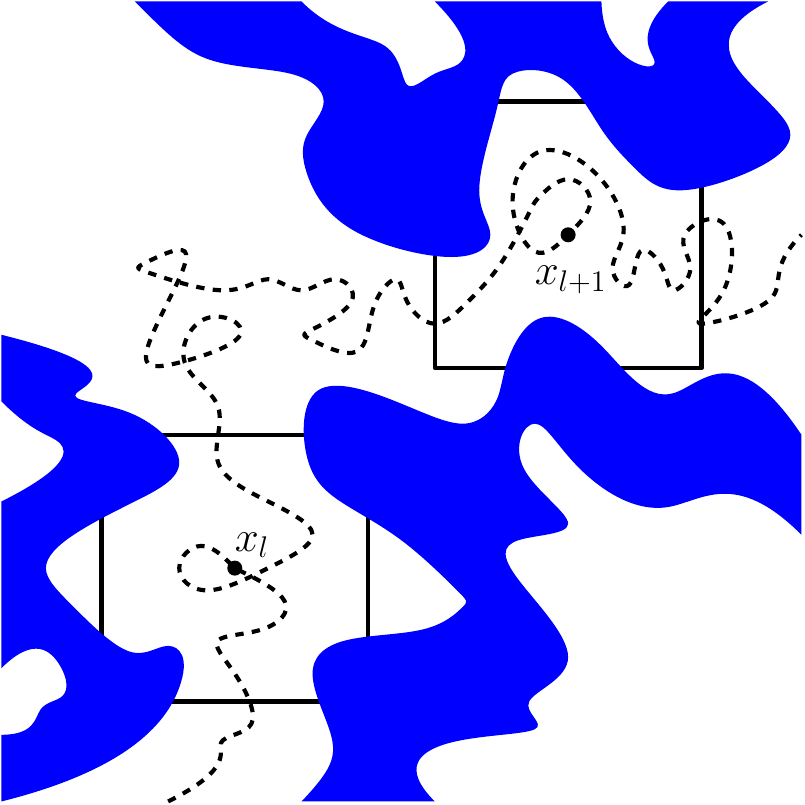}
	 	\caption{An illustration for the proof of Lemma \ref{lemma5.12}. Every time the Brownian motion (i.e. the dashed curve) starts from a $j$-nice point $x_l$, it has a significant probability of hitting the negative cluster $\mathcal{C}^{-}_{\partial^{\mathrm{e}}  \mathcal{B}(3N_*/8)}$ (i.e. the blue area) before exiting $B_{x_l}(2\widecheck{N}_{j-1})$. Moreover, with high probability the number of points visited by the Brownian motion upon exiting $B_{x_l}(2\widecheck{N}_{j-1})$ is $O(j^2k_{*}^2\widecheck{N}_{j}^2)$.   }\label{fig_block_rw}.  
	 \end{figure}

	  In what follows, we discuss the lower bounds for $\widecheck{\mathbb{I}}_j\mathcal{J}_*$ separately for the cases when $d\in \{3,4\}$, $d=5$ and $d=6$. Let $\cl\label{const_section5_4}:=(\cref{const_section5_1}\Cref{const_crossing}^{-1}\Cref{const_ldp}^{-1})^{10d}$.

	 \textbf{When $d\in \{3,4\}$:} In this case, since $2^{j(4-d)}\ge 1$ and $2^{(\frac{6-d}{2}) k_*}\ge k_*^{300d}$, we derive from (\ref{cal1}) and (\ref{cal2}) that
	 \begin{equation}\label{5.60}
	 	\widecheck{\mathbb{I}}_{j}\mathcal{J}_*\ge  \cref{const_section5_4}j^2k_*^{10} \log_2^{10}(\lambda_*).
	 \end{equation}

	 \textbf{When $d=5$:} For $d\in \{5,6\}$, by (\ref{cal1}) and (\ref{cal2}), we have 
	 \begin{equation}\label{5.61}
	 	\begin{split}
	 		&j^{-6}2^{j(4-d)}k_*^{-240d}2^{(\frac{6-d}{2}) k_*}\lambda_*^2 \\
	 		\ge& \Cref{const_crossing}^{-2d}j^2k_*^{-1000d}2^{(\frac{6-d}{d-2})k_*}  \lambda_*^{\frac{2(6-d)}{d-2}} \Big[\frac{\lambda_*}{\lambda(\widecheck{N}_{j})}\Big]^{\frac{2(d-4)}{d-2}}\log_2^{-8}(\lambda_*). 
	 	\end{split}
	 \end{equation}
    In the particular case when $d = 5$, since $2^{(\frac{6-d}{d-2})k_*}\ge k_*^{2000d}$ and $\lambda_*= \lambda(\widecheck{N}_{j})$, we have 
    \begin{equation}\label{5.62}
    	\widecheck{\mathbb{I}}_{j}\mathcal{J}_*\ge \cref{const_section5_4}j^2 k_*^{10}\lambda_*^{\frac{2}{3}}\log_2^{-8}(\lambda_*)\ge  \cref{const_section5_4}j^2k_*^{10} \log_2^{10}(\lambda_*).  
    \end{equation}

	  \textbf{When $d=6$:} By (\ref{4.29}), (\ref{bound_check_Nj}) and (\ref{5.61}), we have
	  \begin{equation}\label{5.63}
	  	\begin{split}
	  		\widecheck{\mathbb{I}}_{j}\mathcal{J}_*\overset{(\ref{5.61})}{\ge} & \cref{const_section5_4} j^2k_*^{-1000d}\Big[\frac{\lambda_*}{\lambda(\widecheck{N}_{j})}\Big]^{-1}\log_2^{-8}(\lambda_*)\\
	  	\overset{(\ref{bound_check_Nj})}{\ge } &\cref{const_section5_4} j^2k_*^{-1000d}\Big[\frac{\lambda_*}{\lambda\big((\lambda_*^{-1}2^{-k_*})^{1/d}N_*\big)}\Big]^{-1}\log_2^{-8}(\lambda_*)\\
	  	\overset{(\ref{4.29})}{\ge } & \cref{const_section5_4} j^2k_*^{-1000d} \log_2^{2000d}(\lambda_*2^{k_*})\log_2^{-8}(\lambda_*)\\
	  	\ge& \cref{const_section5_4}j^2k_*^{10}\log_2^{10}(\lambda_*). 
	  	\end{split}
	  \end{equation}
In conclusion, for any $3\le d\le 6$, combining (\ref{5.60}), (\ref{5.62}) and (\ref{5.63}), we have $\widecheck{\mathbb{I}}_{j}\mathcal{J}_*\ge \cref{const_section5_4}j^2k_*^{10}\log_2^{10}(\lambda_*)$ for all $ 1\le j\le j_*$. Combined with (\ref{5.57}), it yields that 
	\begin{equation}\label{5.66}
		\widetilde{\mathbb{P}}_z(l_\ddagger\ge \mathcal{J}_* ) \le e^{-\cref{const_section5_4}j^2k_*^{10}\log_2^{10}(\lambda_*)}.
	\end{equation}
	By (\ref{equation_def_C9}), (\ref{5.56}), (\ref{5.66}) and $\lambda_*\ge \Cref{const_section_block}$, we conclude this lemma: 
	\begin{equation}
		\begin{split}			&\sum\nolimits_{1\le j \le j_*}\widetilde{\mathbb{P}}_z\big(\mathsf{G}_j,\tau_{\partial^{\mathrm{e}}B(40N_\diamond)}< \tau_{\mathcal{C}^{-}_{\partial^{\mathrm{e}}  \mathcal{B}(3N_*/8)}}\big)\\
		\overset{(\ref{5.56}),(\ref{5.66})}{\le } &  \sum\nolimits_{j\ge 1}   e^{-\cref{const_section5_4}j^2k_*^{10}\log_2^{10}(\lambda_*)} +  e^{-\frac{1}{2}j^2k_*^2}   \\
		\overset{(\lambda_*\ge \Cref{const_section_block})}{\le }	 & 10\big(e^{-\cref{const_section5_4}k_*^{10}\log_2^{10}(\Cref{const_section_block})} +  e^{-\frac{1}{2}k_*^2} \big) \overset{(\ref{equation_def_C9})}{\le }10\big(e^{-k_*^{10}} +  e^{-\frac{1}{2}k_*^2} \big) <e^{-k_*} .    \qedhere  
		\end{split}
	\end{equation}
	\end{proof}
	

	Based on Lemma \ref{lemma5.12}, it is now straightforward to prove Lemma \ref{lemma_escape}. 
	\begin{proof}[Proof of Lemma \ref{lemma_escape}]
		 
		 For any $z\in B(N_\diamond)$, by (\ref{new_5.39}) we have 
		 \begin{equation*}
		 	\begin{split}
		 			\widetilde{\mathbb{P}}_z\big(\tau_{\partial^{\mathrm{e}}B(40N_\diamond)} > \tau_{\mathcal{C}^{-}_{\partial^{\mathrm{e}}  \mathcal{B}(3N_*/8)}}\big)
		 	\ge \widetilde{\mathbb{P}}_z\big(\mathsf{G}\big)-\sum_{1\le j\le j_*}\widetilde{\mathbb{P}}_z\big(\mathsf{G}_j,\tau_{\partial^{\mathrm{e}}B(40\widecheck{N}_\diamond)} > \tau_{\mathcal{C}^{-}_{\partial^{\mathrm{e}}  \mathcal{B}(3N_*/8)}}\big). 
		 	\end{split}
		 \end{equation*}
		 Combined with (\ref{ineq_PG}) and Lemma \ref{lemma5.12}, it concludes the desired bound:
		\begin{equation}
			\begin{split}
					\widetilde{\mathbb{P}}_z\big(\tau_{\partial^{\mathrm{e}}B(40N_\diamond)} > \tau_{\mathcal{C}^{-}_{\partial^{\mathrm{e}}  \mathcal{B}(3N_*/8)}}\big)
			\ge k_*^{-3}-e^{-k_*}>k_*^{-4}.  \qedhere
			\end{split}
		\end{equation}
	\end{proof}

\begin{lemma}\label{lemma_F}
	With the same condition as in Lemma \ref{lemma_escape}, on the event $\mathsf{F}$, 
	\begin{equation}\label{claim_F}
		\widetilde{\mathbb{P}}_{\bm{0}}\Big[\tau_{\partial \mathcal{B}(0.01N_*)}< \tau_{\mathcal{C}^-_{\partial^{\mathrm{e}} \mathcal{B}(N_*/4)}} \Big]\le e^{-k_*^{5.5}}.  
	\end{equation}
\end{lemma}
\begin{proof}
Before delving into the details of the proof, one may see Figure \ref{fig_cg} for an illustration. We define a sequence of stopping times as follows. For any $l\ge 1$, let 
\begin{equation}
	\tau_l^-:= \inf\big\{t\ge 0: \widetilde{S}_t\in \partial \mathcal{B}(2l k_*^{0.1}N_\diamond) \big\},
\end{equation}
\begin{equation}
	\tau_l^+:= \inf\Big\{t\ge \tau_l^-: \widetilde{S}_t\in \partial \mathcal{B}\big((2l-1) k_*^{0.1}N_\diamond\big)\cup \partial \mathcal{B}\big((2l+1) k_*^{0.1}N_\diamond\big) \Big\}. 
\end{equation}
For the Brownian motion $\widetilde{S}_\cdot\sim \widetilde{\mathbb{P}}_{\bm{0}}$, one has 
\begin{equation*}
	0<\tau_1^-<\tau_1^+< \tau_2^-<\tau_2^+<...<\tau_{k_*^{9.8}}^-<\tau_{k_*^{9.8}}^+< \tau_{\partial \mathcal{B}(0.01N_*)}. 
\end{equation*}
This implies that on the event $\big\{\tau_{\partial \mathcal{B}(0.01N_*)}< \tau_{\mathcal{C}^-_{\partial \mathcal{B}(N_*/4)}}\big\}$, the Brownian motion $\widetilde{S}_\cdot$ cannot hit $\tau_{\mathcal{C}^-_{\partial \mathcal{B}(N_*/4)}}$ during every time interval $[\tau_l^-,\tau_l^+]$ for $1\le l\le k_*^{9.8}$. As a result, by the strong Markov property, we have 
\begin{equation}\label{5.39}
	\begin{split}
		\widetilde{\mathbb{P}}_{\bm{0}}\Big[\tau_{\partial \mathcal{B}(0.01N_*)}< \tau_{\mathcal{C}^-_{\partial \mathcal{B}(N_*/4)}} \Big] \le \prod_{l=1}^{k_*^{9.8}}\max_{x\in \partial \mathcal{B}(2l k_*^{0.1}N_\diamond)} \widetilde{\mathbb{P}}_{x}\Big[\tau_{\partial B_x(200N_\diamond)}<\tau_{\mathcal{C}^-_{\partial \mathcal{B}(N_*/4)}} \Big]. 
	\end{split}
\end{equation}
On the event $\mathsf{F}$ (recalling (\ref{def_F})), for any $1\le l\le k_*^{9.8}$ and $x\in \partial \mathcal{B}(2l k_*^{0.1}N_\diamond)$, there exists $x'\in \mathbb{Z}^d$ such that $B_{x'}(N_\diamond)$ is a nice box containing $x$. Since $x\in \mathcal{B}(0.01N_*)$ and $|x-x'|\le d^{\frac{1}{2}}N_\diamond$, we have $B_{x'}(40N_\diamond)\subset B_x(200N_\diamond)$ and for any $v\in \widetilde{B}_x(200N_\diamond)$, 
\begin{equation*}
  \big\{ v\xleftrightarrow{\le 0} \partial^{\mathrm{e}} \mathcal{B}_{x'}(3N_*/8) \big\}  \subset   \big\{ v\xleftrightarrow{\le 0} \partial^{\mathrm{e}} \mathcal{B}(N_*/4) \big\}, 
\end{equation*}
which implies the following inclusion: 
\begin{equation}\label{inclusion_5.59}
	 \mathcal{C}^{-}_{\partial^{\mathrm{e}} \mathcal{B}_{x'}(3N_*/8)}\cap \widetilde{B}_x(200N_\diamond) \subset \mathcal{C}^{-}_{\partial^{\mathrm{e}} \mathcal{B}(N_*/4)}. 
\end{equation}
Note that $\tau_{\partial B_x(200 N_\diamond)} \ge \tau_{\partial^e B_{x'}(40N_\diamond)}$ because $B_{x'}(40N_\diamond)\subset B_x(200N_\diamond)$. Therefore, by (\ref{inclusion_5.59}) and Lemma \ref{lemma_escape}, we obtain 
\begin{equation*}
	\begin{split}
		\widetilde{\mathbb{P}}_{x}\Big[\tau_{\partial B_x(200N_\diamond)}<\tau_{\mathcal{C}^-_{\partial^{\mathrm{e}} \mathcal{B}(N_*/4)}} \Big] 
		\le  \widetilde{\mathbb{P}}_{x}\Big[\tau_{\partial^{\mathrm{e}} B_{x'}(40N_\diamond)}<\tau_{\mathcal{C}^{-}_{\partial^{\mathrm{e}} \mathcal{B}_{x'}(3N_*/8)}} \Big] \overset{(\text{Lemma}\ \ref{lemma_escape})}{\le } 1-k_*^{-4}.  
	\end{split}
\end{equation*}
Combined with (\ref{5.39}), it concludes this lemma:
\begin{equation}
	\begin{split}
			\widetilde{\mathbb{P}}_{\bm{0}}\Big[\tau_{\partial \mathcal{B}(0.01N_*)}< \tau_{\mathcal{C}^-_{\partial^{\mathrm{e}} \mathcal{B}(N_*/4)}} \Big] \le (1-k_*^{-4})^{k_*^{9.8}}\le e^{-k_*^{5.5}}.   \qedhere  
	\end{split}
\end{equation}
\end{proof}

 	\begin{figure}[h!]
	\centering
	\includegraphics[width=0.4\textwidth]{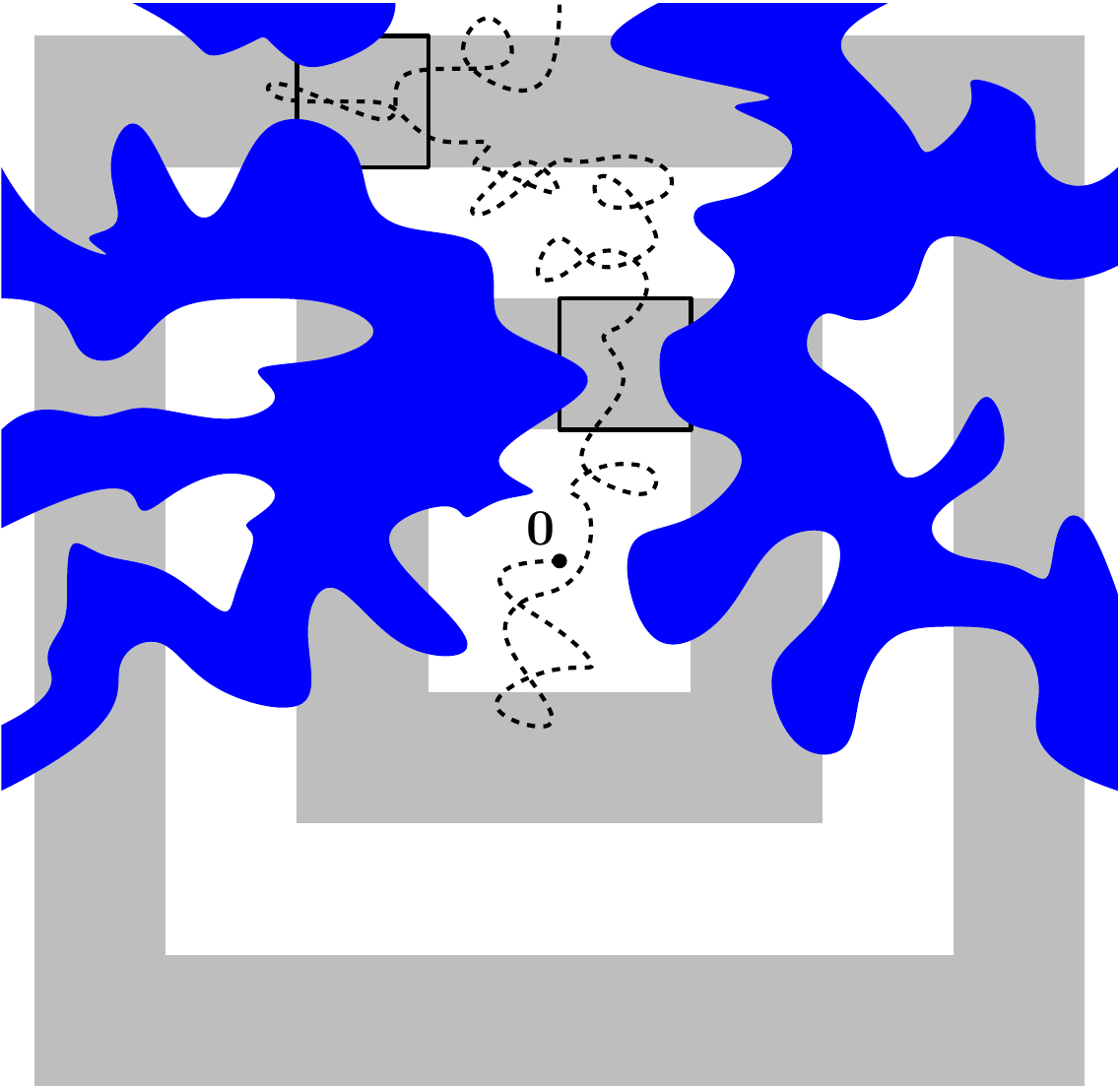}
	\caption{An illustration for the proof of Lemma \ref{lemma_F}. For the Brownian motion starting from the origin (i.e. the dashed curve), before exiting $\mathcal{B}(0.01N_*)$ it must cross numerous annuli of the form $\mathcal{B}\big((2l+1) k_*^{0.1}N_\diamond\big)\setminus \mathcal{B}\big((2l-1) k_*^{0.1}N_\diamond\big)$ for $l\ge 1$ (i.e. the gray area). In each crossing, the event $\mathsf{F}$ ensures that the Brownian motion must encounter some nice box (i.e. the box with a black boundary) and thus has a significant probability of hitting the negative cluster $\mathcal{C}^-_{\partial^{\mathrm{e}} \mathcal{B}(N_*/4)}$ (i.e. the blue area).  }\label{fig_cg}
\end{figure}

\subsection{Proof of Proposition \ref{prop_1}}  \label{section_prove_prop1}
 For any $v\in \widetilde{\mathbb{Z}}^d$, we abbreviate that 
\begin{equation}
	\mathcal{H}_{v}':= \mathcal{H}_{v}\big( \partial^{\mathrm{e}} \mathcal{B}(N_*/2) , \mathcal{C}^{-}_{\partial^{\mathrm{e}} \mathcal{B}(N_*/2)}\big),
\end{equation} 
\begin{equation}
	\mathcal{H}_{v}'':= \mathcal{H}_{v}\big( \partial^{\mathrm{e}} \mathcal{B}(N_*/4) , \mathcal{C}^{-}_{\partial^{\mathrm{e}} \mathcal{B}(N_*/4)}\big).
\end{equation} 
For any $1\le n\le \tfrac{N_*}{4}$, we denote 
\begin{equation}
	\overbar{\mathcal{H}}_{n}':= |\partial^{\mathrm{e}}\mathcal{B}(n)|^{-1} \sum_{v\in \partial^{\mathrm{e}}\mathcal{B}(n)} \mathcal{H}_{v}'\ \ \text{and}\ \	\overbar{\mathcal{H}}_{n}'':= |\partial^{\mathrm{e}}\mathcal{B}(n)|^{-1} \sum_{v\in \partial^{\mathrm{e}}\mathcal{B}(n)} \mathcal{H}_{v}''. 
\end{equation}

On the event $\mathsf{F}$ (recall (\ref{def_F})), $\bm{0}$ is a good point. I.e., $\mathcal{H}_{\bm{0}}'\ge \cref{const_small}^2\lambda_* N_*^{-\frac{d}{2}+1} k_*^{-3}2^{k_*}$. By the strong Markov property, we know that $\mathcal{H}_{\bm{0}}'$ is bounded from above by 
\begin{equation}\label{add5.73}
	\widetilde{\mathbb{P}}_{\bm{0}}\Big[\tau_{\partial \mathcal{B}(0.01N_*)}< \tau_{\mathcal{C}^-_{\partial^{\mathrm{e}} \mathcal{B}(N_*/4)}} \Big]  \max_{x\in \partial \mathcal{B}(0.01N_*)} \sum_{y\in \partial^{\mathrm{e}} \mathcal{B}(\frac{3N_*}{16})} \widetilde{\mathbb{P}}_{x}\Big[\tau_{\partial^{\mathrm{e}} \mathcal{B}(\frac{3N_*}{16})}=\tau_{y} \Big] \mathcal{H}_{y}'. 
\end{equation}
Moreover, according to \cite[Lemma 6.3.7]{lawler2010random}, there exists $\Cl\label{const_lemma_6.3.7}(d)>0$ such that for any $x\in \partial \mathcal{B}(0.01N_*)$ and $y\in \partial^{\mathrm{e}} \mathcal{B}(\tfrac{3N_*}{16})$, 
\begin{equation}\label{add5.74}
	\widetilde{\mathbb{P}}_{x}\Big[\tau_{\partial^{\mathrm{e}} \mathcal{B}(\frac{3N_*}{16})}=\tau_{y} \Big] \le \Cref{const_lemma_6.3.7}\big|\partial^{\mathrm{e}} \mathcal{B}(\tfrac{3N_*}{16})\big|^{-1}. 
\end{equation}
By (\ref{require_constant1_section5}), (\ref{add5.73}), (\ref{add5.74}) and Lemma \ref{lemma_F}, we have 
\begin{equation}
	\begin{split}
			\mathcal{H}_{\bm{0}}'\overset{(\ref{add5.74})}{\le }& \Cref{const_lemma_6.3.7}  \widetilde{\mathbb{P}}_{\bm{0}}\Big[\tau_{\partial \mathcal{B}(0.01N_*)}< \tau_{\mathcal{C}^-_{\partial^{\mathrm{e}} \mathcal{B}(N_*/4)}} \Big]\overbar{\mathcal{H}}_{\frac{3N_*}{16}}'\\
			  \overset{(\text{Lemma}\ \ref{lemma_F})}{\le }&\Cref{const_lemma_6.3.7}e^{-k_*^{5.5}} \overbar{\mathcal{H}}_{\frac{3N_*}{16}}' \overset{(\ref{require_constant1_section5})}{\le }e^{-k_*^{5.4}} \overbar{\mathcal{H}}_{\frac{3N_*}{16}}'. 
	\end{split}
\end{equation}
Combined with $\mathcal{H}_{\bm{0}}'\ge \cref{const_small}^2\lambda_* N_*^{-\frac{d}{2}+1} k_*^{-3}2^{k_*}$, it yields the following inclusion:
\begin{equation}\label{inclusion_5.63}
	\mathsf{F} \subset \Big\{ \overbar{\mathcal{H}}_{\frac{3N_*}{16}}'\ge   e^{k_*^{5.3}}\lambda_* N_*^{-\frac{d}{2}+1}  \Big\}.
\end{equation}
The inclusion (\ref{inclusion_5.63}) implies that (recalling that $\overbar{\mathcal{H}}_*=\overbar{\mathcal{H}}_{\frac{3N_*}{16}}''$ in (\ref{4.23}))
\begin{equation}\label{5.45}
	\begin{split}
		&\mathbb{P}\Big(\overbar{\mathcal{H}}_{*}\ge e^{k_*^{5}}\lambda_*N_*^{-\frac{d}{2}+1}\Big)\\
			 \ge  &	\mathbb{P}\big(\mathsf{F}\big)- \mathbb{P}\Big(\overbar{\mathcal{H}}_{\frac{3N_*}{16}}'\ge   e^{k_*^{5.3}}\lambda_* N_*^{-\frac{d}{2}+1} ,\overbar{\mathcal{H}}_{\frac{3N_*}{16}}''\le  e^{k_*^{5}}\lambda_*N_*^{-\frac{d}{2}+1} \Big). 
	\end{split} 
\end{equation}

For any $1\le n\le \tfrac{N_*}{4}$, we define 
\begin{equation}
	\Phi_n':= |\partial^{\mathrm{e}}\mathcal{B}(n)|^{-1} \sum_{z\in \partial^{\mathrm{e}}\mathcal{B}(n)} \widetilde{\phi}_z\cdot \mathbbm{1}_{\big\{z\xleftrightarrow{\le 0}\partial^{\mathrm{e}} \mathcal{B}(N_*/2)\big\}^c }, 
\end{equation}
\begin{equation}
	\Phi_n'':= |\partial^{\mathrm{e}}\mathcal{B}(n)|^{-1} \sum_{z\in \partial^{\mathrm{e}}\mathcal{B}(n)} \widetilde{\phi}_z\cdot \mathbbm{1}_{\big\{z\xleftrightarrow{\le 0}\partial^{\mathrm{e}} \mathcal{B}(N_*/4)\big\}^c }. 
\end{equation}
For any $z\in \partial^{\mathrm{e}}\mathcal{B}(n)$, since $\big\{z\xleftrightarrow{\le 0}\partial^{\mathrm{e}} \mathcal{B}(N_*/4)\big\}\subset \{\widetilde{\phi}_z\le 0\}$, one has 
\begin{equation*}
	\begin{split}
			&\widetilde{\phi}_z\cdot \mathbbm{1}_{\big\{z\xleftrightarrow{\le 0}\partial^{\mathrm{e}} \mathcal{B}(N_*/2)\big\}^c }-\widetilde{\phi}_z\cdot  \mathbbm{1}_{\big\{z\xleftrightarrow{\le 0}\partial^{\mathrm{e}} \mathcal{B}(N_*/4)\big\}^c } \\
			=&\widetilde{\phi}_z\cdot\mathbbm{1}_{\big\{z\xleftrightarrow{\le 0}\partial^{\mathrm{e}} \mathcal{B}(N_*/2)\big\}^c\cap \big\{z\xleftrightarrow{\le 0}\partial^{\mathrm{e}} \mathcal{B}(N_*/4)\big\} } \le 0, 
	\end{split}
\end{equation*}
which implies that $\Phi_n''\ge \Phi_n'$. As a result, we have 
\begin{equation}\label{5.48}
	\begin{split}
		&	\mathbb{P}\Big(\overbar{\mathcal{H}}_{\frac{3N_*}{16}}'\ge   e^{k_*^{5.3}}\lambda_* N_*^{-\frac{d}{2}+1} ,\overbar{\mathcal{H}}_{\frac{3N_*}{16}}''\le  e^{k_*^{5}}\lambda_*N_*^{-\frac{d}{2}+1} \Big)\\
		\le &\mathbb{P}\Big(\overbar{\mathcal{H}}_{\frac{3N_*}{16}}'\ge   e^{k_*^{5.3}}\lambda_* N_*^{-\frac{d}{2}+1} ,\Phi_{\frac{3N_*}{16}}'\le  e^{k_*^{5.2}}\lambda_*N_*^{-\frac{d}{2}+1} \Big) \\
		&+\mathbb{P}\Big( \overbar{\mathcal{H}}_{\frac{3N_*}{16}}''\le  e^{k_*^{5}}\lambda_*N_*^{-\frac{d}{2}+1} ,\Phi_{\frac{3N_*}{16}}''\ge  e^{k_*^{5.2}}\lambda_*N_*^{-\frac{d}{2}+1}\Big). 
	\end{split}
\end{equation}
By Lemma \ref{lemma_strong_markov}, conditioning on $\mathcal{F}_{\mathcal{C}^{-}_{\partial^{\mathrm{e}} \mathcal{B}(N_*/2)}}$, $\{\widetilde{\phi}_z\}_{z\in \mathcal{B}(N_*/2)\setminus \mathcal{C}^{-}_{\partial^{\mathrm{e}} \mathcal{B}(N_*/2)}}$ has the same distribution as $\{\widetilde{\phi}_z'+\mathcal{H}'_z\}_{z\in \mathcal{B}(N_*/2)\setminus \mathcal{C}^{-}_{\partial^{\mathrm{e}} \mathcal{B}(N_*/2)}}$, where $\widetilde{\phi}'_\cdot\sim \mathbb{P}^{\mathcal{C}^{-}_{\partial^{\mathrm{e}} \mathcal{B}(N_*/2)}}$. By the symmetry of $\widetilde{\phi}'$ and (\ref{newadd_2.30}), we have (letting $\widehat{\Phi}_{\frac{3N_*}{16}}':=|\partial^{\mathrm{e}}\mathcal{B}(\tfrac{3N_*}{16})|^{-1} \sum_{z\in \partial^{\mathrm{e}}\mathcal{B}(\frac{3N_*}{16})} \widetilde{\phi}_z'$) 
\begin{equation}\label{5.68}
	\begin{split}
		&\mathbb{P}\Big(\overbar{\mathcal{H}}_{\frac{3N_*}{16}}'\ge   e^{k_*^{5.3}}\lambda_* N_*^{-\frac{d}{2}+1} ,\Phi_{\frac{3N_*}{16}}'\le  e^{k_*^{5.2}}\lambda_*N_*^{-\frac{d}{2}+1} \Big) \\
		\le &\mathbb{E}\Big[\mathbb{P}^{\mathcal{C}^{-}_{\partial^{\mathrm{e}} \mathcal{B}(N_*/2)}}\Big(\widehat{\Phi}_{\frac{3N_*}{16}}'\le (e^{k_*^{5.2}}-e^{k_*^{5.3}})\lambda_*N_*^{-\frac{d}{2}+1}  \Big)\Big]\\
		\overset{(\text{symmetry})}{=}&\mathbb{E}\Big[\mathbb{P}^{\mathcal{C}^{-}_{\partial^{\mathrm{e}} \mathcal{B}(N_*/2)}}\Big(\widehat{\Phi}_{\frac{3N_*}{16}}'\ge (e^{k_*^{5.3}}-e^{k_*^{5.2}})\lambda_*N_*^{-\frac{d}{2}+1}  \Big)\Big]
		\overset{(\ref{newadd_2.30})}{\le}e^{-\cref{const_average_2}\lambda_*^2e^{k_*^5}}. 
	\end{split}
\end{equation}
Similarly, conditioning on $\mathcal{F}_{\mathcal{C}^{-}_{\partial^{\mathrm{e}} \mathcal{B}(N_*/4)}}$, $\{\widetilde{\phi}_z\}_{z\in \mathcal{B}(N_*/4)\setminus \mathcal{C}^{-}_{\partial^{\mathrm{e}} \mathcal{B}(N_*/4)}}$ has the same distribution as $\{\widetilde{\phi}_z''+\mathcal{H}''_z\}_{z\in \mathcal{B}(N_*/4)\setminus \mathcal{C}^{-}_{\partial^{\mathrm{e}} \mathcal{B}(N_*/4)}}$,  where $\widetilde{\phi}''_\cdot\sim \mathbb{P}^{\mathcal{C}^{-}_{\partial^{\mathrm{e}} \mathcal{B}(N_*/4)}}$. Thus, by the arguments employed in proving (\ref{5.68}), we also have 
\begin{equation}\label{5.52}
	\begin{split}
		\mathbb{P}\Big( \overbar{\mathcal{H}}_{\frac{3N_*}{16}}''\le  e^{k_*^{5}}\lambda_*N_*^{-\frac{d}{2}+1} ,\Phi_{\frac{3N_*}{16}}''\ge  e^{k_*^{5.2}}\lambda_*N_*^{-\frac{d}{2}+1}\Big)
		\le e^{-\cref{const_average_2}\lambda_*^2e^{k_*^5}}.
	\end{split}
\end{equation}
Combining (\ref{5.48}), (\ref{5.68}) and (\ref{5.52}), we get 
\begin{equation}\label{5.53}
	\mathbb{P}\Big(\overbar{\mathcal{H}}_{\frac{3N_*}{16}}'\ge   e^{k_*^{5.3}}\lambda_* N_*^{-\frac{d}{2}+1} ,\overbar{\mathcal{H}}_{\frac{3N_*}{16}}''\le  e^{k_*^{5}}\lambda_*N_*^{-\frac{d}{2}+1} \Big) \le 2e^{-\cref{const_average_2}\lambda_*^2e^{k_*^5}}\overset{(\ref{equation_def_C9})}{\le } 2e^{-e^{k_*^5}}.    
\end{equation}

By (\ref{ineq_PF}), (\ref{5.45}) and (\ref{5.53}), we conclude Proposition \ref{prop_1}: 
\begin{equation}
	\mathbb{P}\Big(\overbar{\mathcal{H}}_{*}\ge e^{k_*^{5}}\lambda_*N_*^{-\frac{d}{2}+1}\Big) \ge 2^{-k_*^{30d}} -2e^{-e^{k_*^5}} \ge 2^{-k_*^{40d}}.    \pushQED{\qed} 
	\qedhere
	\popQED  
\end{equation}

    \section{Proof of Theorem \ref{thm2}}\label{section_thm2}

    As mentioned in Remark \ref{remark_thm2}, all upper bounds in Theorem \ref{thm2} are now confirmed since Theorem \ref{thm1} and Proposition \ref{lemma_bound_crossing} have been established. To prove the lower bounds in Theorem \ref{thm2}, in the low dimensional cases (i.e. $3\le d\le 6$), we use the exploration martingale (recall Section \ref{section_EM}). As for the high dimensional cases (i.e. $d\ge 7$), we employ the tree expansion argument presented in \cite[Section 3.4]{cai2023one}.

    \subsection{Proof for low dimensions} \label{section6.1}
    
    Note that the proof in this subsection holds for all $d\ge 3$. For any $N\ge 1$, we denote the normal random variable
    \begin{equation}\label{6.1}
    	\mathcal{Q}_N:= \sum\nolimits_{z\in \partial B(N)} \widetilde{\mathbb{P}}_{\bm{0}}\big(\tau_{\partial B(N)}=\tau_{z}\big)\widetilde{\phi}_z. 
    \end{equation}
    Since $\mathbb{E}(\widetilde{\phi}_{z_1}\widetilde{\phi}_{z_2})=\widetilde{G}(z_1,z_2)$, the variance of $\mathcal{Q}_N$ can be written as 
    \begin{equation}\label{6.2}
    	\sigma_N^2:= \sum\nolimits_{z_1,z_2\in \partial B(N)} \widetilde{\mathbb{P}}_{\bm{0}}\big(\tau_{\partial B(N)}=\tau_{z_1}\big)\widetilde{\mathbb{P}}_{\bm{0}}\big(\tau_{\partial B(N)}=\tau_{z_2}\big)\widetilde{G}(z_1,z_2). 
    \end{equation}
    Since $\sum_{z\in \partial B(N)}\widetilde{\mathbb{P}}_{\bm{0}}\big(\tau_{\partial B(N)}=\tau_{z}\big)=1$, there exists $\cl\label{const_section6}(d)>0$ such that 
    \begin{equation}\label{ineq_lemma_6.1}
    \sigma_N^2\ge \min\nolimits_{z_1,z_2\in \partial B(N)} \widetilde{G}(z_1,z_2) \ge  \cref{const_section6}N^{2-d}.
    \end{equation}

 We consider the exploration process $\mathcal{I}^{\partial B(N),\pm}_t$ and the corresponding martingale $\mathcal{M}_{\bm{0},t}^{\partial B(N),\pm}$ (recall Section \ref{section_EM}). For any $A\subset \mathbb{Z}^d$, we denote the sign cluster containing $A$ by 
 \begin{equation}\label{new6.7}
 	\mathcal{C}_A^{\pm}:=  \Big\{v\in \widetilde{\mathbb{Z}}^d: v\xleftrightarrow{\le 0} A\ \text{or}\  v\xleftrightarrow{\ge 0} A \Big\}. 
 \end{equation} 
 Note that $\mathcal{I}^{\partial B(N),\pm}_0= \partial B(N)$, $\mathcal{I}^{\partial B(N),\pm}_\infty= \mathcal{C}_{\partial B(N)}^{\pm}$, $\mathcal{M}_{\bm{0},0}^{\partial B(N),\pm}= \mathcal{Q}_N$ (by (\ref{6.1})) and 
 \begin{equation}\label{6.8}
 \big\{\mathcal{C}_{\partial B(N)}^{\pm}\cap B(n)\neq \emptyset\big\}=  \big\{B(n) \xleftrightarrow{\ge 0} \partial B(N)   \big\}\cup  \big\{B(n) \xleftrightarrow{\le 0} \partial B(N)   \big\}. 
 \end{equation}
On the event $\{\mathcal{C}_{\partial B(N)}^{\pm}\cap B(n)= \emptyset\}$, we know that the exploration process $\mathcal{I}^{\partial B(N),\pm}_t$ stops before intersecting $B(n)$ and thus, we have $\mathcal{M}_{\bm{0},\infty}^{\partial B(N),\pm}-\mathcal{M}_{\bm{0},0}^{\partial B(N),\pm}= -\mathcal{Q}_N$ (since $\mathcal{M}_{\bm{0},\infty}^{\partial B(N),\pm}=0$) and 
\begin{equation*}
	\begin{split}
		\langle\mathcal{M}_{\bm{0}}^{\partial B(N),\pm}\rangle_\infty \overset{(\ref{new_2.17})}{\le }  \max\nolimits_{z\in \partial B(n)} \widetilde{G}(\bm{0}, z) \overset{(\ref{bound_green})}{\le } \Cref{const_green_1}n^{2-d}. 
	\end{split}
\end{equation*}
Therefore, by Lemma \ref{lemma_EM} and the symmetry of the Brownian motion, we have 
\begin{equation}\label{6.10}
	\begin{split}
		&\mathbb{P}\big[\mathcal{C}_{\partial B(N)}^{\pm}\cap B(n)= \emptyset\mid  \mathcal{F}_{\partial B(N)}\big]\\
		\le &\mathbb{P}\big[\mathcal{M}_{\bm{0},\infty}^{\partial B(N),\pm}-\mathcal{M}_{\bm{0},0}^{\partial B(N),\pm}= -\mathcal{Q}_N,\langle\mathcal{M}_{\bm{0}}^{\partial B(N),\pm}\rangle_\infty\le \Cref{const_green_1}n^{2-d}\mid  \mathcal{F}_{\partial B(N)}\big]\\
		\le  &\mathbb{P}\big(|X|\ge\Cref{const_green_1}^{-\frac{1}{2}}n^{\frac{d}{2}-1} |\mathcal{Q}_N|\big), 
	\end{split}
\end{equation}
where $X\sim N(0,1)$ is independent of $\mathcal{Q}_N$. Recall that $\mathcal{Q}_N$ is a mean-zero normal random variable with variance $\sigma_N^2\ge \cref{const_section6}N^{2-d}$ (by (\ref{ineq_lemma_6.1})). Thus, by taking the integral on the both sides of (\ref{6.10}) (with respect to $\mathcal{F}_{\partial B(N)}$), we get 
\begin{equation}\label{new6.10}
	\begin{split}
		\mathbb{P}\big[\mathcal{C}_{\partial B(N)}^{\pm}\cap B(n)= \emptyset\big]
		\le \mathbb{P}\Big[|XY^{-1}|\ge\Cref{const_green_1}^{-\frac{1}{2}}\cref{const_section6}^{\frac{1}{2}}\big(nN^{-1}\big)^{\frac{d}{2}-1}\Big],  
	\end{split}
\end{equation}
where $Y\sim N(0,1)$ is independent of $X$. Since $Z:=XY^{-1}$ has the the Cauchy distribution with density function $\pi^{-1}(1+t^2)^{-1}$, we derive from (\ref{new6.10}) that 
\begin{equation}\label{6.11}
	\begin{split}
		\mathbb{P}\big[\mathcal{C}_{\partial B(N)}^{\pm}\cap B(n)\neq  \emptyset\big] \ge &	\mathbb{P}\Big[|Z|\le \Cref{const_green_1}^{-\frac{1}{2}}\cref{const_section6}^{\frac{1}{2}}\big(nN^{-1}\big)^{\frac{d}{2}-1}\Big]\\
		\ge &\frac{2}{\pi}\int_{0\le t\le\Cref{const_green_1}^{-\frac{1}{2}}\cref{const_section6}^{\frac{1}{2} }(nN^{-1})^{\frac{d}{2}-1}} (1+t^{2})^{-1}dt\\
		\ge & c(nN^{-1})^{\frac{d}{2}-1}. 
	\end{split}
\end{equation}
By the symmetry of $\widetilde{\phi}$, (\ref{6.8}) and (\ref{6.11}), we conclude the lower bounds in Theorem \ref{thm2} for $3\le d\le 6$:
\begin{equation*}
	\begin{split}
		\mathbb{P}\Big[B(n) \xleftrightarrow{\ge 0} \partial B(N)  \Big]\overset{(\text{symmetry})}{\ge } & \tfrac{1}{2}\mathbb{P}\Big[B(n) \xleftrightarrow{\ge 0} \partial B(N) \ \text{or}\ B(n) \xleftrightarrow{\le 0} \partial B(N) \Big]\\
		\overset{(\ref{6.8})}{= }& \tfrac{1}{2}\mathbb{P}\Big[\mathcal{C}_{\partial B(N)}^{\pm}\cap B(n)\neq \emptyset\Big]\overset{(\ref{6.11})}{\ge }c\Big(\frac{n}{N}\Big)^{\frac{d}{2}-1}.    \pushQED{\qed} 
		\qedhere
		\popQED
	\end{split} 
\end{equation*}

    \subsection{Proof for high dimensions}


  When $N\ge n\ge 1$ and $n/N\ge 0.1$, according to Section \ref{section6.1}, there exists $\cl\label{const_section_6.2}(d)\in (0,1)$ such that 
  \begin{equation*}
  	\mathbb{P}\big[B(n)\xleftrightarrow{\ge 0}  \partial B(N)\big] \ge \cref{const_section_6.2}.
  \end{equation*}
  As a result, in this case, in order to have (\ref{thm2_1.9}) it suffices to take $\cref{const_thm2_6}=\cref{const_section_6.2}$.
  

  Now we assume that $N\ge n\ge 1$ and $n/N\le 0.1$. Let $\mathbf{X}:= \sum_{x\in \partial B(n)}\mathbbm{1}_{x\xleftrightarrow{\ge 0}\partial B(N)}$. By the Paley–Zygmund inequality, we have 
    \begin{equation}\label{6.12}
    	\mathbb{P}\big[B(n)\xleftrightarrow{\ge 0} \partial B(N)\big]=\mathbb{P}\big(\mathbf{X}>0\big)\ge \frac{(\mathbb{E}\mathbf{X})^2}{\mathbb{E}(\mathbf{X}^2)}. 
    \end{equation}
    By (\ref{ineq_high_d}) and $|\partial B(n)|\asymp n^{d-1}$, we have $\mathbb{E}(\mathbf{X})\ge cn^{d-1}N^{-2}$. Thus, to obtain the lower bounds in Theorem \ref{thm2} for $d\ge 7$, it remains to upper-bound 
    \begin{equation}\label{add6.10}
    	\mathbb{E}(\mathbf{X}^2) = \sum\nolimits_{x,y\in \partial B(n)} \mathbb{P}\big[x\xleftrightarrow{\ge 0} \partial B(N),y\xleftrightarrow{\ge 0} \partial B(N)\big].
    \end{equation}

   In this subsection, we abbreviate ``$\xleftrightarrow{\cup\widetilde{\mathcal{L}}_{1/2}}$'' as ``$\xleftrightarrow{}$''. By Lemma \ref{lemma_iso} we have 
    \begin{equation}
    	\begin{split}
    	 \mathbb{P}\big[x\xleftrightarrow{\ge 0} \partial B(N),y\xleftrightarrow{\ge 0} \partial B(N)\big]
    	  \le \mathbb{P}\big[x\xleftrightarrow{} \partial B(N),y\xleftrightarrow{} \partial B(N)\big]. 
    	\end{split}
    \end{equation}
     Recall the definitions of glued loops and the BKR inequality (see Lemma \ref{lemma_BKR}) in Section \ref{section_loop_soup}. Applying the tree expansion for loop soups (see \cite[Lemma 3.5]{cai2023one}), we know that on the event $\big\{ x\xleftrightarrow{} \partial B(N),y\xleftrightarrow{} \partial B(N)\big\}$, there exists a glued loop $\gamma_*$ and three loop clusters $\{\mathcal{C}_i\}_{i=1}^{3}$ composed of different collections of glued loops such that $x\in \mathcal{C}_1$, $y\in \mathcal{C}_2$, $\mathcal{C}_3\cap \partial B(n)\neq \emptyset$ and $\gamma_*\cap \mathcal{C}_i\neq \emptyset$ for all $i\in \{1,2,3\}$. In what follows, we estimate the probability of $\big\{ x\xleftrightarrow{} \partial B(N),y\xleftrightarrow{} \partial B(N)\big\}$ with different restrictions on $\gamma_*$ and $\{\mathcal{C}_i\}_{i=1}^{3}$.
    \begin{enumerate}
    	\item  When $\mathcal{C}_1$ or $\mathcal{C}_2$ intersects $\partial B(N/3)$: Without loss of generality, we assume $\mathcal{C}_1\cap \partial B(N/3)\neq \emptyset$, which implies $x \xleftrightarrow{\mathcal{C}_1}\partial B(N/3)$. Combined with the fact that $\mathcal{C}_2\cup\mathcal{C}_3\cup \gamma_*$ connects $y$ and $\partial B(N)$ and contains a different collection of glued loops from $\mathcal{C}_1$, it yields that $\{x\xleftrightarrow{}\partial B(N/3)\}\circ \{y\xleftrightarrow{} \partial B(N)\}$ happens, whose probability is at most (using the BKR inequality)
    	\begin{equation}
    		\begin{split}
    			\mathbb{P}\big[x\xleftrightarrow{} \partial B(N/3)\big]\mathbb{P}\big[y\xleftrightarrow{} \partial B(N/3)\big]
    		\overset{(n/N\le 0.1)}{\le } \theta_d^2(N/5) \overset{(\ref{ineq_high_d})}{\le } CN^{-4}. 
    		\end{split}
    	\end{equation}

    	\item  When $\mathcal{C}_1\cup \mathcal{C}_2\subset \widetilde{B}(N/3)$ and $\mathcal{C}_3\cap \gamma_*\subset \widetilde{B}(2N/3)$: For each $i\in \{1,2,3\}$, there exists $z_i\in \mathcal{C}_i\cap \mathbb{Z}^d$ such that $\mathrm{dist}(\{z_i\},\mathcal{C}_i\cap \gamma_*)\le d$ (since $\mathcal{C}_i$ is continuous and the length of every interval $I_e$ is $d$). Then we have $z_1\xleftrightarrow{\mathcal{C}_1} x$, $z_2\xleftrightarrow{\mathcal{C}_2} y$, $z_3\xleftrightarrow{\mathcal{C}_3} \partial B(N)$ and $\gamma_*$ intersects $\widetilde{B}_{z_i}(1)$ for all $i\in \{1,2,3\}$ (these four events happen disjointly), whose probability can be upper-bounded by (using the BKR inequality and the fact that the loop measure of loops intersecting $\widetilde{B}_{z_i}(1)$ for $i\in \{1,2,3\}$ is at most $C|z_1-z_2|^{2-d} |z_2-z_3|^{2-d} |z_3-z_1|^{2-d}$; see \cite[Equation (\ref{4.14})]{cai2023one})
    	\begin{equation}
    		\begin{split}
    			&C\sum_{z_1,z_2\in B(N/3),z_3\in B(\frac{2N}{3}+1)}|z_1-z_2|^{2-d} |z_2-z_3|^{2-d} |z_3-z_1|^{2-d}\\
    			&\ \ \ \ \ \ \ \ \ \ \ \ \ \ \ \ \ \ \ \ \ \ \ \ \ \ \  \ \ \ \cdot |z_1-x|^{2-d} |z_2-y|^{2-d} \mathbb{P}\big[z_3\xleftrightarrow{} \partial B(N)\big]\\
    			\overset{(\ref{ineq_high_d})}{\le} & C'N^{-2}\sum_{z_1,z_2\in B(N/3),z_3\in B(\frac{2N}{3}+1)}|z_1-z_2|^{2-d} |z_2-z_3|^{2-d} |z_3-z_1|^{2-d}\\
    			&\ \ \ \ \ \ \ \ \ \ \ \ \ \ \ \ \ \ \ \ \ \ \ \ \ \ \ \ \ \ \ \ \  \ \ \ \cdot |z_1-x|^{2-d} |z_2-y|^{2-d}\\
    			\le &C''N^{-2}|x-y|^{4-d}
    		\end{split}
    	\end{equation}
    	 where we used \cite[Equation (4.12)]{cai2023one} in the last inequality.

    	\item  When $\mathcal{C}_1\cup \mathcal{C}_2\subset \widetilde{B}(N/3)$ and $\mathcal{C}_3\cap \gamma_*\not\subset \widetilde{B}(2N/3)$: We keep the notations $z_1,z_2$ as in Case (2). If $\gamma_*\subset \widetilde{B}(N)$, we define $z_3$ as in Case (2). In this case, $\mathcal{C}_3\cap \gamma_*\not\subset \widetilde{B}(2N/3)$ implies that $z_3\in B(N)\setminus B(\frac{2N}{3}-1)$. Otherwise (i.e. $\gamma_*\not\subset \widetilde{B}(N)$), since $\gamma_*$ is continuous and intersects $\mathcal{C}_1\subset \widetilde{B}(N/3)$, there exists $z_3\in \gamma_*\cap \partial B(N)$($\subset B(N)\setminus B(\frac{2N}{3}-1)$). Similar to Case (2), we have $z_1\xleftrightarrow{\mathcal{C}_1} x$, $z_2\xleftrightarrow{\mathcal{C}_2} y$, $z_3\xleftrightarrow{\mathcal{C}_3} \partial B(N)$ and $\gamma_*$ intersects $\widetilde{B}_{z_i}(1)$ for all $i\in \{1,2,3\}$, whose probability can be bounded from above by 
    	\begin{equation}\label{new_add_6.17}
    		\begin{split}
    			&\mathbb{I}:=C\sum_{z_1,z_2\in B(N/3),z_3\in B(N)\setminus B(\frac{2N}{3}-1)}|z_1-z_2|^{2-d} |z_2-z_3|^{2-d} |z_3-z_1|^{2-d}\\
    			&\ \ \ \ \ \ \ \ \ \ \ \ \ \ \ \ \ \ \ \ \ \ \ \ \ \  \ \ \ \ \ \  \  \ \ \ \ \ \ \ \ \ \cdot |z_1-x|^{2-d} |z_2-y|^{2-d}\mathbb{P}\big[z_3\xleftrightarrow{} \partial B(N)\big].
    		\end{split}
    	\end{equation}
    	For the sum in (\ref{new_add_6.17}), restricted to $z_3\in \partial B(k)$ (where $\frac{2N}{3}-1\le k\le N$), we have that $\min\{|z_2-z_3|,|z_3-z_1|\}\ge cN$ (since $z_1,z_2\in B(N/3)$) and that $\mathbb{P}\big[z_3\xleftrightarrow{} \partial B(N)\big]\le C(N-k)^{-2}$ (by (\ref{ineq_high_d})). As a result, for $\frac{2N}{3}-1\le k\le N$, 
    	\begin{equation}\label{new_6.18}
    		\begin{split}
    			&\sum_{z_1,z_2\in B(N/3),z_3\in \partial B(k)}|z_1-z_2|^{2-d} |z_2-z_3|^{2-d} |z_3-z_1|^{2-d}\\
    			&\ \ \ \ \ \ \ \ \ \ \ \ \ \ \ \ \ \ \ \ \ \ \ \ \  \cdot |z_1-x|^{2-d} |z_2-y|^{2-d}\mathbb{P}\big[z_3\xleftrightarrow{} \partial B(N)\big]\\
    			\overset{(|\partial B(k)|\asymp N^{d-1})}{\le } &CN^{3-d}(N-k)^{-2} \sum_{z_1,z_2\in B(N/3)}|z_1-z_2|^{2-d}|z_1-x|^{2-d} |z_2-y|^{2-d}\\
    			\le &C'N^{3-d}(N-k)^{-2}\sum_{z_2\in B(N/3)}|z_2-x|^{4-d} |z_2-y|^{2-d}\\
    			\le &C''N^{3-d}(N-k)^{-2}|x-y|^{6-d}, 
    		\end{split}
    	\end{equation}
    	where we used \cite[Lemma 4.3 and Equation (4.10)]{cai2023one} respectively in the last two inequalities. Combining (\ref{new_6.18}) and $\sum_{\frac{2N}{3}-1\le k\le N}(N-k)^{-2}\le 3$, we have (recalling $\mathbb{I}$ in (\ref{new_add_6.17}))
    	\begin{equation}
    		\mathbb{I}\le CN^{3-d}|x-y|^{6-d}. 
    	\end{equation}
  
    \end{enumerate}
    In conclusion, we obtain that 
    \begin{equation}\label{new_add_6.21}
    	\begin{split}
    	&\mathbb{P}\big[x\xleftrightarrow{} \partial B(N),y\xleftrightarrow{} \partial B(N)\big]\\
    	\le &C\big(N^{-4}+ N^{-2}|x-y|^{4-d}+N^{3-d}|x-y|^{6-d}\big). 
    	\end{split}  
    \end{equation}
    Moreover, \cite[Equation (4.4)]{cai2023one} shows that for any $a< d-1$, 
    \begin{equation*}
    	\max\nolimits_{x\in \mathbb{Z}^d}\sum\nolimits_{y\in \partial B(n)} |x-y|^{-a}\le Cn^{d-1-a}.
    \end{equation*}
    Combined with $|\partial B(n)|\asymp n^{d-1}$, it implies 
    \begin{equation}\label{6.22}
    	\sum\nolimits_{x,y\in \partial B(n)}|x-y|^{4-d}\le Cn^{d-1}\max\nolimits_{x\in \mathbb{Z}^d}\sum\nolimits_{y\in \partial B(n)} |x-y|^{4-d}\le C'n^{d+2},
    \end{equation}
    \begin{equation}\label{6.23}
    	\sum\nolimits_{x,y\in \partial B(n)}|x-y|^{6-d}\le Cn^{d-1}\max\nolimits_{x\in \mathbb{Z}^d}\sum\nolimits_{y\in \partial B(n)} |x-y|^{6-d}\le C'n^{d+4}.
    \end{equation}
    Combining (\ref{new_add_6.21}), (\ref{6.22}) and (\ref{6.23}), we have 
    \begin{equation}\label{6.24}
    	\begin{split}
    		\mathbb{E}(\mathbf{X}^2)\overset{(\ref{add6.10})}{=} &  \sum\nolimits_{x,y\in \partial B(n)} \mathbb{P}\big[x\xleftrightarrow{\ge 0} \partial B(N),y\xleftrightarrow{\ge 0} \partial B(N)\big]\\
    		\le& C\big(n^{2d-2}N^{-4}+n^{d+2}N^{-2}+n^{d+4}N^{3-d}\big)\\
    		\le&  C'\big(n^{2d-2}N^{-4}+n^{d+2}N^{-2}\big),
    	\end{split}
    \end{equation}
    where in the last inequality we used  $\frac{n^{d+2}N^{-2}}{n^{d+4}N^{3-d}}= N^{d-5}n^{-2}\ge 1$ for $d\ge 7$ .

    By (\ref{6.12}), $\mathbb{E}(\mathbf{X})\ge cn^{d-1}N^{-2}$ and (\ref{6.24}), we conclude the lower bounds for $d\ge 7$ in Theorem \ref{thm2}:
    \begin{equation*}
    	\begin{split}
    		\mathbb{P}\big[B(n)\xleftrightarrow{\ge 0} \partial B(N)\big]\ge \frac{c(n^{d-1}N^{-2})^2}{n^{2d-2}N^{-4}+n^{d+2}N^{-2}}= c\big(1+ n^{4-d}N^2\big)^{-1}. \pushQED{\qed} 
    		\qedhere
    		\popQED
    	\end{split}
    \end{equation*}

\section*{Acknowledgments}

We warmly thank Tom Hutchcroft and Gady Kozma for helpful discussions. J. Ding is partially supported by NSFC Key
Program Project No. 12231002 and the Xplorer prize.

	\bibliographystyle{plain}
	\bibliography{ref}

\begin{thebibliography}{10}

\bibitem{aizenman1987uniqueness}
M.~Aizenman, H.~Kesten, and C.~M. Newman.
\newblock Uniqueness of the infinite cluster and continuity of connectivity
  functions for short and long range percolation.
\newblock {\em Communications in Mathematical Physics}, 111(4):505--531, 1987.

\bibitem{barsky1991percolation}
D.~J. Barsky and M.~Aizenman.
\newblock Percolation critical exponents under the triangle condition.
\newblock {\em The Annals of Probability}, pages 1520--1536, 1991.

\bibitem{bauerschmidt2015critical}
R.~Bauerschmidt, D.~C. Brydges, and G.~Slade.
\newblock Critical two-point function of the 4-dimensional weakly self-avoiding
  walk.
\newblock {\em Communications in Mathematical Physics}, 338:169--193, 2015.

\bibitem{bricmont1987percolation}
J.~Bricmont, J.~L. Lebowitz, and C.~Maes.
\newblock Percolation in strongly correlated systems: the massless gaussian
  field.
\newblock {\em Journal of statistical physics}, 48:1249--1268, 1987.

\bibitem{cai2023one}
Z.~Cai and J.~Ding.
\newblock One-arm exponent of critical level-set for metric graph gaussian free
  field in high dimensions.
\newblock {\em arXiv preprint arXiv:2307.04434}, 2023.

\bibitem{vcerny2023critical}
J.~{\v{C}}ern{\`y} and R.~Locher.
\newblock Critical and near-critical level-set percolation of the gaussian free
  field on regular trees.
\newblock {\em arXiv preprint arXiv:2302.02753}, 2023.

\bibitem{chang2023percolation}
Y.~Chang, H.~Du, and X.~Li.
\newblock Percolation threshold for brownian loop soup on metric graphs.
\newblock {\em arXiv preprint arXiv:2304.08225}, 2023.

\bibitem{chiarini2016extremes}
A.~Cipriani, A.~Chiarini, and R.~Hazra. 
\newblock Extremes of the supercritical gaussian free field.
\newblock {\em Alea}, 13(2):711--724, 2016.

\bibitem{ding2020percolation}
J.~Ding and M.~Wirth.
\newblock Percolation for level-sets of gaussian free fields on metric graphs.
\newblock {\em The Annals of Probability}, 48(3):1411--1435, 2020.

\bibitem{drewitz2018sign}
A.~Drewitz, A.~Pr{\'e}vost, and P.-F. Rodriguez.
\newblock The sign clusters of the massless gaussian free field percolate on $\mathbb{Z}^d$, $d\ge 3$ (and more).
\newblock {\em Communications in Mathematical Physics}, 362:513--546, 2018.

\bibitem{drewitz2023arm}
A.~Drewitz, A.~Pr{\'e}vost, and P.-F. Rodriguez.
\newblock Arm exponent for the gaussian free field on metric graphs in
  intermediate dimensions.
\newblock {\em arXiv preprint arXiv:2312.10030}, 2023.

\bibitem{drewitz2023critical}
A.~Drewitz, A.~Pr{\'e}vost, and P.-F. Rodriguez.
\newblock Critical exponents for a percolation model on transient graphs.
\newblock {\em Inventiones mathematicae}, 232(1):229--299, 2023.

\bibitem{drewitz2024critical}
A.~Drewitz, A.~Pr{\'e}vost, and P.-F. Rodriguez.
\newblock Critical one-arm probability for the metric gaussian free field in
  low dimensions.
\newblock {\em arXiv preprint arXiv:2405.17417}, 2024.

\bibitem{drewitz2014chemical}
A.~Drewitz, B.~R{\'a}th, and A.~Sapozhnikov.
\newblock On chemical distances and shape theorems in percolation models with
  long-range correlations.
\newblock {\em Journal of Mathematical Physics}, 55(8), 2014.

\bibitem{drewitz2015high}
A.~Drewitz and P.-F. Rodriguez.
\newblock High-dimensional asymptotics for percolation of gaussian free field
  level sets.
\newblock {\em Electron. J. Probab}, 20(47):1--39, 2015.

\bibitem{duminil2020existence}
H.~Duminil-Copin, S.~Goswami, A.~Raoufi, F.~Severo, and A.~Yadin.
\newblock Existence of phase transition for percolation using the gaussian free
  field.
\newblock {\em Duke Mathematical Journal}, 169(18):3539--3563, 2020.

\bibitem{duminil2023equality}
H.~Duminil-Copin, S.~Goswami, P.-F. Rodriguez, and F.~Severo.
\newblock Equality of critical parameters for percolation of gaussian free
  field level sets.
\newblock {\em Duke Mathematical Journal}, 172(5):839--913, 2023.

\bibitem{fitzner2017mean}
R.~Fitzner and R.~van~der Hofstad.
\newblock Mean-field behavior for nearest-neighbor percolation in $d>10$.
\newblock {\em Electronic Journal of Probability}, 22:43, 2017.

\bibitem{ganguly2024ant}
S.~Ganguly and K.~Nam.
\newblock The ant on loops: Alexander-orbach conjecture for the critical level
  set of the gaussian free field.
\newblock {\em arXiv preprint arXiv:2403.02318}, 2024.

\bibitem{goswami2022radius}
S.~Goswami, P.-F. Rodriguez, and F.~Severo.
\newblock On the radius of gaussian free field excursion clusters.
\newblock {\em The Annals of Probability}, 50(5):1675--1724, 2022.

\bibitem{hara1990mean}
T.~Hara and G.~Slade.
\newblock Mean-field critical behaviour for percolation in high dimensions.
\newblock {\em Communications in Mathematical Physics}, 128(2):333--391, 1990.

\bibitem{hutchcroft2022critical}
T.~Hutchcroft.
\newblock Critical cluster volumes in hierarchical percolation.
\newblock {\em arXiv preprint arXiv:2211.05686}, 2022.

\bibitem{kozma2011arm}
G.~Kozma and A.~Nachmias.
\newblock Arm exponents in high dimensional percolation.
\newblock {\em Journal of the American Mathematical Society}, 24(2):375--409,
  2011.

\bibitem{lawler2013intersections}
G.~Lawler.
\newblock {\em Intersections of random walks}.
\newblock Springer Science \& Business Media, 2013.

\bibitem{lawler2010random}
G.~Lawler and V.~Limic.
\newblock {\em Random walk: a modern introduction}, volume 123.
\newblock Cambridge University Press, 2010.

\bibitem{lupu2016loop}
T.~Lupu.
\newblock From loop clusters and random interlacements to the free field.
\newblock {\em Annals of Probability}, 44(3):2117--2146, 2016.

\bibitem{lupu2022equivalence}
T.~Lupu.
\newblock An equivalence between gauge-twisted and topologically conditioned
  scalar gaussian free fields.
\newblock {\em arXiv preprint arXiv:2209.07901}, 2022.

\bibitem{lupu2018random}
T.~Lupu and W.~Werner.
\newblock The random pseudo-metric on a graph defined via the zero-set of the
  gaussian free field on its metric graph.
\newblock {\em Probability Theory and Related Fields}, 171:775--818, 2018.

\bibitem{morters2010brownian}
P.~M{\"o}rters and Y.~Peres.
\newblock {\em Brownian motion}, volume~30.
\newblock Cambridge University Press, 2010.

\bibitem{muirhead2024percolation}
S.~Muirhead.
\newblock Percolation of strongly correlated gaussian fields ii. sharpness of
  the phase transition.
\newblock {\em The Annals of Probability}, 52(3):838--881, 2024.

\bibitem{Franco2024percolation}
S.~Muirhead and F.~Severo.
\newblock Percolation of strongly correlated gaussian fields, i: Decay of
  subcritical connection probabilities.
\newblock {\em Probability and Mathematical Physics}, 5(2):357--412, 2024.

\bibitem{popov2015decoupling}
S.~Popov and B.~R{\'a}th.
\newblock On decoupling inequalities and percolation of excursion sets of the
  gaussian free field.
\newblock {\em Journal of Statistical Physics}, 159(2):312--320, 2015.

\bibitem{popov2015soft}
S.~Popov and A.~Teixeira.
\newblock Soft local times and decoupling of random interlacements.
\newblock {\em Journal of the European Mathematical Society},
  17(10):2545--2593, 2015.

\bibitem{rodriguez2013phase}
P.-F. Rodriguez and A.-S. Sznitman.
\newblock Phase transition and level-set percolation for the gaussian free
  field.
\newblock {\em Communications in Mathematical Physics}, 320:571--601, 2013.

\bibitem{vershynin2020high}
R.~Vershynin.
\newblock {\em High-dimensional probability: An introduction with applications
  in data science}, volume~47.
\newblock Cambridge university press, 2018.

\bibitem{werner2021clusters}
W.~Werner.
\newblock On clusters of brownian loops in $d$ dimensions.
\newblock {\em In and Out of Equilibrium 3: Celebrating Vladas Sidoravicius},
  pages 797--817, 2021.

\end{thebibliography}
	
\end{document}